\pdfoutput=1
\documentclass[a4paper, 12pt]{article}

\usepackage[final]{showkeys} % show references
\usepackage{standalone}
\usepackage{header}
\usepackage{abbrev}
\usepackage{abbrev_qi}

\usepackage{tikz}
\usepackage{enumitem}

\newtheorem{theo}[thm]{Theorem}
\newtheorem{rem}[thm]{Remark}
\usepackage{caption}

\usepackage{color}
\definecolor{applegreen}{rgb}{0.55, 0.71, 0.0}
\definecolor{yellowcut}{RGB}{0, 200, 0}
\definecolor{orangecross}{RGB}{248, 48, 8}
\definecolor{violetcross}{RGB}{178, 8, 248}
\definecolor{ao(english)}{rgb}{0.0, 0.5, 0.0}
\definecolor{green(html/cssgreen)}{rgb}{0.0, 0.5, 0.0}
\definecolor{lincolngreen}{rgb}{0.11, 0.35, 0.02}

\newcommand\Pf{\mathbf{Pf}}
\newcommand\Corr[2]{\langle#2 \rangle_{#1}}

\def\Re{\mathop{\mathrm{Re}}\nolimits}

\def\supp{\mathop{\mathrm{supp}}\nolimits}

\def\Xsym{X^\mathrm{sym}_{[\mathbf v, \mathbf u]}}
\def\Xanti{X^\mathrm{anti}_{[\mathbf v, \mathbf u]}}

\newcommand\tempskipped[1]{}

\definecolor{darkblue}{RGB}{0, 0, 139}

\title{Conformal Invariance in the Quantum Ising Model}
\author{
  Jhih-Huang~Li \footnote{National Taiwan University. \Email{lijhih@ntu.edu.tw}}, 
  Rémy~Mahfouf \footnote{École Normale Supérieure. \Email{remy.mahfouf@ens.fr}}}
\date{\today}

\begin{document}

% \pdfbookmark[0]{Convergence of the Quantum Ising Model}{title}
\maketitle

\vspace{1cm}

\begin{abstract}
  We introduce Kadanoff--Ceva order-disorder operators in the quantum Ising model.
  This approach was first used for the classical planar Ising model~\cite{KC-order-disorder} and recently put back to the stage~\cite{Chelkak-Cimasoni-Kassel, Chelkak-state-art}.
  This representation turns out to be equivalent to the loop expansion of Sminorv's fermionic observables~\cite{Smirnov-conformal} and is particularly interesting due to its simple and compact formulation.
  Using this approach, we are able to extend different results known in the classical planar Ising model, such as the conformal invariance / covariance of correlations and the energy-density, to the spin-representation of the (1+1)-dimensional quantum Ising model.
\end{abstract}

\vspace{1cm}

\newpage
% \pdfbookmark[0]{\contentsname}{}
\tableofcontents

\newpage

\section{Introduction}

The one-dimensional quantum Ising model is a quantum spin chain defined on $\bbZ$.
Given two positive parameters $\tau, \theta > 0$, the interaction of the system is described by the following Hamiltonian,
\[
  \bfH = - \theta \sum_{x \sim y} \Pauli{x}{3} \Pauli{y}{3} 
  - \tau \sum_{x \in V} \Pauli{x}{1},
\]
acting on the Hilbert space $\bigotimes_\bbZ \bbC^2$.
In the above definition, $\theta$ is the intensity of the interaction between neighboring particles and $\tau$ is the intensity of the interaction with the transverse field.
We recall that the $\tfrac12$-spin Pauli matrices $\Pauli{}{1}$ and $\Pauli{}{3}$ are given by
\begin{equation*}
  \Pauli{}{1} = \left(
    \begin{array}{cc}
      0 & 1 \\ 1 & 0
    \end{array}
  \right),\quad
  \Pauli{}{3} = \left(
    \begin{array}{cc}
      1 & 0 \\ 0 & -1
    \end{array}
  \right),
\end{equation*}
which act on the state space of a spin $\bbC^2 \cong {\textrm{Span}} ( \ket{+}, \ket{-} )$, where we identify $\ket{+}$ with $(1, 0)$ and $\ket{-}$ with $(0, 1)$, for instance.
The operators $\Pauli{x}{1}$ is defined by the tensor product which takes the Pauli matrix $\Pauli{}{1}$ at coordinate $x$ and identity operator elsewhere; the same applies to $\Pauli{x}{3}$.
As consequence, the operator $\bfH$ makes sense and acts effectively on $\bigotimes_\bbZ \bbC^2$.

\medskip

The one-dimensional quantum Ising model is an exactly solvable one-dimensional quantum model~\cite{Pfeuty-QI}.
The quantization $e^{-\beta H}$ of the Gibbs measure of the classical Ising model will be the operator of our interest and the ground state, or the limit of the operator when $\beta \to \infty$, is the quantum state, or the probability measure we are interested in.

This ground state has several graphical representations, such as the usual spin-representation, the FK-representation or the random-current representation.
Readers may have a look at~\cite{Ioffe-QI} for a nice and complete exposition on this topic.
These representations are useful in interpreting results from the classical Ising model~\cite{GOS-entanglement, BG-QI-sharp, Bjornberg-infrared, Li-QI-SLE}, and in particular, it has been shown that the ground state of the one-dimensional quantum Ising model goes through a continuous phase transition at $\rho = \rho_c := 2$~\cite{BG-QI-sharp, DCLM-universality}.

% Note that only the ratio $\rho = \tau / \theta$ matters when it comes to the phase transition since the Hamiltonian $H$ is of degree 1 in $\tau$ and $\theta$.

\medskip

In this paper, we will tackle a finer study of the behavior of the model at the critical point $\rho_c = 2$ on one side, and the magnetization in the spin subcritical regime on the other side.
The main tool is based on an approach of fermionic observables using Kadanoff--Ceva correlators, which, to our knowledge, has not been done yet for the quantum Ising model.
This approach allows us to derive, in a more systematic way, local relations of the observables, leading to results for the space-time representation of the quantum model.
Similar local relations were obtained previously FK-loop representation~\cite{Bjornberg-obs, Li-QI-SLE}.

This formalism allows us to derive the horizontal full-plane spin-spin correlations at any temperature, and in any directions at criticality.
We also prove the convergence of the energy density (Theorem~\ref{thm:single_energy} and Theorem~\ref{thm:multiple_energy}) and $n$-spin correlations (Theorem~\ref{thm:spins}) in simply-connected domains.
The last mentioned results have a particularly nice physical interpretation when considering the space-time representation of the one-dimensional quantum Ising model.

Moreover, this approach also provides us a way to compute the magnetization in the spin sub-critical regime, see Theorem~\ref{thm:magnetization_below_criticality}.

\medskip

The paper is organized as follows.
In Section~\ref{sec:model} we recall the definition of the model and its multiple equivalent representations.
In Section~\ref{sec:disorder_insertion}, we introduce the Kadanoff--Ceva formalism in the semi-discrete lattice and use it to define mixed-correlators.
In Sections~\ref{sec:bvp} and~\ref{sec:derivation_main_theorems}, we recall the main tools to study boundary value problems (BVP) that arise naturally from semi-discrete fermionic observables, and prove the convergence to their continuous counterparts.

Those tools were originally used to prove conformal invariance of the interface on the square lattice~\cite{Smirnov-towards, Smirnov-conformal, DCS-CI}, which were generalized later to isoradial graphs~\cite{CS-universality}.
Recently, the first author used them in a similar way in the quantum model~\cite[Section~4]{Li-QI-SLE}.
The convergence of BVP allows then to derive the existence and the conformal invariance of the scaling limit.
Finally, in Section~\ref{sec:polynomials}, we follow the formalism of~\cite{CHM-zig-zag-Ising} to rederive directly in the semi-discrete lattice full-plane asymptotics of correlations.

\paragraph{Acknowledgments} 
J.-H. L. acknowledges support from EPRSC through grant EP/R024456/1. R.M. acknowledges the support from the ANR-18CE40-0033 project DIMERS.
This research was initiated during the stay of the second author at the University of Warwick.
J.-H. L. is grateful to Nikolaos Zygouras for his support at University of Warwick and to Hugo Duminil-Copin and Stanislav Smirnov for introducing him to the quantum Ising model.
R.M. is grateful to the University of Warwick for its hospitality, Dmitry Chelkak, Sung-Chul Park and Konstantin Izyurov for fruitful discussions.

\section{The quantum Ising model}\label{sec:model}

\subsection{Semi-discrete graph notations}

% Given a domain $\Omega$ and $\delta > 0$.
% Denote $\primaldomain$ (resp. $\dualdomain$) its semi-discretized primal (resp. dual) domain.

Here we recall some definitions from~\cite[Section~3.1]{Li-QI-SLE}, 
given a simply connected domain $\Omega \subseteq \bbC$ and $\delta > 0$.
We write $\primaldomain$ for the semi-discretized \emph{primal domain} of $\Omega$ and $\dualdomain$ for the dual of $\primaldomain$, which is called the \emph{dual domain}.
Mathematically, one can take for instance $\primaldomain = \Omega \cap 2\delta (\bbZ \times \bbR)$ and $\dualdomain = \Omega \cap \delta ( 2\bbZ + 1) \times \bbR)$.
We call the joint domain $ \medialdomain = \diamonddomain := \primaldomain \cup \dualdomain$ the \emph{medial domain}.
The graphs $\medialdomain$ and  $\diamonddomain$ are have the same vertex set but are denoted differently to be closer to the current litterature and avoid confusions on local relations of s-holomorphic functions and the primitive of their square. We discuss their graph properties in the following paragraph.
In this article, unless otherwise specified, $u$ (resp. $v$) mostly denotes a point on the primal (resp. dual) lattice.

We say that $v $ and $u$ are neighbours in $\medialdomain $ if they share the same vertical cordinate while their horizontal cordinate only differs by $\pm \delta $.
This makes $\medialdomain$ a bipartite graph, whose edges (linking primal to dual vertices) are in bijection with the $\emph{corner graph} $ $\midedgedomain$ (also called $\emph{midedge domain} $ in~\cite{Li-QI-SLE}).
Given $p \in \medialdomain$, write $p^-$ (resp. $p^+$) its closest left (resp. right) vertex in $\medialdomain$ (hence $p^{\pm} $ belong to the graph which is dual to the one containing $p$).

A corner $c$ (or a \emph{mid-edge}) is the middle of the segment formed by two neighbours in $\medialdomain$.
We write $u_c$ (resp. $v_c$) its neighboring primal (resp. dual) vertex make the identification $c=(v_c u_c)$.
We also write $c^-$ (resp. $c^+$) the closest left (resp. right) corner with the same vertical cordinate as $c$.
The set of edges linking nearby corners is in bijection with the diamond graph $\Omega_{\delta}^{\diamond}$ i.e. for any neighbouring corners $c_{1, 2}$, we set $z:=z(c_1, c_2)$ the middle of segment $[c_1 c_2]$, which belongs to $\Omega_{\delta}^{\diamond}$. 
See Figure~\ref{fig:graph_relation}.

\begin{figure}[htb] \centering
  \includegraphics[scale=1, page=1]{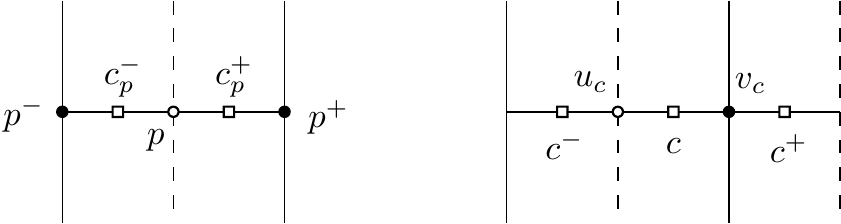}
  \caption{Notations}
  \label{fig:graph_relation}
\end{figure}

\subsection{The model}

Let $\Omega$ be a domain in $\bbC$.
Denote $\primaldomain$ (resp. $\dualdomain$) the semi-discretized primal (resp. dual) domain of $\Omega$.
We will study the quantum Ising model defined on $\primaldomain$ with + boundary condition, or its dual model, the quantum Ising model defined on $\dualdomain$ with free boundary condition.
Both models are defined as a Radon-Nikodym derivative with respect to some Poisson point processes.

Note that the following probability measures in different representations are only well-defined on bounded domains.
For unbounded domains, we follow a classical approach.
We construct compatible measures on a sequence of increasing bounded domains whose limit is the unbounded one, and the weak limit of the sequence of measures will be the measure on the unbounded domain.

Given two positive parameters $\tau$ and $\theta$.
We write $\PPP^\bullet_\tau$ to be the Poisson point process of intensity $\tau$ on primal lines of $\primaldomain$.
Similary, $\PPP^\circ_\theta$ denotes the Poisson point process of intensity $\theta$ on dual lines of $\dualdomain$.
The two Poisson point processes $\PPP^\bullet_\tau$ and $\PPP^\circ_\theta$ are assumed to be independent.
We write $\PPP_{\tau, \theta}$ to be the tensor product $\PPP^\bullet_\tau \otimes \PPP^\circ_\theta$.
We also write $\Lebp$ (resp. $\Lebd$) for the one-dimensional Lebesgue measure on primal lines of $\primaldomain$ (resp. dual lines of $\dualdomain$).
We drop the superscript and write $\Leb$ if the context is clear.

\subsubsection{The FK representation}

The FK representation is defined through the Poisson point process $\PPP_{\tau, \theta}$.
To define the notion of connected components for the FK representation, we shall see the points on primal lines to be death points (denoted $D$) and the points on dual lines to be bridges (denoted $B$) connected neighboring lines.
As such, the FK measure is given by the Random-Nikodym derivative, 
\[
  \frac{\dd \bbPFK (\omega)}{\dd \PPP_{\tau, \theta} (D, B)} \propto 2^{k(\omega)}, 
\]
where $k(\omega)$ counts the number of connected components in $\omega$ with some prescribed boundary condition that is omitted in the above notation.

% Recall that in our setting, we are interested in the free boundary condition if the model is defined on $\primaldomain$ and + boundary condition if the model is defined on $\dualdomain$. \comment{Here is tipically where I feel there is no need to specify every time. In some respect the important thing is the space-time model in that case, which is the one who carries more clearly the physic	al content.}

\subsubsection{The space-time spin representation}

The space-time spin representation is defined through the Poisson point process $\PPP^\bullet_\tau$ on the primal lines.
Write $D$ for a (countable) set of points given by this point process, which are called \emph{death points}.
% A function $\sigma : \primaldomain \to \{ +1, -1 \}$ is said to be a \emph{spin configuration} (compatible with $D$) if it is continuous on $\primaldomain \backslash D$.
% \comment{R: Here is for example where you can shorten, just say it is constant on each connected component}
% In other words, $\sigma$ is constant on each connected component of $\primaldomain \backslash D$.
A function $\sigma : \primaldomain \to \{ +1, -1 \}$ is said to be a \emph{spin configuration} (compatible with $D$) if it is constant on each connected component of $\primaldomain \backslash D$.
Write $\Sigma(D)$ for the set of spin configurations compatible with $D$.
Note that if $\sigma \in \Sigma(D)$ for some countable set $D$, then $\sigma \in \Sigma(D')$ for any $D' \supseteq D$.
This shows that $\Sigma(D)$ is increasing for inclusion and that $\Sigma(D)$'s are not pairwise disjoint.

The spin measure is characterized as follows.
Given a measurable function $F : \bigcup \Sigma(D) \to \bbR$, its expectation is given by
\begin{equation} \label{eq:spin_F}
  \bbPspin [ F(\sigma) ] :=
  \frac{  \PPP^\bullet_\tau \left[ \sum_{\sigma \in \Sigma(D)} F(\sigma) \exp \left(- \frac{\theta}{2} \calH(\sigma) \right) \right] }
  { \PPP^\bullet_\tau \left[ \sum_{\sigma \in \Sigma(D)} \exp \left(- \frac{\theta}{2} \calH(\sigma) \right) \right] }, 
\end{equation}
where $\theta > 0$ and $\calH$ denotes the Hamiltonian,
\[
  \calH(\sigma) := - \Lebd( \eps_{v} ), \qquad \mbox{ where } \quad \eps_v = \sigma_{v^+} \sigma_{v^-} \mbox{ for } v \in \dualdomain.
\]
The denominator in the above Equation~\eqref{eq:spin_F} is called the \emph{partition function}, 
\begin{equation} \label{eq:spin_pf}
  Z (\tau, \theta)
  := \PPP^\bullet_\tau \left[ \sum_{\sigma \in \Sigma(D)} \exp \left(- \textstyle \frac\theta2 \calH(\sigma) \right) \right]
  = \PPP^\bullet_\tau \left[ \sum_{\sigma \in \Sigma(D)} \exp \left( \textstyle \frac\theta2 \cdot \Lebd(\eps_v) \right) \right].
\end{equation}

Note that we could have defined $\Sigma'(D)$ to be the subset of $\Sigma(D)$ such that $\sigma \in \Sigma'(D)$ changes value at \emph{every} point of $D$.
As such, $\Sigma'(D)$'s are disjoint and the following Radon-Nikodym derivative makes sense.
For $\sigma \in \Sigma'(D)$, one defines
\[
  \frac{\dd \bbPspin (\sigma)}{\dd \PPP^\bullet_\tau (D)}
  \propto \exp \left(- \textstyle \frac{\theta}{2} \calH(\sigma) \right).
\]
In this case, \eqref{eq:spin_F} and~\eqref{eq:spin_pf} stay the same except that $\Sigma(D)$ is replaced by $\Sigma'(D)$ in the summations.

Additionally, we can define the energy density $\psi_\sigma : \dualdomain \to \{ +1, -1\}$ by $\psi_\sigma(v) = \eps_v$ for $v \in \dualdomain$.
Let $E_+(\sigma) = \psi_\sigma^{-1}(+1)$, $E_-(\sigma) = \psi_\sigma^{-1}(-1)$ and $S_\pm(\sigma) = \Lebd(E_\pm(\sigma))$.
Using the fact that $S_+(\sigma) + S_-(\sigma) = \Lebd(\dualdomain)$ is a constant not depending on $\sigma$ (but on $\dualdomain$) and that $\calH(\sigma) = S_+(\sigma) - S_-(\sigma)$, we can rewrite, 
\begin{equation} \label{eq:spin_measure}
  \frac{\dd \bbPspin (\sigma) }{\dd \PPP^\bullet_\tau (D)}
  \propto \exp(- \theta S_-(\sigma))
  \propto \exp( \theta S_+(\sigma)).
\end{equation}

Let $A$ be a subset of points in $\primaldomain$ and write $\sigma_A = \prod_{x \in A} \sigma_x$ for a spin configuration $\sigma \in \coprod \Sigma'(D) = \bigcup \Sigma'(D)$.
We have, 
\begin{equation} \label{eq:low-T}
  \bbPspin ( \sigma_A ) =
  \frac{  \PPP^\bullet_\tau \left[ \sum_{\sigma \in \Sigma'(D)} \sigma_A \exp \left(- \theta S_-(\sigma) \right) \right] }
  { \PPP^\bullet_\tau \left[ \sum_{\sigma \in \Sigma'(D)} \exp \left(- \theta S_-(\sigma) \right) \right] }, 
\end{equation}

\subsubsection{Random-parity representation}

The randon-parity representation is particularly useful to write the $n$-spin correlation in an alternative way.
It is defined with resepect to the Poisson point process $\PPP^\circ_{\theta/2}$ defined on $\dualdomain$, which, in the aforementioned FK-representation, can be interpreted as bridges.
Wirte $B$ for a (countable) set of points given by $\PPP^\circ_{\theta/2}$.
Denote $B^\pm \subseteq \primaldomain$ the set of extremities of bridges in $B$, i.e. $B^\pm = \bigcup_{e \in B} \{ e^+, e^- \}$.

Consider additionally a finite subset $A \subseteq \primaldomain$ called \emph{source}.
A function $\psi : \primaldomain \to \{ 0, 1 \}$ is said to be a \emph{random-parity function}, with source $A$ and compatible with $B$, if
(a) it is continuous on $\primaldomain \backslash (A \cup B^\pm)$ and
(b) the subgraph $(A \cup B^\pm, \psi^{-1}(1))$ has an even degree at vertices in $B^\pm$ and an odd degree at vertices in $A$.
Define $I(\psi) = \Lebp(\psi^{-1}(1))$ to be the one-dimension Lebesgue measure of $\psi^{-1}(1)$.
Write $\Psi_A(B)$ for the set of such random-parity functions.

The random-parity function being defined for any finite source $A$, the $n$-spin correlation reads, 
\begin{equation} \label{eq:high-T}
  \bbPspin( \sigma_A )
  = \frac{ \PPP^\circ_{\theta/2} \left[ \sum_{\psi \in \Psi_A(B)} \exp \left(- 2 \tau I(\psi) \right) \right] }
  { \PPP^\circ_{\theta/2} \left[ \sum_{\psi \in \Psi_\emptyset(B)} \exp \left(- 2 \tau I(\psi) \right) \right] }
\end{equation}
Same as we said just after~\eqref{eq:spin_pf}, we could have also defined $\Psi'_A(B)$ (resp. $\Psi'_\emptyset(B)$) with a mandatory discontinuity at \emph{every} point of $A \cup B^\pm$ (resp. $B^\pm$) and replaced $\Psi_A(B)$ and $\Psi_\emptyset(B)$ with $\Psi'_A(B)$ and $\Psi'_\emptyset(B)$ in~\eqref{eq:high-T}.

\subsection{FK-spin coupling}
\label{sec:FK-spin_coupling}

Given a spin configuration $\sigma$ and a FK configuration $\omega$, we say that they are \emph{compatible} if $x$ and $y$ are in the same component in $\omega$, then the spins $\sigma_x$ and $\sigma_y$ coincide.
Below is the construction of the Edward-Sokal coupling for the quantum Ising model.

Given a FK configuration $\omega$ sampled according to $\bbPFK$, define a (random) spin configuration $\sigma$ in the following way.
To each connected component of $\omega$, we choose a spin ($+$ or $-$) uniformly at random.
We write $\bbP_1$ for this measure.
% We write $\bbP_1$ for this measure and write its Radon-Nikodym derivative with respect to $\PPP_{\tau, \theta}$, 
% \[
%   \frac{\dd \bbP_1 (\omega, \sigma)}{\dd \PPP_{\tau, \theta} (D, B)}
%   \propto 2^{k(\omega)} \cdot \left( \frac12 \right)^{k(\omega)} = 1
% \]

Conversely, given a spin configuration $\sigma$ sampled according to $\bbPspin$, define a (random) FK configuration $\omega$ by adding Poisson points of parameter $\theta$ on the dual lines where the neighboring spins coincide, which are bridges connected different components having the same spin.
We write $\bbP_2$ for this measure. Using elementary properties of the Poisson point processes, one can show that the two measures $\bbP_1$ and $\bbP_2$ have the same  distribution.

\subsection{Kramers-Wannier duality}
\label{sec:KW-duality}

The space-time spin representation and the random-parity representation are dual representations to each other. To be more precise, the space-time spin representation corresponds to the so-called \emph{low-temperature} expansion.
A spin configuration defined on $\primaldomain$ is in bijection with collection of contours on $\dualdomain$ which appear as interfaces between + and - spins, see Equation~\eqref{eq:spin_measure} and Figure~\ref{fig:KW_duality}.
As such, the spin measure is exactly the measure defined on the set of contours.

In a similar manner, the random-parity representation corresponds to the so-called \emph{high-temperature} expansion.
It does not consist in rewriting the probability measure configuration-wise.
Instead, it gives an alternative way to rewrite the $n$-point correlation, after summing/integrating over spin configurations, using contours and paths as in Equation~\eqref{eq:high-T}.
An example is given in Figure~\ref{fig:KW_duality}.

\begin{figure}[htb]
  \centering
  \includegraphics[scale=1, page=2]{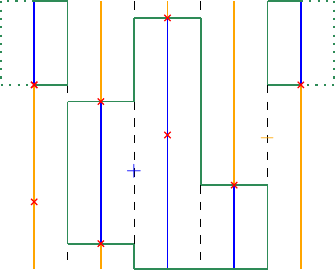}
  \qquad
  \qquad
  \includegraphics[scale=1, page=3]{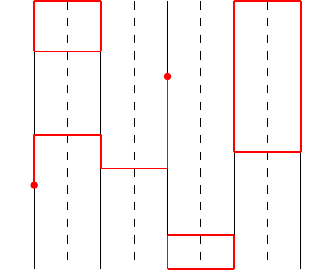}
  \caption{We work with a primal domain with free boundary condition.
    \textbf{Left:} An example of the space-time spin representation with + spin in blue and - spin in orange. The contours separating different spins are given in green.
    \textbf{Right:} An example of the random-parity representation with source given by two points. }
  \label{fig:KW_duality}
\end{figure}

Note that the Kramers-Wannier duality reads as, 
\begin{equation} \label{eq:KW_duality}
  (\tau^\star, \theta^\star) = (\textstyle \frac\theta2, 2 \tau).
\end{equation}
It is an involution, i.e. $((\tau^\star)^\star, (\theta^\star)^\star) = (\tau, \theta)$.
The self-dual point is given by $\frac{\tau}{\theta} = 2$ which is also the crtical point of the quantum Ising model, see~\cite{BG-QI-sharp, DCLM-universality}.

We remind the reader that the ``usual'' way of talking about sub- or super-criticality differs in the spin- and the FK-representation.
Due to the coupling mentioned in Section~\ref{sec:FK-spin_coupling}, the existence of an infinite cluster in the FK-representation is equivalent to a positive spontaneous magnetization.
However, in the FK setting, we usually refer this case as \emph{super-critical}, whereas in the spin setting, we usually refer this as \emph{sub-critical}.

\subsection{From discrete to semi-discrete}

% \comment{R 18/02 :  As discussef, we keep this part, fill it explaining that one can rederive results for interface but normalizations procedure doesen't survive to unbounded angles, thus it is not possible a priori to derive the result passing to exchange the limits $angles \to 0 $ and $\delta \to 0 $. }

The different representations described above can be seen as their corresponding representations in the classical two-dimension Ising model on a flattened lattice.
As shown in Figure~\ref{fig:flattened_lattice}, there are two types of edges, long ``horizontal'' edges of length $\frac{\delta}{2} \cos \frac\eps2$ and short ``vertical'' edges of length $\frac{\delta}{2} \sin \frac\eps2$ in the flattend lattice $\calG^\eps$ and we note that when $\eps \to 0$, it ``converges'' to the semi-discrete lattice.

\begin{figure}[htb]
  \centering
  \includegraphics[scale=0.75, page=1]{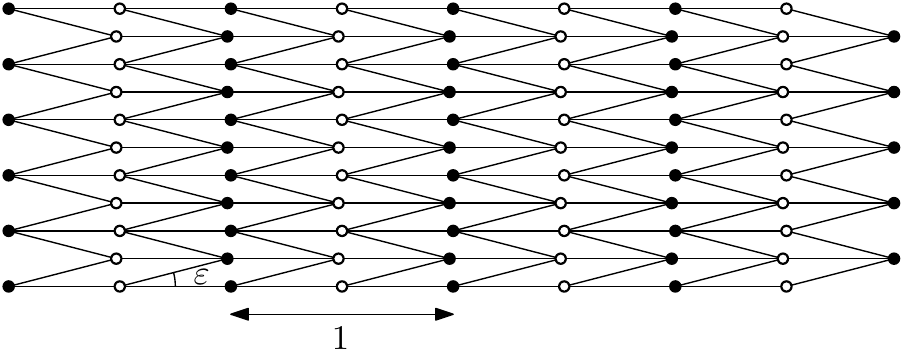}
  \hfill
  \includegraphics[scale=0.75, page=2]{images/flattened_lattice.pdf}
  \caption{The flattened lattice $\calG^\eps$ which ``converges'' to the semi-discrete lattice in the limit. \textbf{Left:} The isoradial representation (diamond graph) when $\delta = \frac12 $. \textbf{Right:} The primal lattice with dual vertices.}
  \label{fig:flattened_lattice}
\end{figure}

At discrete level, let us fix the parameters as follows.
To each edge $e$ one associates the coupling constant $J_e > 0$ which only depends on the type, horizontal or vertical: each horizontal edge $e$ has the same coupling constant $J_h$ and each vertical edge has the same coupling constant $J_v$.
And we recall that the Hamiltonian in the discrete model is given by
\begin{align*}
  \calH(\sigma) & = - \sum_{e \in E} J_e \sigma_x \sigma_y, \qquad \mbox{ where } e = (x, y) \mbox{ for } e \in E.
\end{align*}
Note that in the above expression, one drops the additional parameter $\beta$ which is the inverse temperature.
One can include it in $J_e$ and the regime of the model (subcritical, critical, supercritical) is thus determined by how the parameters $J_e$ are chosen.

In the low-temperature expansion, a spin configuration is in bijection with a collection of contours in the dual lattice.
For each edge $e$ in the primal lattice, the weight of its dual edge is given by $t_e^\star = e^{-2 J_e}$.
The coupling constatns $J_h$ and $J_v$ are chosen such that
% $t_h = 1 - \theta \eps + \calO(\eps^2)$ and $t_v = \tau \eps + \calO(\eps^2)$
$1 - t_h^\star \sim \theta \eps$ and $t_v^\star \sim \tau \eps$.
Recall that the dual of a horizontal edge is vertical and vice versa.
As such, when the limit $\eps \to 0$ is taken, one obtains contours defined by Poisson Point Processes as weak limit of families of i.i.d. Bernoulli random variables whose parameters are scaled as described above.

In the corresponding high-temperature expansion, we make use of $\tanh(J_e) = \frac{1 - t_e^\star}{1 + t_e^\star}$.
This gives $\tanh(J_h) \sim \frac{\theta}{2} \eps$ and $1 - \tanh(J_v) \sim 2 \tau \eps$.
Using the same argument to obtain Poisson Point Processes, this also confirms that the dual model is given by parameters as shown in~\eqref{eq:KW_duality}.

Above we describe the procedure from discrete to semi-discrete in the spin-representation.
When it comes to the FK-representation, we refer interested readers to~\cite[Section~5.1]{DCLM-universality}.

\medskip

This similarity between the discrete model and the semi-discrete representation of the quantum model suggests that the same results in the classical Ising model are expected to hold in the quantum setting.
However, when dealing with universality of results in the discrete setting, one important assumption is the so-called \emph{bounded-angle property}, i.e. angles of the graph need to be uniformly bounded away from 0 and $\pi$.
This property does not survive in the procedure from discrete to semi-discrete for obvious reasons.

In~\cite{DCLM-universality}, authors obtained finer estimations in this particular procedure from discrete to semi-discrete, thus extending  phase transition results from discrete graphs described in Figure~\ref{fig:flattened_lattice} to the semi-discrete lattice $2\bbZ \times \bbR$. This approach might be tempting knowing universality of spin-correlation on isoradial graphs proven in~\cite{CIM-universality}, provided they keep the bounded angle-property. Still, in this case, additional difficulties rise due to local normalizing factor appearing before smashing the grid in the vertical direction. Thus, we prefer to present a more straightforward and elegant method by working directly in the semi-discrete setting, following the spirit of~\cite{Li-QI-SLE}.

% \begin{figure}[htb]
%   \centering
%   \includegraphics[scale=1.3, page=3]{images/flattened_lattice.pdf} \\
%   \includegraphics[scale=1.3, page=4]{images/flattened_lattice.pdf}
%   \caption{The spin representation (low-temperature) and the random-parity (high-temperature) representation on the flattened lattice.
%   \textbf{Top:} $+$ spins are in blue and $-$ spins are in orange. The boundary condition + is marked with blue squares and contours are drawn in green.
%   \textbf{Bottom:} An example of the random-parity representation (in red) with source marked with red squres. All the vertices are of even degree except those in the source which are of odd degree. }
%   \label{fig:discrete_KW}
% \end{figure}

\section{Results} \label{sec:results}

In this section, we give a summary of the main results proved in this paper.
They are expressed in terms of the space-time spin representation of the quantum Ising model.
They also extend a series of results proved on isoradial graphs~\cite{HS-energy-Ising, Hongler-PhD, CHI-spin, CIM-universality} to the semi-discrete lattice, showing the existence and the conformal covariance of a scaling limit in simply connected domains.
Such an extension has previously been made for interfaces in the FK-loop representation~\cite{Li-QI-SLE}, but more work is required in this paper for general correlations, one of the reasons being that normalizing factors depending on the mesh size $\delta$ of the lattice are now required.

For concision, we mostly state those results for with $+$ boundary conditions, although one can extend them to any boundary condition composed of a finite sequence of boundary arc in $\{ +, -, \free \}$ using the theory developped in~\cite{CHI-mixed}.
For more general boundary conditions, the scaling limit is simply muliplied by a conformally invariant factor.
It is also enough to prove the convergence results in smooth domains, since the monotonicity with respect to boundary conditions and approximations by smooth domains allow to prove convergence in general domains once it is done in smooth domains.

Since the (1+1)-dimensional quantum Ising model has several equivalent representations, the results will be stated in the three following settings.
Unless otherwise specified, the model is considered at its criticality, i.e., $\theta = \theta^\star$.
\begin{enumerate}
  \item In $\Omega_\delta$, which is the semi-discretization by $2\delta(\bbZ \times \bbR)$ of a given simply connected domain $\Omega \subset \bbC$.
  The quantum Ising measure at the criticality will be denoted by $\QIOmega$ in this case.
  \item In $\bbH_\delta := 2 \delta (\bbZ \times \bbR)$, which can be regarded as the graphical time-evolution of the quantum Ising model on $2 \delta \bbZ$, or as the special case $\Omega = \bbH$ in the first point.
  The quantum Ising measure at the criticality will be denoted by $\QIH$ in this case.
  \item In $\calR_\delta(k)$, where $\calR_\delta(k)$ stands for the semi-discretization by $2 \delta( \bbZ \times \bbR)$ of $\calR(k)$, which is the rectangle defined in Appendix~\ref{sec:hyperbolic_rectangle}.
  We have that $\calR(k) = (-K(k), K(k)) \times (0, K'(k))$, where, for $k \in (0, 1)$,
  \[
    K(k) := \int_0^1 \frac{\dd s}{ \sqrt{ (1-s^2)(1-k^2s^2) } }
  \]
  is the complete elliptic integral of the first kind, and $K'(k) = K(\sqrt{1-k^2})$.
  Since $k \mapsto K(k) / K'(k)$ is a diffeomorphism between $(0, 1)$ and $(0, +\infty)$, the set $\{ \calR(k), k \in (0, 1)\}$ describes all the possible aspect ratios of a rectangle.
  More details are given in Appendix~\ref{sec:hyperbolic_rectangle}.
  This can be understood as the thermodynamic limit of a finite-size quantum Ising model where we scale both the number of vertices and the time together by $\delta^{-1}$.
  The corresponding quantum Ising measure at the criticality will be denoted by $\QIrect$.
\end{enumerate}

In the second (resp. the third) setting, due to the natural space-time representation, we introduce some additional notations.
For a space-time point $(x^\delta, t)$ in $\bbH_\delta$ (resp. $\calR_\delta(k)$

\medskip

Our first result states the conformal covariance of the horizontal energy density in simply connected domains.

\begin{theo}\label{thm:single_energy}
  Let $a\in \Omega $ an interior point of a simply connected domain approximated by $a^{\delta} \in \primaldomain$.
  Set $\epsilon_{a^{\delta}}:= \sigma_{a^{\delta}} \sigma_{a^{\delta}+2\delta} - \bbE_{\bbC_\delta}^{+}[\sigma_{a^{\delta}}\sigma_{a^{\delta}+2\delta}] $ the \emph{energy density} random variable.
  For the critical space-time representation of the quantum Ising model, one has
  \begin{equation}
    \tfrac{1}{\delta} \,\bbE_{\primaldomain}^{+}[\epsilon_{a^{\delta}}]
    \underset{\delta \to 0}{\longrightarrow}
    \tfrac{1}{\sqrt{2}\pi} \ell_{\Omega}(a),
  \end{equation}
  where $\ell_{\Omega}(a)$ is the hyperbolic metric of $\Omega $ at the point $a$, i.e. $\ell_{\Omega}(a)$ is the modulus of the derivative at $a$ of any conformal from $\bbD $ to $\Omega $ vanishing at $a$. The above convergence is uniform over points remaining at a definite distance from $\partial \Omega $.
\end{theo}

The substraction of the full-plane energy density $\bbE_{\bbC_\delta}^{+}[\sigma_{a^{\delta}}\sigma_{a^{\delta}+2\delta}]=\frac{\sqrt{2}}{2}$ (see Remark~\ref{rem:energy_polynomials}) is necessary to underline the influence of the domain $\Omega $.
Indeed, using standard Russo-Seymour-Welch type estimates~\cite{DCLM-universality} and comparaison between boundary conditions, one can easily see that $\bbE_{\Omega_\delta}^{+}[\epsilon_{a^{\delta}}] \to 0$ as $\delta \to 0 $. The effect of the boundary of $\Omega $ reads in the sub-leading term of the expansion of $\bbE_{\Omega_\delta}^{+}[\sigma_{a^{\delta}}\sigma_{a^{\delta}+2\delta}]$.

One can also state a more general result involving several energy density variables, whose proof is based upon the special case $r=1$ and the Pfaffian structure of the Ising Model.

\begin{theo}\label{thm:multiple_energy}
  Let $\Seq{a_{?1}} \in \Omega$ be distinct interior points of $\Omega$, approximated respectively by $\Seq{a_{?1}^{\delta}}$ in $\primaldomain$.
  In the setup of the previous theorem, one has
  \begin{equation}
    \tfrac{1}{\delta^n}\,
    \bbE_{\primaldomain}^{+} [ \varepsilon_{a^{\delta}_1}\cdots\varepsilon_{a^{\delta}_n} ]
    \underset{\delta \to 0}{\longrightarrow}
    \langle \Seq[]{\epsilon_{a_{?1}}} \rangle_{\Omega}^{+},
  \end{equation}
  where the multiple energy correlator is a conformally covariant quantity defined via ~\eqref{defn:multiple_energ_correlator}.
  Moreover, the convergence is uniform over points $\Seq{a_{?1}}$ remaining at a definite distance from each other and from the boundary.
\end{theo}

Let $\bbK(a_1,\ldots,a_r)$ be the $r \times r$ matrix defined by $\bbK_{i,j} := \mathbf{1}_{i\neq j} (a_i-a_j)^{-1} $ and denote $\Pf $ the usual Pfaffian operator on matrices.

\begin{thm}\label{thm:energy_formulas}
  Consider the critical Ising quantum spin chain on $2\delta \bbZ $, fix $(x_k)_{1\leq k \leq r} $ distinct points of $\bbR $ approximated by $(x^{\delta}_k)_{1 \leq k \leq r} $ on $2\delta \bbZ $. Fix also a sequence of positive instants $(t_k)_{1 \leq k\leq r} $. For a space-time configuration $(x^\delta,t) $ in $2\delta \bbZ \times \bbR^{\star}_{+} $, denote  $\epsilon^{(t)}_{x^{\delta}}:= \sigma^{(t)}_{x^{\delta}}\sigma^{(t)}_{x^{\delta}+2\delta} - \frac{\sqrt{2}}{2}$ the horizontal energy density at site $x^\delta$ at time $t$. On has
  \begin{equation}\label{eq:energy_half_plane_+}
    \tfrac{1}{\delta}
    \,\bbE_{\mathrm{QI}_{\delta}}^{+}[\epsilon^{(t)}_{x^{\delta}}]
    \underset{\delta \to 0}{\longrightarrow}
    \tfrac{1}{\sqrt{2}\pi t},
    \quad \tfrac{1}{\delta}
    \,\bbE_{\mathrm{QI}_{\delta}}^{\free}[\epsilon^{(t)}_{x^{\delta}}]
    \underset{\delta \to 0}{\longrightarrow}
    -\tfrac{1}{\sqrt{2}\pi t},
  \end{equation}
  \begin{equation}
    \tfrac{1}{\delta^n}
    \,\bbE_{\mathrm{QI}_{\delta}}^{+} \big[ \varepsilon_{x_{1}^{\delta}}^{(t_1)} \cdots \varepsilon_{x_{n}^{\delta}}^{(t_n)} \big]
    \underset{\delta \to 0}{\longrightarrow}
    \tfrac{1}{(\sqrt{2}i\pi)^n} \Pf \big[ \bbK( \Seq{x_{?1} + \icomp t_{?1}}, \Seq{x_{?1} - \icomp t_{?1}} ) \big], \nonumber
  \end{equation}
  \begin{equation}
    \tfrac{1}{\delta}
    \,\QIrect[+] [\epsilon^{(t)}_{x^{\delta}}]
    \underset{\delta \to 0}{\longrightarrow}
    \tfrac{1}{\sqrt{2}\pi} \tfrac{ \cn(x + \icomp t, k) \dn(x + \icomp t, k) }{ \myi \sn(x + \icomp t, k) }, \nonumber \\
    \tfrac{1}{\delta} \QIrect[\free] [\epsilon^{(t)}_{x^{\delta}}]
    \underset{\delta \to 0}{\longrightarrow}
    -\tfrac{1}{\sqrt{2}\pi} \tfrac{ \cn(x + \icomp t, k) \dn(x + \icomp t, k) }{ \myi \sn(x + \icomp t, k) }, \nonumber
  \end{equation}
  where $\cn ,\dn ,\sn $ denote usual Jacobi theta elliptic function, $QI^{\delta} $ denotes critical the quantum spin chain on $2\delta\bbZ $ while $QI^{\delta}(k) $ denotes the quantum spin chain on $(-K(k),K(k)) \cap 2\delta \bbZ $ until time $K'(k)$.
\end{thm}

\begin{rem}
  The convergence un the upper-half plane is invariant by translation, which comes directly from translation invarariance of the model.
  One also notices that the energy density blows up near time $t=0$.
  For example, in the case of $+$ boundary conditions, both nearby spins involved in the energy density are more likely to be correlated as they feel the $+$ imposed at the boundary when they approach it.
  With the $\free$ boundary conditon, it is the other way around and the nearby spins are less likely to be correlated as approaching the boundary.
\end{rem}

We now state the results concerning the general spin-correlations, i.e. the correlations between spins at macroscopic distances from each other.

\begin{theo}\label{thm:ratios_spins}
  Let $\Omega $ be a simply connected domain and $a_1,\ldots,a_n\in \Omega$ distinct interior points approximated respectively by $a_{1}^{\delta},\ldots,a_{n}^{\delta} $ in $\Omega^{\bullet}_{\delta}$. Let also $b_1,\ldots,b_n\in \Omega$ be distinct interior points approximated respectively by $b_{1}^{\delta},\ldots,b_{n}^{\delta}$ in $\Omega^{\bullet}_{\delta}$. For the critical space-time representation of the quantum Ising model one has
  \begin{equation}
    \frac{\bbE_{\primaldomain}^{+}[\sigma_{a^{\delta}_1}\cdots\sigma_{a^{\delta}_n}]}{\bbE_{\primaldomain}^{+}[\sigma_{b^{\delta}_1}\cdots\sigma_{b^{\delta}_n}]} \underset{\delta \to 0}{\longrightarrow} \frac{\langle \sigma_{a_1}\ldots\sigma_{a_n} \rangle_{\Omega}^{+}}{\langle \sigma_{b_1}\ldots\sigma_{b_n} \rangle_{\Omega}^{+}},
  \end{equation}
  where the function $\langle \sigma_{(\cdot)}\ldots \sigma_{(\cdot)} \rangle_\Omega^{+} $ is defined in~\eqref{eq:def_correlation_+}. The convergence is again uniform over points  remaining at a definite distance from each other and from the boundary.
\end{theo}

Using Krammer-Wannier duality and the formalism of disorders (which can be understood as dual-spins), one can also prove the convergence of correlation ratios between primal and dual model.

\begin{theo}\label{thm:disorder-spins}
  Let $u^{\delta}_{1} \in \primaldomain $ and $v^{\delta}_{1} \in \dualdomain $ approximating $a_1 \in \overset{\circ}{\Omega}$ (respectively  $u^{\delta}_{2} \in \primaldomain $ and $v^{\delta}_{2} \in \dualdomain $ approximating $a_2 \in \overset{\circ}{\Omega}$). One has
  \begin{equation}
    \frac{\bbE_{\Omega^{\circ}_\delta}^{\free}[\sigma_{v^{\delta}_1}\sigma_{v^{\delta}_2}]}{\bbE_{\primaldomain}^{+}[\sigma_{u^{\delta}_1}\sigma_{u^{\delta}_2}]} \underset{\delta \to 0}{\longrightarrow} \calB_{\Omega}(a_1,a_2),
  \end{equation}
  where the coefficient $\calB_{\Omega}(a_1,a_2) $ is defined in~\eqref{eq:def_coefficient_B} via the expansion of the  boundary value problem~\eqref{eq:BVP-spins} near one of its branchings.
\end{theo}

Using the last two theorems and an induction, one can recover the convergence of rescaled correlation, with a fully explicit lattice dependant normalization. The next theorem states this complete result.

\begin{theo}\label{thm:spins}
  Set $\calC := 2^{\frac{1}{6}}e^{\frac{3}{2}\zeta'(-1)} $.
  Under the previous hypothesis, given $a_1, \dots, a_n$ interior points of $\Omega $ approximated by $a^{\delta}_1,\ldots,a^{\delta}_n $ in  $\Omega^{\bullet}_{\delta}$, one has
  \begin{equation}
    \delta^{-\frac{n}{8}} \; \bbE_{\primaldomain}^{+}[\sigma_{a^{\delta}_1}\cdots\sigma_{a^{\delta}_n}] \underset{\delta \to 0}{\longrightarrow} \calC^{n} \langle \sigma_{a_1}\ldots\sigma_{a_n} \rangle_{\Omega}^{+}.
  \end{equation}
  The convergence is uniform over points remaining at a definite distance from each other and from the boundary.
\end{theo}

Once the previous theorem is proven, one can compute explicit formulas in specific domains to recover analytic expressions of correlation. Those expressions can be used to compute space-time correlations of the one dimensional quantum Ising model. The next theorem provides those formulas for the two most natural setups.

\begin{thm}\label{thm:spins_formulas}
  Consider the critical one dimensional Ising quantum spin chain on $2\delta \bbZ $. Set 
  $x_1,\ldots,x_n $ horizontal cordinates approximatated by $x^{\delta}_1,\ldots,x^{\delta}_n $  and $t_1,\ldots,t_n $ positives times. Let $\sigma^{(t)}_{x^{\delta}} $ be the spin variable $x^{\delta}$ at time $t$. One has the space-time spin-spin correlations,
  \begin{align*}
    \delta^{-\frac{n}{8}} \; \bbE_{\mathrm{QI}_{\delta}}^{+}[\sigma^{(t_1)}_{x^{\delta}_1} \cdots \sigma^{(t_n)}_{x^{\delta}_n} ] & \underset{\delta \to 0}{\longrightarrow}  \calC^n \prod_{r=1}^n \big( \tfrac{2}{t_r} \big)^{\frac18} 
    \Big( 2^{-\frac{n}{2}} \sum_{\mu \in \{ \pm 1 \}^n }
    \prod_{1 \leq r < m \leq n} \Big| \tfrac{(x_r - x_m)^2 + (t_r - t_m)^2}{(x_r - x_m)^2 + (t_r+t_m)^2} \Big|^{\frac{\mu_r \mu_m}{4}} 
    \Big)^{\frac12}, \\
    \frac{\bbE_{\mathrm{QI}_{\delta}(k)}^{+}[\sigma^{(t_1)}_{x^{\delta}_1} \cdots \sigma^{(t_n)}_{x^{\delta}_n} ]}{\bbE_{\mathrm{QI}_{\delta}}^{+}[\sigma^{(t_1)}_{x^{\delta}_1} \cdots \sigma^{(t_n)}_{x^{\delta}_n} ]} & \underset{\delta \to 0}{\longrightarrow}  \prod_{j=1}^n | \cn(x_j + i t_j, k) \dn(x_j + i t_j, k) |. \nonumber    
  \end{align*}
  In particular, the one-point function and the two-point function write,
  \begin{align*}
    \delta^{-\frac{1}{8}} \; \bbE_{\mathrm{QI}_{\delta}}^{+}[\sigma_{x}^{(t)}]
    & \underset{\delta \to 0}{\longrightarrow} \calC (\tfrac{2}{t})^{\frac{1}{8}}, \\
    \delta^{-\frac{1}{4}} \; \bbE_{\mathrm{QI}_{\delta}}^{+}[\sigma_{x_1}^{(t_1)}\sigma_{x_2}^{(t_2)}]
    & \underset{\delta \to 0}{\longrightarrow} \calC^2 (\tfrac{1}{t_1 t_2})^{\frac{1}{8}} \big{[}  \big| \tfrac{(x_1-x_2)^2 + (t_1+t_2)^2}{(x_1-x_2)^2 + (t_1-t_2)^2} \big|^{\frac14} + \big| \tfrac{(x_1-x_2)^2 + (t_1-t_2)^2}{(x_1-x_2)^2 + (t_1+t_2)^2} \big|^{\frac14}  \big]^\frac12, \\
    \delta^{-\frac{1}{4}} \; \bbE_{\mathrm{QI}_{\delta}}^{\free}[\sigma_{x_1}^{(t_1)}\sigma_{x_2}^{(t_2)}]
    &\underset{\delta \to 0}{\longrightarrow} \calC^2 (\tfrac{1}{t_1 t_2})^{\frac{1}{8}} \big{[}  \big| \tfrac{(x_1-x_2)^2 + (t_1+t_2)^2}{(x_1-x_2)^2 + (t_1-t_2)^2} \big|^{\frac14} - \big| \tfrac{(x_1-x_2)^2 + (t_1-t_2)^2}{(x_1-x_2)^2 + (t_1+t_2)^2} \big|^{\frac14}  \big]^\frac12.    
  \end{align*}
\end{thm}

The lattice dependant scaling factor $\delta^{\frac{n}{8}}\calC^{n}$ is obtained by replacing the expectation in bounded regions by the expectation in the full-plane.
Theorems~\ref{thm:energy_formulas} and~\ref{thm:spins_formulas} can be seen as a byproduct of the space-time interpretation for the critical 1D Ising model, as well as explicit formulas for correlation functions in the upper-half plane and in rectangles. 
\bigskip

In the above-mentioned theorems, explicit lattice scaling factors appear and are related to the full-plane asymptotics of the two points functions, which we compute explicitely, already on the semi-discrete lattice. We developp a formalism, close to the one introduced in~\cite{CHM-zig-zag-Ising}, that allows to derive the desired asymptotics in the vertical direction. Combined with the convergence of correlation ratios mentioned above, it implies in particular the rotational invariance of the critical model, together with its explicit two point function.

\begin{theo}\label{thm:rotational_invariance}
  For the critical space-time representation of the quantum Ising model in full-plane $\bbC_{1}$, one has the asymptotic
  \begin{equation}
    \bbE^{+}_{\bbC_{1}}[\sigma_0 \sigma_{z}]
    \underset{z \to \infty}{\sim}
    \calC^2 |2z|^{-\frac{1}{4}}.
  \end{equation} 
\end{theo}
In particular, the critical model is rotationally invariant. The uniqueness of the Gibbs measure at criticality allows to drop $+$ boundary condition at infinity from the definition. We now pass to the results outside of criticality, obtained by the same formalism of Topelitz+Hankle determinants.

\begin{theo}
  For the sub-critical $(\theta < \theta^\star) $ space-time representation of quantum Ising model in $\bbC_{1}$ with parameters $(\theta,\theta^{\star})$ and $+$ boundary conditions at infinity, the spontaneous magnetization satisfies
  \begin{equation}
    \label{eq:M=(KOY-thm)}
    \calM(\theta, \theta^{\star})\ := \lim\limits_{n\to \infty} \bbE^{+}_{\bbC_{1}}[\sigma_0 \sigma_{n}]^{1/2}
    \ =\ (\theta^2 + {\theta^\star}^2)^{-\frac12} (1 - (\tfrac{\theta}{\theta^\star})^2)^{\frac18}.
  \end{equation}
\end{theo}
This result generalizes the full-plane magnetization result below criticality proven for rectangular grids in~\cite{CHM-zig-zag-Ising} and in general $Z$-invariant isoradial graphs with bounded angles in~\cite{CIM-universality}.
It is known by~\cite{DCLM-universality} that above criticality, the spin-spin correlation decays exponentially fast.
In particular, using quasi-multiplicativity arguments, one can see that the so-called \emph{correlation length}~\cite[Thm.~1.5]{DCLM-universality} in the horizontal direction $\xi $ is well defined.
The next result gives its exact value, extending to the horizontal direction the result proved in~\cite{beffara-duminil}.
\begin{theo}
  For the super-critical $(\theta > \theta^\star) $ quantum Ising model in full-plane $\bbC_{1}$ with parameters $(\theta,\theta^{\star})$ and $+$ boundary conditions at infinity, one has 
  \begin{equation}
    \xi = \lim\limits_{n\to \infty} -\tfrac{1}{n} \log \bbE^{+}_{\bbC_{1}}[\sigma_0\sigma_{n}]
    = \tfrac{1}{2} \log( \tfrac{\theta}{\theta^{\star}} ). 
  \end{equation} 

\end{theo}

An interesting question would be interesting to find a way to exploit this first result to prove at least the exponential decay in all directions. As for the general correlation lenght expression (which is not isotropic), it looks to be a more challenging task and is not adressed here.

\section{Disorder insertion}\label{sec:disorder_insertion}

We are interested in the space-time spin representation of the quantum Ising model on the primal domain $\primaldomain$ with $+$ boundary condition.
The underlying Poisson point process is of parameter $\tau$ on primal lines (death points) and the coupling constant is denoted by $\theta$.
The latter parameter is also the underlying Poisson point process for bridges in the FK representation, but not needed here for the spin representation.

Its dual model is the quantum Ising model defined on the dual domain $\dualdomain$ with free boundary condition whose underlying Poisson point process is of parameter $\tau^\star$ on dual lines and whose coupling constant is given by $\theta^\star$.
The relation between these parameters $(\tau^\star, \theta^\star) = (\frac\theta2, 2\tau)$ is given in~\eqref{eq:KW_duality}. 
Recall that the critical point is also the self-dual point, and the choice of parameters we make will be $\tau^\star = \tau = \frac{1}{4\delta}$ and $\theta^\star = \theta = \frac{1}{2\delta}$.
Only the ratio $\frac\theta\tau$ determines for the regime of the model, but this particular choice allows us to have some nice isotropic property as explained in~\cite{Li-QI-SLE} and will be discussed later in Section~\ref{sec:Dotsenko_equation} and~\ref{sec:correlators}.

\subsection{Definition}

% Contrary to the classical literature of the planar Ising model~\cite{}, we prefer keeping the conventions of [], defining the spins on the dual graph and inserting disorders in the primal lines.
Below we give the formal definition of the disorder operator which is valid for any parameters.
In particular we do \emph{not} assume the self-dual condition~\eqref{eq:KW_duality}, which will not be used when working with the model outside of criticality.

% \comment{R 19/07 Please make the changes here, this is not the dual lines, this is the quantum Ising model with spins assigned to the dual graphs. It removes many stars I feel, no need to say every time it is dual to stg.}

% \comment{JH:
% If we define the model on the dual graph, the spins live on the dual lines (which are faces of the primal graph).
% As such, the parameters have a star on the top.
% Personally I really don't like to define the (primal) model on the dual graph keeping the primal notations (without stars).
% It is not notationally consistant with all the other existing work on quantum models.

% I initially wrote this for the primal version (without stars) then realized that you wanted me to work on the dual graph so a changed everything in this way. It should be correct I think.
% }

As in the discrete setup, the disorder insertion has two possible equivalent definitions. It can be seen as the ratio of partition functions between a modified model and the original one, or as an expectation of a random variables defined along disorder lines in the primal graph.
More precisely, 

\begin{defn}
  \label{def:path_pairing}
  For pairwise disjoint vertices $(v_i)_{1 \leq i \leq 2n}$ in the dual domain $\dualdomain$, we define the disorder operator by
  \begin{equation} \label{eq:disorder}
    \primalmeas [ \mu_{v_1} \dots \mu_{v_{2n}} ]
    := \frac{ Z(\tau, \tilde{\theta} )}{Z( \tau, \theta )}
    = \primalmeas
    \left[ \exp \big( -\theta \cdot \Lebd( \mathbbm{1}_{v \in L_1 \cup \cdots \cup L_n} \eps_v ) \big) \right], 
  \end{equation}
  where $(L_i)_{1 \leq i \leq n}$ is a path collection pairing vertices in $(v_i)_{1 \leq i \leq 2n}$.
\end{defn}
% Note that all the paths should be oriented from bottom to top.
The parameter $\tilde{\theta}$ here is defined to be $-\theta$ along all $(L_i)$ and $\theta$ elsewhere. By reversing along the lines $(L_i)$ the coupling constants, we make the model anti-ferromagnetic there. 
The partition functions $Z(\tau, \theta)$ and $Z(\tau, \tilde{\theta})$ are defined in~\eqref{eq:spin_pf}.
We need to check that the above definition does not depend either on (a) how the vertices $v_i$ are paired or (b) the path taken between two vertices.

Indeed, it is enough to check the second condition since by deforming paths, the first one is a special case of the second.
More precisely, we deform a path so that it goes through another disorder vertex, cut the path and glue in a different manner from that vertex.
See Figure~\ref{fig:path_deformation} for an illustration.
The definition~\eqref{eq:disorder} can also be generalized to an odd number or disorders, but since there is no such a pairing of disorders, the quantitiy is zero.
If two disorders are located at the same vertex, we formally pair them together and say they cancel each other (i.e.  $\mu_v \mu_v =1$ almost surely).

\begin{figure}[htb] \centering
  \includegraphics[scale=0.75, page=1]{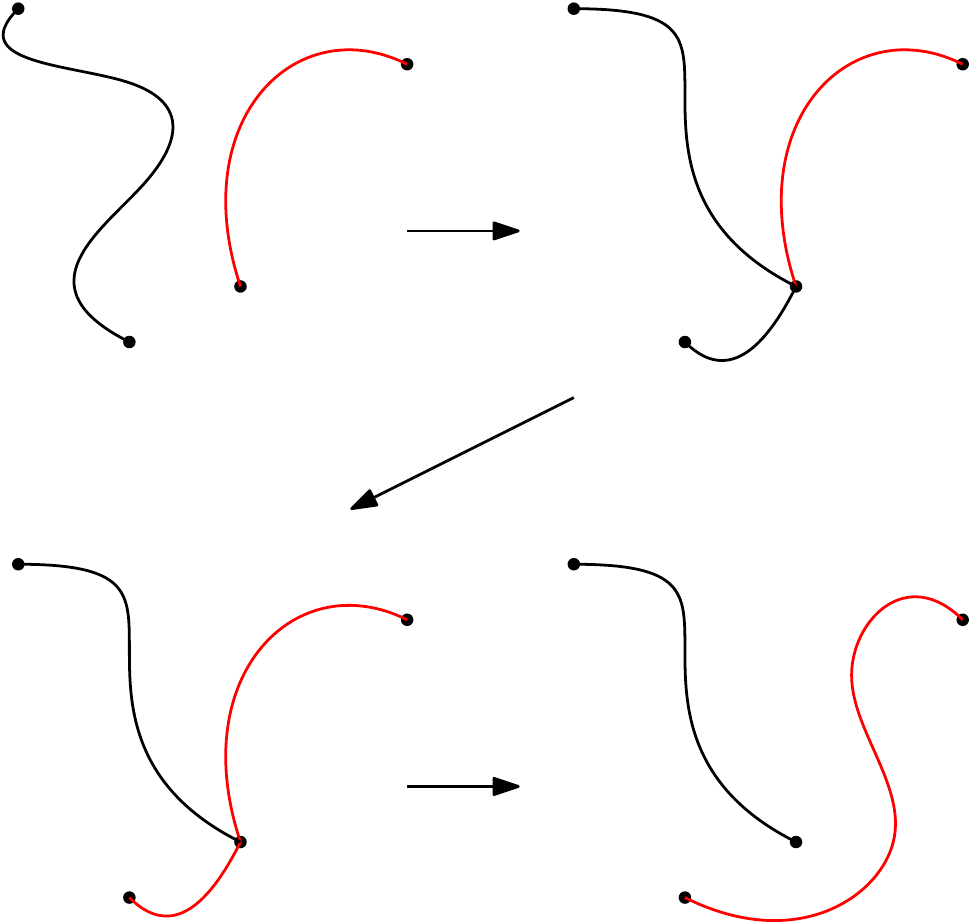}
  \caption{Path deformation showing that the way the disorder lines pair vertices up does not matter.}
  \label{fig:path_deformation}
\end{figure}

Now, we show that the definition~\eqref{eq:disorder} is path-independent.
It is enough to show this for a disorder line connecting $v_1$ to $v_2$.
Let $L_1$ and $\tilde{L}_1$ be two such paths.
Wrtie $\theta_1$ and $\tilde{\theta}_1$ as in the definition of the partition function~\eqref{eq:disorder}.
Write $\Lambda$ for the (closed) region delimited by $L_1$ and $\tilde{L}_1$, $L_\Lambda$ for the collection of dual lines in $\Lambda$ (boundary included).
For a spin configuration $\sigma$, write $\sigma_{\Lambda^c}$ for $\sigma$ restricted on $\Lambda^c$ and $\tilde{\sigma}$ for the configuration which coincides with $\sigma$ on $\Lambda^c$ but with reversed sign on $\Lambda$.
Then, 
\begin{align*}
  \frac{Z(\tilde{\theta}_1)}{Z(\theta)}
  & \overset{(1)}{=} Z(\theta)^{-1} \cdot \PPP^\bullet_{\tau} \left[
    \sum_{\sigma \in \Sigma(D)} C(\sigma_{\Lambda^c}) \exp \big( \textstyle \frac{\tilde{\theta}_1}{2} \cdot \Lebd( \mathbbm{1}_{v \in L_\Lambda} \eps_v ) \big)
    \right] \\
  & \overset{(2)}{=} Z(\theta)^{-1} \cdot \PPP^\bullet_{\tau} \left[
    \sum_{\sigma \in \Sigma(D)} C(\tilde{\sigma}_{\Lambda^c}) \exp \big( \textstyle \frac{\theta_1}{2} \cdot \Lebd( \mathbbm{1}_{v \in L_\Lambda} \eps_v ) \big)
    \right] \\
  & \overset{(3)}{=} Z(\theta)^{-1} \cdot \PPP^\bullet_{\tau} \left[
    \sum_{\sigma \in \Sigma(D)} C(\sigma_{\Lambda^c}) \exp \big( \textstyle \frac{\theta_1}{2} \cdot \Lebd( \mathbbm{1}_{v \in L_\Lambda} \eps_v ) \big)
    \right], 
\end{align*}
where $D$ is sampled according to $\PPP^\bullet_{\tau}$.
In the above computation, the first $(1)$ line is the definition and $C(\sigma_{\Lambda^c})$ is some constant depending only on the spin outside of $\Lambda$.
In the second line $(2)$, we flip all the spins inside $\Lambda$ which is equivalent to switching the coupling constant from $\tilde{\theta}_1$ to $\theta_1$.
It is possible to do so due to the following equality in distribution for any finite set $X \subseteq \primaldomain$, 
\begin{equation*}
  \calL(D \mid X \subseteq D) = \calL(D) \cup X.
\end{equation*}
Finally, in the last line $(3)$, we use the bijection $\sigma \mapsto \tilde{\sigma}$.

Using the duality between the space-time spin representation and the random-parity representation, one can see the disorders as a dual spin configuration, for the dual model with dual parameters, e.g. 
\begin{equation} \label{eq:disorder_duality}
  \primalmeas [ \mu_{v_1} \dots \mu_{v_{n}} ]
  = \dualmeas [ \sigma_{v_1} \dots \sigma_{v_{n}} ], 
\end{equation}
where $(\tau^\star, \theta^\star) = (\frac{\theta}{2}, 2 \tau)$ is given in~\eqref{eq:KW_duality}.
Note again that above we defined the disorder operator when $n$ is even.
When the number of disorders is odd, the LHS of~\eqref{eq:disorder_duality} is zero because the disorders do not pair up; the RHS is also zero because under the free boundary condition, the measure is invariant under the globlal spin-flip.

We can give another interpretation of $Z(\tau, \tilde{\theta})$ in~\eqref{eq:disorder} using the quantum Ising model defined on some double cover.
This viewpoint will turn out to be very useful later for the definition of mixed correlators~\eqref{eq:order-disorder}.

\medskip

Let  $\calD_\delta = \primaldomain$, $\dualdomain$ or $\midedgedomain$.
A double cover of $\calD_\delta$  is formally defined as the disjoint union of two identical copies (also called sheets) $\calD_\delta \sqcup \calD_\delta$ with the following additional data telling about its branching structure:
\begin{itemize}
  \item $\textbf{v} = (v_1, \dots, v_n)$ where $v_1, \dots, v_n \in \dualdomain$ if $\calD_\delta = \primaldomain$ or $\midedgedomain$;
  \item $\textbf{u} = (u_1, \dots, u_m)$ where $u_1, \dots, u_m \in \primaldomain$ if $\calD_\delta = \dualdomain$ or $\midedgedomain$.
\end{itemize}
We denote this double cover by $[\primaldomain, \textbf{v}]$, $[\dualdomain, \textbf{u}]$ and $[\midedgedomain, \textbf{v}, \textbf{u}]$ respectively.
% A double cover of $\domain$  is formally defined as the disjoin union of two identical copies $\domain \sqcup \domain$ whose branching structure is specified by the following additional data: 
% $v_1, \dots, v_n \in \primaldomain$ and $u_1, \dots, u_m \in \dualdomain$.
% We denote this double cover by $\domain^{\vsbraket, \usbraket}$.
Its branching structure is described as follows:
\begin{itemize}
  \item If $\calD_\delta = \primaldomain$, the cuts are described by \emph{any} pairing of $v_i$'s;
  \item if $\calD_\delta = \dualdomain$, the cuts are described by \emph{any} pairing of $u_j$'s;
  \item if $\calD_\delta = \midedgedomain$, the cuts are described by \emph{any} pairing of $v_i$'s and $u_j$'s.
\end{itemize}
See Figure~\ref{fig:double_cover} for an illustration in the case of $[\primaldomain, v_1, v2]$.

\medskip

Given $v_1, \dots, v_n \in \dualdomain$.
\begin{itemize}
  \item Let $\doublecover$ be the double cover of $\primaldomain$ that branches over $v_1, \dots, v_{n} \in \dualdomain$. We consider the quantum Ising model on $\doublecover$.
  \item The Poisson point process $\eta_\tau^\bullet$ is sampled on one copy and is taken to be the same on the other copy.
  \item A spin configuration $\sigma$ on the double cover $\doublecover$ should satisfy the \emph{sign-flip symmetry}: $\sigma_{u} \sigma_{u^\#} = -1$ for any $u \in \doublecover$ (we recall that $u^\#$ is the other fiber in $\doublecover$ of the natural projection of $ u$ in $\primaldomain $).
  \item Given a finite set of death points $D$, denote by $\Sigma(D)$ the set of spin configurations on the double cover $\doublecover$ satisfying the sign-flip symmetry.
  Note that this does not depend on $v_1, \dots, v_n$ but only on $D$.
\end{itemize}
As such, the quantity $Z(\tau, \tilde{\theta})$ can be rewritten as the partition function of the Ising model defined on the double cover $\doublecover$ with a modified Hamiltonian~\footnote{The partition function sums ``twice'' the contribution of each pair of spins (on two sheets)}.
Denoting this quantity by $Z^{\vsbraket}(\tau, \theta)$, we can write, 

\begin{align} \label{eq:pf_double_cover_H}
  Z^{\vsbraket} (\tau, \theta)
  & = \PPP^\bullet_\tau \left[ \sum_{\sigma \in \Sigma(D)} \exp \left(- \textstyle \frac\theta2 \cdot \calH^\vsbraket(\sigma) \right) \right] \\
  \label{eq:pf_double_cover_Leb}
  & = \PPP^\bullet_\tau \left[ \sum_{\sigma \in \Sigma(D)} \exp \left( \textstyle \frac\theta2 \cdot \frac12 \Lebd(\eps_p) \right) \right], 
\end{align}
where the Hamiltonian $\calH^\vsbraket(\sigma)$ is defined to be $-\frac12 \Lebd(\eps_v)$, 
\begin{equation}
  \label{eq:Hamiltonian_dc}
  \calH^\vsbraket(\sigma) = -\frac12 \Lebd(\eps_v) =
  -\frac12 \int_{\doublecover} \eps_v |\dd v|.
\end{equation}
Note that the factor $\frac12$ comes from the fact that when working on the double cover, the contribution coming from each \emph{planar} pair of neighboring spins is counted twice.
In particular, the energy density $\eps_v = \sigma_{v^-} \sigma_{v^+}$ now takes into account the structure of the double cover.
More precisely, if $v$ lies on one of the disorder lines, then $v^-$ and $v^+$ belong to two different sheets, and that is where the sign-flip symmetry makes the difference.

The quantum Ising measure on this double cover is characterized by
\begin{equation}
  \bbE_{\doublecover, \tau, \theta}^+ \left[ F(\sigma) \right]
  = Z^{\vsbraket}(\tau, \theta)^{-1} \cdot
  \PPP^\bullet_\tau \left[ \sum_{\sigma \in \Sigma(D)} F(\sigma) \exp \left( - \textstyle \frac\theta2 \cdot \calH^\vsbraket(\sigma) \right) \right], 
\end{equation}
for any measurable function $F : \cup \Sigma(D) \to \bbR$.

Using this interpretation, for given $v_1, \dots, v_n \in \dualdomain$ and $u_1, \dots, u_m \in \primaldomain$, we can introduce the \emph{mixed correlator}, 
% \begin{align} \label{eq:order-disorder}
%   & \dualmeas [ \mu_{v_1} \dots \mu_{v_{2n}} \sigma_{u_1} \dots \sigma_{u_m} ] \nonumber \\
%   & \qquad := \eps \cdot
%   \dualmeas \left[
%     \exp \big(
%     -\theta^\star \cdot \Lebd( \mathbbm{1}_{v \in L_1 \cup \cdots \cup L_n} \eps_v )
%     \big) \sigma_{u_1} \dots \sigma_{u_m}
%   \right], 
% \end{align}
\begin{align} \label{eq:order-disorder}
  \primalmeas [ \Seq[]{\mu_{v_{#1}}} \Seq[]{ \sigma_{u_{?1}} }[1][m] ] 
  := \primalmeas [ \Seq[]{\mu_{v_{#1}}} ]
  \bbE_{\doublecover}^+ \left[ \Seq[]{ \sigma_{u_{?1}} }[1][m] \right].
\end{align}
In other words, denote $\DB{\Omega_\delta}{m}{n}$ the double cover of $(\dualdomain)^n \times (\primaldomain)^m$ with the following branching structure.
For $\bfv = (v_1, \dots, v_n) \in (\dualdomain)^n$ and $\bfu = (u_1, \dots, u_m) \in (\primaldomain)^m$, 
we fix a path collection connecting $v_i$'s to the boundary and another path collection connecting $u_j$'s to the boundary.
% we fix a path collection connecting $v_i$'s as in Definition~\ref{def:path_pairing};
% and similarly, due to the $+$ boundary condition, we also fix a path collection connecting $u_j$'s to the boundary.
When $u_j$'s wind around $v_i$'s, the cuts are given by the path collection of $v_i$'s; 
similarly, when $v_i$'s wind around $u_j$'s, the cuts are given by the path collection of $u_j$'s.

Using the dualiy between representations, one has, for any integers $n$ and $m$, 
\begin{align} \label{eq:order-disorder_duality}
  \primalmeas[ \mu_{v_1} \dots \mu_{v_{n}} \sigma_{u_1} \dots \sigma_{u_m} ]
  = \dualmeas[ \sigma_{v_1} \dots \sigma_{v_{n}} \mu_{u_1} \dots \mu_{u_m} ].
\end{align}
% This quantity is simply a function on $(\dualdomain)^n \times (\primaldomain)^m$.
In what follows, we write this last quantity as $\braket{ \mu_{v_1} \dots \mu_{v_{n}} \sigma_{u_1} \dots \sigma_{u_m} }$ or, equivalently, $\braket{ \sigma_{v_1} \dots \sigma_{v_{n}} \mu_{u_1} \dots \mu_{u_m} }$, since the object at stake is now canonical and does not ``favor'' spins or disorders (i.e. primal or dual spins).
Note that these quantities are defined on the double cover $\DB{\Omega_\delta}{m}{n}$ with the branching structure described above.
Such a function is called \emph{spinor} because the sign flips when replacing $ u $ (resp. $v$) by $ u^{\#}$ (resp. $ v^{\#}$).

\begin{figure}[htb]
  \centering%
  \begin{tikzpicture}[
    ext/.style = {line width=0pt, draw, fill, circle, inner sep=1pt, outer sep=0pt},
    ]
    % \definecolor{invempha}{HTML}{0070FF}
    % \definecolor{invemphb}{HTML}{0A2891}
    % \definecolor{empha}{HTML}{FF6600}
    % \definecolor{emphb}{HTML}{F5D76E}
    % \definecolor{empha}{HTML}{FFFFFF}

    \node at (0, 0) {\resizebox{7cm}{5cm}{\includegraphics{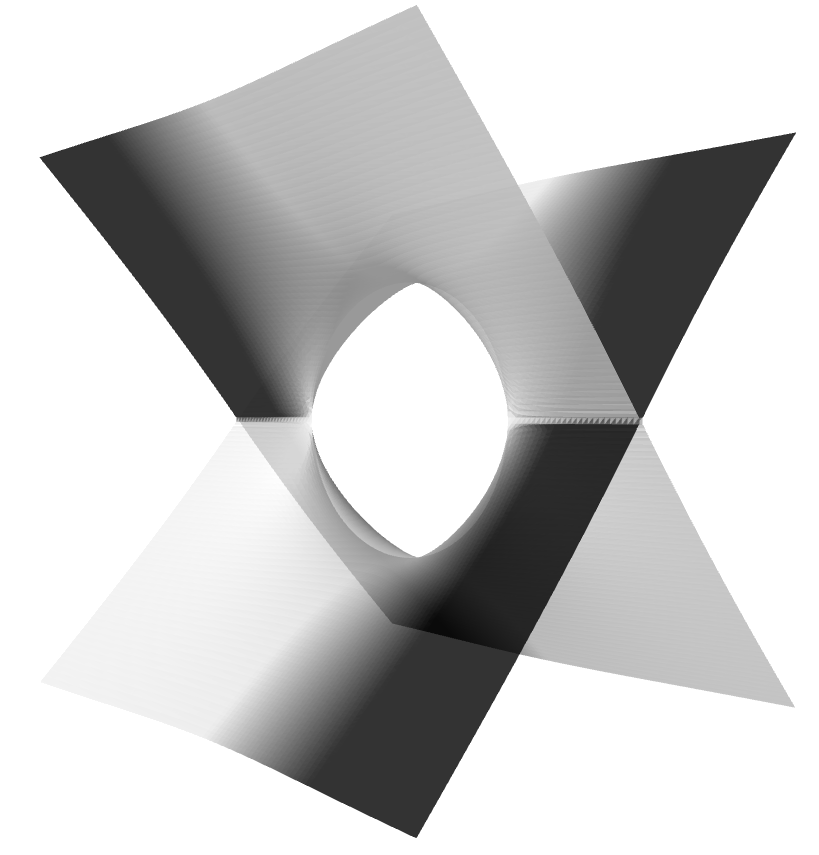}}};
    \node[ext, black, label={[white, xshift=-5pt]above:$v_1$}] (v1) at (-0.85, 0.05) {};
    \node[ext, black, label={[white, xshift=5pt]below:$v_2$}] (v2) at (0.85, 0.05) {};
    \draw[dashed, black, line width=1pt] (v1) -- (v2);
    \node[ext, white, label={[white]left:$u$}] (u) at (-1.5, 1.2) {};
    \node[ext, black, label={[black]left:$u^\#$}] (usharp) at (-1.5, -1.2) {};
    \path[top color=white, bottom color=black] ([xshift=-\pgflinewidth] u) rectangle ([xshift=\pgflinewidth] usharp);
    % \draw[dashed, emphb, top color=red, bottom color=blue] (u) -- (usharp);
  \end{tikzpicture}
  \caption{A graphical representation of ${\primaldomain}^{[v_1, v_2]}$ locally around the branching given by $v_1, v_2 \in \dualdomain$. The vertices $u, u^\# \in \primaldomain$ are two different vertices belonging to the same fiber.}
  \label{fig:double_cover}
\end{figure}

\subsection{Local propagation Equation} \label{sec:Dotsenko_equation}

The local propagation equation for Kadanoff--Ceva fermions (also called Dostenko propagation equation in~\cite{Mercat-Riemann-Ising}) in the classical Ising model~\cite{KC-order-disorder} has an equivalent in the graphical representation of the quantum Ising model.
Instead of having a three-term relation, we have a two-term relation which involves a vertical derivative of the order-disorder operator.
This is due to the degeneration when we squeeze the rectangular lattice to obtain the semi-discrete lattice.

% \begin{prop}
%   Let $\calO$ be a mixed correlator.
%   For two dual vertices $v_1$ and $v_2$ on the same sheet and the same vertical line, as $|v_2 - v_1| \to 0$, we have
%   \begin{equation} \label{eq:primal_three_term}
%     \braket{ \mu_{v_1} \mu_{v_2} \calO }
%     = \braket{ \calO }
%     - \theta \cdot \Lebd( \mathbbm{1}_{v \in [v_1, v_2]} \braket{ \sigma_{v^-} \sigma_{v^+} \calO } )
%     + O(|v_2 - v_1|^2).
%   \end{equation}
% \end{prop}

\begin{prop}
  \label{prop:two_term}
  Let $\calO$ be a mixed correlator.
  For a dual vertex $v$ on the double cover, 
  \begin{equation}
    \label{eq:partial_vert_der}
    \partial_y \braket{ \mu_{v} \calO } = \partial_y \braket{ \mu_{v + \icomp y} \calO }_{| y = 0}
    := \lim_{y \to 0} \frac{1}{y} \big[ \braket{\mu_{v + \icomp y} \calO } - \braket{\mu_v \calO} \big], 
  \end{equation}
  where for all small enough $y$, $v + \icomp y$ needs to be chosen such that it belongs to the same sheet of the double cover as $v$.
  Moreover, this derivative satisfies the following two-term relation, 
  \begin{equation}
    \label{eq:primal_two_term}
    \partial_y \braket{ \mu_v \calO }
    = \sgn(\calO, v) \cdot \theta \braket{\mu_v \sigma_{v^-} \sigma_{v^+} \calO}, 
  \end{equation}
  where $v^-$ and $v^+$ are chosen to belong to the same sheet of the double cover
  and $\sgn(\calO, v)$ takes value in $\{ 0, \pm 1 \}$ depending on the path $L$ that pairs $v$ with another dual vertex in $\calO$, 
  \begin{equation}
    \label{eq:sign_Ov}
    \sgn(\calO, v) = \left\{
      \begin{array}{ll}
        0 & \mbox{ if such a path $L$ does not exist}, \\
        +1 & \mbox{ if $L$ leaves $v$ from above}, \\
        -1 & \mbox{ if $L$ leaves $v$ from below}.
      \end{array}
    \right.
  \end{equation}
\end{prop}

Note that it is indeed important to take the quantity $\sgn(\calO, v)$ into account since the mixed correlators $\braket{ \mu_{v + \icomp y} \calO }$ and $\braket{\mu_v \sigma_{v^-} \sigma_{v^+} \calO}$ are defined on a double cover where cuts depend on how the path $L$ changes.

\begin{proof}
  If the path $L$ defined as in the statement does not exist, then $\calO$ has an even number of disorders and $\braket{\mu_{v+\icomp y} \calO}$ is identically 0 for small enough $y$.
  The right side of~\eqref{eq:primal_two_term} is also identically zero for the same reason.
  
  Assume that $L$ exists.
  Using~\eqref{eq:disorder} and~\eqref{eq:order-disorder}, we get, 
  \begin{align*}
    Z(\tau, \theta) \cdot \partial_y \braket{ \mu_{v + \icomp y} \calO }_{|y=0}
    & = \partial_y \PPP^\bullet_\tau \left[ \sum_{\sigma \in \Sigma(D)} \calO_\sigma
      \exp \left( - \textstyle \frac\theta2 \cdot \calH^{[v + \icomp y, \textbf{v}]}(\sigma) \right) \right]_{|y=0}, \\
    & = \theta \cdot \sgn(\calO, v) \cdot \PPP^\bullet_\tau \left[
      \sum_{\sigma \in \Sigma(D)} \calO_\sigma
      \exp \left( - \textstyle \frac\theta2 \cdot \calH^{[v, \textbf{v}]}(\sigma) \right)
      \sigma_{v^-} \sigma_{v^+}
      \right], 
  \end{align*}
  where we use the dominated convergence to switch the expectation $\PPP^\bullet_\tau$ and the partial derivative $\partial_y$ along with~\eqref{eq:Hamiltonian_derivative}.
  This gives the final result~\eqref{eq:primal_two_term}.
\end{proof}

\begin{prop}
  \label{prop:two_term_dual}
  Let $\calO$ be a mixed correlator. For a primal vertex $u$ on the double cover, 
  \begin{equation}
    \label{eq:partial_vert_der_dual}
    \partial_u \braket{ \sigma_{u} \calO }= \partial_u \braket{ \sigma_{u + \icomp y} \calO }_{| y = 0}
    := \lim_{y \to 0} \frac{1}{y} \big[ \braket{\sigma_{u + \icomp y} \calO } - \braket{\mu_u \calO} \big], 
  \end{equation}
  where for all small enough $y$, $u + \icomp y$ needs to be chosen such that it belongs to the same sheet of the double cover as $u$.
  Moreover, this derivative satisfies the following two-term relation, 
  \begin{equation} \label{eq:dual_two_term}
    \partial_y \braket{\sigma_u \calO}
    = \sgn(\calO, u) \cdot \theta^\star \braket{ \sigma_u \mu_{u^-} \mu_{u^+} \calO }, 
  \end{equation}
  where $\sgn(\calO, u)$ takes value in $\{ \pm 1 \}$ depending on the path $L$ that connects $u$ with the boundary (thus actually no dependency on $\calO$), 
  \begin{equation}
    \label{eq:sign_Ou}
    \sgn(\calO, u) = \left\{
      \begin{array}{ll}
        +1 & \mbox{ if $L$ leaves $u$ from above}, \\
        -1 & \mbox{ if $L$ leaves $u$ from below}.
      \end{array}
    \right.
  \end{equation}

\end{prop}

\begin{proof}
  One simply uses proposition~\ref{prop:two_term} together with the duality relation~\eqref{eq:order-disorder_duality}.
\end{proof}

\begin{lem}
  We use the notations from Proposition~\ref{prop:two_term} and write $\textup{\textbf{v}}$ for the set of all the disorders in $\supp(\calO)$. 
  Then, we have, 
  \begin{equation}
    \label{eq:Hamiltonian_derivative}
    \partial_y \calH^{[v + \icomp y, \textup{\textbf{v}}]}(\sigma)_{|y=0}
    := \lim_{y \to 0} \frac{ \calH^{[v + \icomp y, \textup{\textbf{v}}]}(\sigma) - \calH^{[v, \textup{\textbf{v}}]}(\sigma) }{y}
    = - \sgn(\calO, v) \cdot 2  \sigma_{v^-} \sigma_{v^+}, 
  \end{equation}
  where $v^-$ and $v^+$ are chosen to be on the same sheet of the double cover.
\end{lem}

\begin{proof}
  We show the result in the case that $\sgn(\calO, v) = 1$ and $y > 0$, the computation in  the other cases is similar.
  
  Write $v_1 = v$ and $v_2 = v + \icomp y$ with $y > 0$.
  % Since $\sgn(\calO, v_1) = 1$, i.e. the path $L$ joining $v_1$ and a vertex of $\textbf{v}$ leaves $v_1$ from above.
  Compute the difference of the two following Hamiltonians using~\eqref{eq:Hamiltonian_dc}, 
  \begin{align} \label{eq:H_identity}
    & \calH^{[v_1, \textbf{v}]}(\sigma) - \calH^{[v_2, \textbf{v}]}(\sigma) \nonumber \\
    & \overset{(1)}{=} - \frac12 \int_{\Omega^{[v_1, \textbf{v}]}} \mathbbm{1}_{v \in [v_1, v_2] \cup [v_1^{\#}, v_2^{\#}]} \sigma_{v^-} \sigma_{v^+} |\dd v|
      + \frac12 \int_{\Omega^{[v_2, \textbf{v}]}} \mathbbm{1}_{v \in [v_1, v_2] \cup [v_1^{\#}, v_2^{\#}]} \sigma_{v^-} \sigma_{v^+} |\dd v| \nonumber \\
    & \overset{(2)}{=} - \int_{\Omega^{[v_1, \textbf{v}]}} \mathbbm{1}_{v \in [v_1, v_2]} \sigma_{v^-} \sigma_{v^+} |\dd v|
      + \int_{\Omega^{[v_2, \textbf{v}]}} \mathbbm{1}_{v \in [v_1, v_2]} \sigma_{v^-} \sigma_{v^+} |\dd v| \nonumber \\
    & \overset{(3)}{=} \int_{\Omega^{[v_1, \textbf{v}]}} \mathbbm{1}_{v \in [v_1, v_2]} \sigma_{(v^-)^{\#}} \sigma_{v^+} |\dd v|
      + \int_{\Omega^{[v_2, \textbf{v}]}} \mathbbm{1}_{v \in [v_1, v_2]} \sigma_{v^-} \sigma_{v^+} |\dd v| \nonumber \\
      % & \overset{(4)}{=}\int_{\Omega^{[v_2, \textbf{v}]}} \mathbbm{1}_{v \in [v_1, v_2] \cup [v_1^{\#}, v_2^{\#}]} \sigma_{v^-} \sigma_{v^+} \dd v \nonumber \\
    & \overset{(4)}{=} 2 \int_{\Omega^{[v_2, \textbf{v}]}} \mathbbm{1}_{v \in [v_1, v_2]} \sigma_{v^-} \sigma_{v^+} |\dd v|.
  \end{align}
  In the second equality $(2)$, we use the fact that both sheets of the double cover contributes equally to the integrals.
  In the third equality, $(3)$, we use the sign-flip symmetry for the double cover $\Omega^{[v_1, \textbf{v}]}$.
  In the first term of the last equality $(4)$, on $\Omega^{}$ and for $v \in [v_1, v_2]$, the primal vertices $(v^-)^\#$ and $v^+$ belong to the same sheet, thus can be rewritten with the cut point $v_1$ shifted to $v_2$.

  To conclude the proof, we divide~\eqref{eq:H_identity} by $y$ and take the limit to obtain the desired formula~\eqref{eq:Hamiltonian_derivative}.
\end{proof}

% \comment{R: Please make once the expansion of the exponentials, this is important to me. Also replace mixed order-dirsorder by mixed correlator for example}

\subsection{Construction of observables via correlators}
In the following subsections, we will define two special correlators used in the convergence proofs of the critical model, which are constructed out of mixed correlators of the type $\calO$ .
We also prove they satisfy the so-called $s$-holomorphicity property, which makes discrete correlators regular already in discrete. As for mixed corelator of the form $\calO$, the construction is appropriate on some double cover. We precise here those definitions.

\begin{defn}
  Let $\tilde{F} : \dualdomain \times \primaldomain \to \bbC$ be a function defined on $\DB{\Omega_\delta}{1}{1}$.
  We say that 
  \begin{itemize}
    \item $\tilde{F}$ has the sign-flip property around $v \in \dualdomain$ if $\tilde{F}(v, u^\#) = - \tilde{F}(v, u)$ for all $u \in \DB{\Omega_\delta}{0}{1}$, 
    \item $\tilde{F}$ has the sign-flip property around $u \in \primaldomain$ if $\tilde{F}(v^\#, u) = - \tilde{F}(v, u)$ for all $v \in \DB{\Omega_\delta}{1}{0}$, 
    \item $\tilde{F}$ has the sign-flip property (everywhere) if the above two items are true for all $v \in \dualdomain$ and $u \in \primaldomain$.
  \end{itemize}
  In the above statement, $p^\#$ denotes the \emph{other fiber} of $p$ on the appropriate  double cover.
\end{defn}

As it is the case in the discrete case, we specify the notion of \emph{sign-flip symmetry}, which will appear naturally later. 

\begin{defn} 
  Let $F : \midedgedomain \to \bbC$ be a function defined on the corner semi-discrete lattice such that for each $c \in \midedgedomain$, $F(c)$ depends only on $v_c \in \dualdomain$ and $u_c \in \primaldomain$.
  We may write $F(c) = \tilde{F}(v_c, u_c)$.
  We say that $F$ has the \emph{sign-flip} property at $c \in \midedgedomain$ if
  % $\tilde{F}$ has the sign-flip property
  \begin{itemize}
    \item $v \mapsto \tilde{F}(u_c, v)$ has the sign-flip property around $u_c$, 
    \item $u \mapsto \tilde{F}(u, v_c)$ has the sign-flip property around $v_c$.
  \end{itemize}
\end{defn}
% Another way of seeing this property is to say that, for a fixed $c \in \midedgedomain$, the function $v \mapsto \tilde{F}(u_c, v)$ (\emph{resp.} $u \mapsto \tilde{F}(u, v_c)$) is defined on the double cover that branches at $u_c$ (\emph{resp.} at $v_c$) such that the sign of the function is flipped for two vertices having the same projection but on different sheets. 

\begin{defn}\label{defn:eta_c}
  Let $\eta_c := [\icomp (v_c - u_c)]^{-1/2}$ be the semi-discrete \emph{Dirac spinor}.  
\end{defn}
Note that $\eta_c$ is defined up to the sign, which explains the need of the double cover structure described above.
The sign of $\eta_c$ depends on the choice of $v_c$ and $u_c$.
In other words, $\eta_c$ has the sign-flip property around \emph{all} vertices of $\midedgedomain$.

\begin{defn}\label{defn:complexified_correlator}
  Let $\varpi = (u_1, \dots, u_{n-1}, v_1, \dots, v_{m-1}) \in (\primaldomain)^{n-1} \times (\dualdomain)^{m-1}$ and $\calO = \mu_{v_1} \dots \mu_{v_{m-1}} \sigma_{u_1} \dots \sigma_{u_{n-1}}$.  For a corner $c\in \Upsilon^{\times}_{\varpi}  $ set formally $\chi_c := \mu_{v_c} \sigma_{u_c}$, where $v_c$ (resp. $u_c$) is the neighboring dual (resp. primal) vertex. This allows to define the complexified correlator $\Psi_{\varpi}$ on $\Omega_{\varpi} $ by
  \begin{equation} \label{eq:correlator}
    \Psi_{\varpi}:= c\mapsto \eta_c \braket{ \chi_c \calO }
    = \eta_c \braket{\chi_c \mu_{u_1} \dots \mu_{u_{n-1}} \sigma_{v_1} \dots \sigma_{v_m-1}}.
  \end{equation} 
\end{defn}
The next proposition shows that the complexified correlator is properly defined and satifies the spinor property.

\begin{prop}
  In the context of the definition~\eqref{defn:complexified_correlator}, the complexified correlator $\Psi_{\varpi} $ is well defined on $\Omega_{\varpi}$. Moreover $\Psi_{\varpi} $ has the sign-flip property i.e. for $c, c^{\#}$ corners of $\Omega_{\varpi} $, one has $\Psi_{\varpi}(c)=-\Psi_{\varpi}(c^{\#})$, which means that $\Psi_{\varpi}$ changes its sign when~$c$ makes a turn around one of the $v_p$ or $u_q$.
\end{prop}

\begin{proof}
  As explained above and in~\cite{CHI-mixed},  Remark 2.5, Lemma 2.6, Def 2.7], the correlator $\braket{ \chi_c \calO }$ can be viewed a \emph{single} multivalued function on \emph{one} double cover. The presence of nearby branchings at $c$ and the multiplication by $\eta_c $ (which branches around all points of $\midedgedomain $) kills the branching structure around each points except those of $\varpi $. Finally, the spinor property comes from the spinor property of $\eta_c $ since replacing $\chi_{c} $ by $\chi_{c^\#} $ doesen't change the sign of $\braket{ \chi_c \calO }$ but the one of $\eta_c $.

\end{proof}

\subsection{s-holomorphicity and fermionic correlators} \label{sec:correlators}
We recall now the definition of a semi-discrete $s$-holomorphic function~\cite[Section~3.7]{Li-QI-SLE}, which is a natural generalisation of the isoradial case.
In particular, we show that mixed correlators introduced above give rise to $s$-holomorphic functions.
Contrarily to the FK-loop representation~\cite{Li-QI-SLE}, the approach of disorder insertion described in previous sections provides the $s$-holomorphicity property almost directly (i.e. there is no need to look at combinatorial bijections between local loop configurations).
Below we define the notion of $s$-holomorphicity, which can be viewed in both the corner and diamond lattices, and provide a simple way to switch point of view depending on the context.

\begin{defn} \label{defn:s-hol}
  Let $F : \midedgedomain \to \bbC$ be a function defined on the \emph{corner} semi-discrete lattice.
  It is said to be \emph{corner s-holomorphic} if it satisfies the two following properties.
  \begin{enumerate}
    \item \emph{Parallelism}: for $c \in \midedgedomain$, we have $F(c) \parallel \eta_c \bbR$ where $\eta_c$ is defined in~\eqref{defn:eta_c}
    In other words, 
    \begin{itemize}
      \item $F(c) \in \nu \bbR$ if $u_c$ is on the left and $v_c$ is on the right (western corners), 
      \item $F(c) \in \icomp \nu \bbR$ if $u_c$ is on the right and $v_c$ is on the left (eastern corners), 
    \end{itemize}
    where $\nu = e^{- \icomp \frac{\pi}{4} }$.
    \item \emph{Holomorphicity}: for any corner $c \in \midedgedomain$, we have $\derzbar F(c) = 0$.
  \end{enumerate}
\end{defn}
Note that in the above definition, the semi-discrete derivative is given by
\begin{equation} \label{eq:sd-derivative}
  \derzbar F(c)
  := \frac{1}{2} \left[ \Deltax F(c) - \frac{\dy F(c)}{\icomp} \right]
  = \frac{1}{2} \left[ \frac{F(c^+) - F(c^-)}{2\delta} - \frac{\dy F(c)}{\icomp} \right].
\end{equation}

% Define a new function $f : \midedgedomain \to \bbR$ by
% $f(c) = F(c) / \nu$ if $p_c^+ \in \dualdomain$ and $f(c) = F(c) / (\icomp \nu)$ if $p_c^+ \in \primaldomain$.
% Then, $F$ is s-holomorphic if and only if
% \begin{itemize}
%   \item $\Deltax f(c) = \deltay F(c)$ for $c \in \midedgedomain$ such that $p_c^+ \in \primaldomain$, 
%   \item $\Deltax f(c) = - \deltay F(c)$ for $c \in \midedgedomain$ such that $p_c^+ \in \dualdomain$.
% \end{itemize}

\begin{defn} \label{defn:s-hol2}
  Let $F^{\diamond} : \Omega_{\delta}^{\diamond} \to \bbC$ be a function defined on the \emph{diamond} semi-discrete domain. It is said to be \emph{diamond s-holomorphic} if it satisfies the two following properties.
  \begin{enumerate}
    \item \emph{Projection}: for every $c = c(z^{-}_{c}, z^{+}_{c}) \in \midedgedomain$, we have
    \begin{equation}
      \label{eq:proj_prop}
      \Proj [F^{\diamond}(z^{-}_{c}), \eta_c \bbR] = \Proj [F^{\diamond}(z^{+}_{c}), \eta_c \bbR]
    \end{equation}
    where $\Proj(X, \eta\bbR)$ denotes the usual projection of $X$ on the line $\eta\bbR$, i.e.
    \[
      \Proj[X, \eta \bbR] = \tfrac{1}{2} \big[ X + \eta^2 \overline{X} \big].
    \]
    \item \emph{Holomorphicity}: for all vertex $z$ on the medial domain $\diamonddomain$, we have $\derzbar F^{\diamond}(z) = 0$.
  \end{enumerate}
\end{defn}

In the convergence proofs, we mainly work with \emph{diamond} holomorphic function, since they behave as continous holomorphic function as the lattice mesh-size goes to $0$. In fact, this passage to complexed valued diamond holomorphic function appears crucial to study the scaling limits, since corner holomorphic functions have a prescribed complex argument, thus cannot behave like holomorphic ones. In fact, there is a natural bijection between \emph{corner} and \emph{diamond} s-holomorphic functions on $\midedgedomain$ and those on $\diamonddomain$. The next proposition explicits this very simple link.

\begin{prop} \label{prop:s_hol_equiv}
  Given a \emph{corner} s-holomorphic function $F : \midedgedomain \mapsto \bbC$, one can define $F^{\diamond} : z\in  \diamonddomain \mapsto
  \bbC$ by:
  \[
    F^{\diamond}(z) := F(c^{-}_{z}) + F(c^{+}_{z}).
  \]
  Then, the function $F^{\diamond}$ is \emph{diamond} $s$-holomorphic. \newline 
  Conversely, given an \emph{diamond} s-holomorphic function $F^{\diamond} : \diamonddomain \mapsto \bbC$, one can define $F : c \in \midedgedomain \mapsto \bbC$ by
  \[
    F(c) := \Proj [F^{\diamond}(z^{-}_{c}), \eta_c \bbR] = \Proj [F^{\diamond}(z^{+}_{c}), \eta_c \bbR]. 
  \]
  Then, the function $F$ is \emph{corner} s-holomorphic.
\end{prop}
When the context is clear, we will drop the denomination corner or diamond s-holomorphic functions and just call them $s$-holomorphic functions, passing from one to the other using Proposition~\ref{prop:s_hol_equiv}.
As a consequence of the projections equality above mentioned, one can see as in~\cite[$(2.20)$ and Remark~2.9]{Che20} that $F^{\diamond}$ satisfies the maximum principle, i.e. for any connected set $S$ of $\diamondsuit(G)$ (where neighbours are either at a distance $2\delta $ from each other or belong to same vertical axis), the maximum of $|F^{\diamond}|$ is attained at the boundary of $S$.

We now prove that the correlator $\Psi_{\varpi} $ introduced in~\ref{defn:complexified_correlator} is $s$-holomorphic at the critical isotropic point. We first need to compute the vertical derivative of mixed correlator.

\begin{lem} \label{lem:chi_dy}
  Let $c \in \midedgedomain$ be a corner and $\calO$ a mix correlator.
  Let the partial derivative $\partial_y$ be taken with respect to the vertical variation of $c$, i.e., 
  \begin{align*}
    \partial_y \braket{\chi_c \calO} := \partial_y \braket{\chi_{c + \icomp y} \calO}_{|y=0}.
  \end{align*}
  Then, this derivative is given according to the positions of $v := v(c)$ and $u := u(c)$ where $c = (v(c), u(c))$.
  \begin{itemize}
    \item If $v$ is on the left and $u$ is on the right, then, 
    
    \begin{equation}
      \partial_y \braket{ \chi_c \calO } 
      = \sgn(\calO, v) \cdot \theta \braket{ \chi_{c^-} \calO }
      + \sgn(\calO, u) \cdot \theta^\star \braket{ \chi_{c^+} \calO }.\label{eq:Chi_dy1}
    \end{equation}
    \item If $v$ is on the right and $u$ is on the left, then, 
    
    \begin{equation}
      \partial_y \braket{ \chi_c \calO } 
      = \sgn(\calO, v) \cdot \theta \braket{ \chi_{c^+} \calO }
      + \sgn(\calO, u) \cdot \theta^\star \braket{ \chi_{c^-} \calO }.\label{eq:Chi_dy2}
    \end{equation}
  \end{itemize}
\end{lem}

\begin{proof}
  First note that for any primal vertex $u$ and dual vertex $v$, we have, 
  \begin{align*}
    \partial_y \braket{ \mu_{v + \icomp y} \sigma_{u + \icomp y} \calO }_{| y = 0}
    & = \partial_y \braket{ \mu_{v + \icomp y} \sigma_{u} \calO }_{| y = 0}
      + \partial_y \braket{ \mu_{v} \sigma_{u + \icomp y} \calO }_{| y = 0} \\
    & = \sgn(\calO, v) \cdot \theta \braket{ \mu_v \sigma_{v^-} \sigma_{v^+} \sigma_u \calO }
      + \sgn(\calO, u) \cdot \theta^\star \braket{ \sigma_u \mu_{u^-} \mu_{u^+} \mu_v \calO }.
  \end{align*}
  The above result follows from~\eqref{eq:primal_two_term} and~\eqref{eq:dual_two_term} by respectively choosing $v^-$, $v^+$ on the same sheet and $u^-$, $u^+$ on the same sheet.
  Now the lemma follows from a direct application of the above formula.
  Let us write it for the first item, 
  \begin{align*}
    \partial_y \braket{ \chi_c \calO } 
    & = \partial_y \braket{ \mu_{v} \sigma_{u} \calO } \\
    & = \sgn(\calO, v) \cdot \theta \braket{ \mu_v \sigma_{v^-} \sigma_{v^+} \sigma_u \calO }
      + \sgn(\calO, u) \cdot \theta^\star \braket{ \sigma_u \mu_{u^-} \mu_{u^+} \mu_v \calO } \\
    & = \sgn(\calO, v) \cdot \theta \braket{ \mu_v \sigma_{v^-} \calO }
      + \sgn(\calO, u) \cdot \theta^\star \braket{ \sigma_u \mu_{u^+} \calO } \\
    & = \sgn(\calO, v) \cdot \theta \braket{ \chi_{c^-} \calO }
      + \sgn(\calO, u) \cdot \theta^\star \braket{ \chi_{c^+} \calO }.
  \end{align*}
\end{proof}

\begin{prop} \label{prop:s-hol}
  Let $\Psi_\varpi$ be the correlator as defined in~\eqref{eq:correlator}. If $c $ is not nearby one of the branchings of $\varpi$, one has 
  \begin{itemize}
    \item If $v$ is on the left and $u$ is on the right, then, 
    \begin{equation}
      \label{eq:Psi_dy1}
      \partial_y \Psi_\varpi(c) = \icomp [ \theta^\star \Psi_\varpi(c^+) - \theta \Psi_\varpi(c^-) ].
    \end{equation}
    \item If $v$ is on the right and $u$ is on the left, then, 
    \begin{equation}
      \label{eq:Psi_dy2}
      \partial_y \Psi_\varpi(c) = \icomp [ \theta \Psi_\varpi(c^+) - \theta^\star \Psi_\varpi(c^-) ].
    \end{equation}
  \end{itemize}
  In particular, at critical and isotropical point $\theta = \theta^\star = \frac{1}{2\delta}$, it is corner s-holomorphic on $\midedgedomain$ (still when $c^\pm \not\in \varpi$).
\end{prop}

\begin{proof}
  Let $\Psi_\varpi(c) = \eta_c \braket{\chi_c \calO}$ where $c$ is a corner of $\Omega_\varpi $ and $\calO $ is a mixed order-disorder operator.
  Since $\eta_c$ stays constant when $c$ varies vertically, we only need to compute the first derivative of $\braket{\chi_c \calO}$ which is done in the above Lemma~\ref{lem:chi_dy}.
  Moreover, if $\eta$ is chosen such that its cuts coincide with those of $\braket{\chi_c \calO}$, we have the following relation, 
  \begin{equation}
    \label{eq:eta_v}
    \frac{\eta_{v v^+}}{\eta_{v v^-}}
    = \frac{[ \icomp (v - v^+) ]^{-1/2} }{[ \icomp (v - v^-) ]^{-1/2} }
    = \left( \frac{v - v^+}{v - v^-} \right)^{-1/2}
    = - \sgn(\calO, v) \cdot \icomp, 
  \end{equation}
  since $v - v^+ = e^{\sgn(\calO, v) \icomp \pi} (v - v^-)$.
  Similarly, we also have, 
  \begin{equation}
    \label{eq:eta_u}
    \frac{\eta_{u u^+}}{\eta_{u u^-}}
    = \frac{[ \icomp (u^+ - u) ]^{-1/2} }{[ \icomp (u^- - u) ]^{-1/2} }
    = \left( \frac{u^+ - u}{u^- - u} \right)^{-1/2}
    = - \sgn(\calO, u) \cdot \icomp
  \end{equation}
  since $u^+ - u = e^{\sgn(\calO, u) \icomp \pi} (u^- - u)$.

  Let us complete the computation for the first item.
  Assume that $v$ is on the left and $u$ is on the right of the corner $c$.
  From~\eqref{eq:eta_u}, we have $\eta_{c^+} = -\sgn(\calO, u) \cdot \icomp \eta_c$;
  from~\eqref{eq:eta_v}, we have $\eta_{c^-} = \sgn(\calO, v) \cdot \icomp \eta_c$.
  Combine these two relations with~\eqref{eq:Chi_dy1}, we find~\eqref{eq:Psi_dy1}.
  A similar computation results in~\eqref{eq:Psi_dy2}.
\end{proof}

\begin{prop}
  Let $\Psi=\Psi_\varpi$ be the correlator as defined in~\eqref{eq:correlator}. If $c$ is not nearby one of the branchings $\varpi$, it satisfies the following local relations, 
  \begin{equation}
    \label{eq:Psi_mLaplacian}
    \partial_{yy} \Psi(c) = (\theta^2 + {\theta^\star}^2) \Psi(c) - \theta \theta^\star \big[ \Psi(c^{++}) + \Psi(c^{--}) \big].
  \end{equation}
\end{prop}

\begin{proof}
  Assume that $c = (v, u)$ is a corner such that $v$ is on the left and $u$ is on the right.
  Then, the corners $c^+$ and $c^-$ are of the opposite type.
  We can take the vertical derivative to~\eqref{eq:Psi_dy1}
  \begin{align*}
    \partial_{yy} \Psi(c) & = \icomp [ \theta^\star \partial_y \Psi(c^+) - \theta \partial_y \Psi(c^-) ] \\
                          & = \icomp \big[ \theta^\star \cdot \icomp [ \theta \Psi(c^{++}) - \theta^\star \Psi(c) ]
                            - \theta \cdot \icomp [ \theta \Psi(c) - \theta^\star \Psi(c^{--}) ] \big] \\
                          & = (\theta^2 + {\theta^\star}^2) \Psi(c) - \theta \theta^\star \big[ \Psi(c^{++}) + \Psi(c^{--}) \big]
  \end{align*}
  The computation is the same if $c$ is of the other type.
\end{proof}

The above proposition suggests the following definition for the massive Laplacian operator $\mDelta = \Delta^{(m)}_{(\theta, \theta^\star)}$, 
\begin{equation} \label{eq:mLaplacian}
  \Delta^{(m)}_{(\theta, \theta^\star)} \Psi(c) =  \mDelta \Psi(c) :=  -\Psi(c) + \frac{1}{\theta^2 + {\theta^\star}^2} \big[ \ddy \Psi(c) + \theta \theta^\star \big[ \Psi(c^{++}) + \Psi(c^{--}) \big].
\end{equation}
The equation~\eqref{eq:Psi_mLaplacian} is equivalent to $\mDelta \Psi(c) = 0$.
The massive Laplacian $\mDelta$ can be interpreted as the infinitesimal generator of the semi-discrete Brownian Motion with horizontal killing rate $ 1 - \frac{2\theta \theta^{\star}}{\theta^2 + {\theta^\star}^2} $.
At the critical and isotropical point of the quantum Ising model $\theta = \theta^\star = \frac{1}{2\delta}$, the massive Laplacian operator~\eqref{eq:mLaplacian} gets simplified and one recovers the (normalized) semi-discrete Laplacian operator, 
\begin{equation} \label{eq:normalized_Laplacian}
  \laplacian \Psi(c) := \frac12 \ddy \Psi(c) + \frac{1}{2\delta^2} \big[ \Psi(c^{++}) + \Psi(c^{--}) -2 \Psi(c) \big].
\end{equation}
This shows that the correlator $\Psi$ is harmonic at the critical and isotropical point and massive harmonic otherwise. This massive harmonicity will be used to derive asymptotics of correlations in the infinite volume limit.

\subsection{The energy and spin correlators}

We introduced above a general formalism for fermionic correlators. We focus now on two specific choices of $\calO$ which that are useful to prove the existence of a scaling limit at criticality for energy density and spin-correlations. In several papers following Smirnov's seminal work~\cite{Smirnov-conformal} on the planar critical Ising model, those observables were introduced in terms of their loops expansion, but we prefer using here the Kadanoff--Ceva formalism developped above. The global strategy is to use the fermionic correlator~\ref{eq:order-disorder} as a functional of corners. We recall that given $\varpi := (v_1, \dots, v_{n-1}, u_1, \dots, u_{m-1}) \in (\primaldomain)^{n-1} \times (\dualdomain)^{m-1}$ and $c \in \Omega^{\Upsilon_\delta}_{\varpi}$, we defined $\Psi_{\varpi}(c)= \eta_c  \langle \chi_c \mu_{v_1}\dots\mu_{v_{m-1}}\sigma_{u_1}\dots\sigma_{u_{n-1}}\rangle$ on $\Omega_\varpi $.

\paragraph{The energy correlator}
We fix a corner $a=(v_a u_a) \in \midedgedomain$ and set $\varpi = \{u_a, v_a \} $. We define the graph $\Omega_{\delta,a} := (\Omega_\delta \backslash \{ a\}) \cup \{ a^{\pm} \} $ where the two corner $a^{\pm}$ are \emph{different} but located at $a$, i.e. $a$ is replaced by two corner $a^{\pm}$, and combinatorially, $a^{+}$ is a neighbour to the right part of the graph near $a$ while $a^{-}$ is a neighbour of the left part of the graph near $a$. We are now able to define the energy correlator using $\calO= \mu_{v_a}\sigma_{u_a} $.  The next definition precises this setup.

\begin{defn}
  Fix a corner $a=(v_a u_a) \in \midedgedomain$. We define the energy correlator double-valued at $a$ by setting for $c \neq a^{\pm} \in \Omega_{\delta,a}$
  \begin{equation}\label{defn:energy_correlator}
    \Psi^{\calE}_{\Omega_\delta, a} (c) := \eta_{c} \braket{ \sigma_{u_a}\mu_{v_a} \sigma_{u(c)}\mu_{v(c)}}_{\Omega_\delta} = \eta_{c} \braket{\chi_{c}\chi_{a}}_{\Omega_\delta}, 
  \end{equation}
  with $\Psi^{\calE}_{\Omega_\delta, a}(a^{\pm}) = \pm \eta_a $.
\end{defn}
The next proposition precises its main caracteristics of $\Psi^{\calE}_{\Omega_\delta, a} $.

\begin{prop}
  The energy correlator defined in~\eqref{defn:energy_correlator} is a well-defined, and  $s$-holomorphic everywhere.

\end{prop}
The catchpoint of the above definition is that one shouldn't look at $\Psi^{\calE}_{\Omega_\delta, a}$ as a spinor on a double-cover, but rather see $\Psi^{\calE}_{\Omega_\delta, a}$ as a \emph{planar} observable (i.e., not defined on the double-cover) with a discrete singularity at $a$.
\begin{proof}
  This is a consequence of Lemma 2.9 and Remark 2.13 in~\cite{CHI-mixed}, or~\cite{CIM-universality}, figure 6. The introduction of $\Omega_{\delta,a} $ amounts to introducing a cut $\gamma_a =[u_a v_a] $ in the double cover $\Omega_{[u_a, v_a]} $ which branches around $u_a$ and $v_a$.
  A priori, the above correlator $\eta_c\braket{ \sigma_{u_a}\mu_{v_a} \sigma_{u(c)}\mu_{v(c)}}_{\Omega_\delta}$ is well defined on $\Omega_{[u_a, v_a]} $ that branches around nearby points $u_a $ and $v_a$. The two set $\Omega_{[u_a, v_a]} $ and $\Omega_\delta $ \emph{can be identified} except at the corner $a$. In terms of observables, this "planarity" statement can be easily seen since whenever $c$ winds around \emph{both} branchings at the same time in $\Omega_{[u_a, v_a]} $, it simply returns to the same sheet, i.e. the branchings effets around $v_a$ and $u_a$ simply kill each other. One can thus restrict $\braket{ \sigma_{u_a}\mu_{v_a} \sigma_{u(c)}\mu_{v(c)}}_{\Omega_\delta}$ to one of its sheets to get the identification with $\Omega_{\delta,a} $. Still, in this sheet restriction, one has to prescribe the representant of $a$. The split of $a$ in $a^{\pm}$ and the choice of valuation $\pm \eta_a$ preserves the local holomorphicity relation (which holded in $\Omega_{\delta,a} $), provided we assume that  $a^{\pm} $ interact (in a combinatorial way) respectively with right and the left and the part of the edge $[v_a u_a]$.
\end{proof}

In particular, the associated \emph{diamond} observable $\Psi^{\calE, \diamond}_{\Omega_\delta, a}$ is not $s$-holomorphic at $a$ (since the associated projections only match up to the sign). Still, as mentioned in the previous paragraph, using the corners $a^{\pm}$ instead of $a$, one can preserve the projection property and thus construct $\Psi^{\calE, \diamond}_{\Omega_\delta, a}$ .
We also define the full-plane observable (see appendix~\ref{sec:infinite_volume_corrlators}) $G_{(a)}$  in a combinatorial way, which turns out to have exactly the \emph{same} singularity as $\Psi^{\calE, \diamond}_{\Omega_\delta, a}$. In particular, since $\Psi^{\calE}_{\Omega_\delta, a}$ and $G_{(a)}$ have the same singularity at $a$, their difference $ \Psi^{\calE}_{\Omega_\delta, a} - G_{(a)}$ is now $s$-holomorphic \emph{everywhere} in $\Omega_{\delta} $ (i.e. there is no need to use the graph $\Omega_a $ where $a$ is replaced by $a^{\pm}$).

Assuming for example that $a$ is a western corner, the definition of the fermionic correlator, the identity $\mu_{v_a} \mu_{v_a} =1$ and the computation $G_{(a)}(a+\delta)= \eta_{a+\delta} \frac{\sqrt{2}}{2}= \eta_{a+\delta}\braket{\sigma_{u_a} \sigma_{u_{a+2\delta}}}_{\bbC_\delta}^+$ directly yields that
\begin{equation}
  [\Psi^{\calE, \diamond}_{\Omega_\delta, a} - G_{(a)}][v_a] = \eta_{a+\delta}[ \braket{\sigma_{u_a} \sigma_{u_{a+2\delta}}}_{\primaldomain}^+-\braket{\sigma_{u_a} \sigma_{u_{a+2\delta}}}_{\bbC_\delta}^+], 
\end{equation}
is the normalized energy density at $u_a$ up to an explicit modular factor.
One can deduce the same way that 
\begin{equation}
  [\Psi^{\calE, \diamond}_{\Omega_\delta, a} - G_{(a)}][u_a] = \eta_{a-\delta}[ \braket{\mu_{v_a} \mu_{v_{a-2\delta}}}_{\primaldomain}^+-\braket{\mu_{v_a} \mu_{v_{a-2\delta}}}_{\bbC_\delta}^+], 
\end{equation}
\begin{rem}

  One could see alternatively $G_{(a)}$ as the limit of functions $\Psi^{\calE}_{\Omega_\delta, a}$ when the domains $\Omega_{\delta} \uparrow \bbC_{\delta} $.
\end{rem}

\paragraph{The spin correlator}

Fix distinct $u_1,u_2,\ldots, u _n \in \primaldomain$, $v\in \dualdomain $ a dual vertex neighbouring $u_1$. Finally denote $ \varpi = \{ v,u_2,\ldots,u_n \} $. We now define the spin correlator using the mixed correlator $\calO=\mu_v\sigma_{u_2}\ldots u_n $ using only one dual vertex.

\begin{defn}
  Given $u_1,u_2,\ldots, u _n,v $ as above, one defines the spin correlator $\Psi^{\calS}_{\Omega_\varpi}$ for $c\in \Omega_\varpi $ by 
  \begin{equation}\label{defn:spin_correlator}
    \Psi^{\calS}_{\Omega_\varpi} := c\mapsto \eta_{c} \braket{\sigma_{u_c}\mu_{v_c} \mu_{v}  \sigma_{u_2} \ldots \sigma_{u_n}}_{\Omega_{\delta}}.
  \end{equation}
\end{defn}

The next proposition precises the properties of $\Psi^{\calS}_{\Omega_\varpi}$.

\begin{prop}
  The correlator defined in~\eqref{defn:spin_correlator} is a spinor, $s$-holomorphic everywhere on $\Omega_\varpi $ and only branches over the points of $\varpi $. Moreover, one has the identity $\Psi^{\calS}_{\Omega_\delta, \varpi}(c(u_1 v))= \pm \bbE^{+}_{\Omega_{\delta}} [ \sigma_{u_1}\ldots \sigma_{u_n} ] $, depending on the fiber $c(u_1 v)$ taken for evaluation.

\end{prop}

\begin{proof}
  This is a simple consequence of the consturctions of $\calO $ and $\eta$. For the last claim, one sees that $\braket{\sigma_{u_1}\mu_{v} \mu_{v}  \sigma_{u_2} \ldots \sigma_{u_n}}_{\Omega_{\delta}} =  \braket{\sigma_{u_1} \sigma_{u_2} \ldots \sigma_{u_n}}_{\Omega_{\delta}}=  \bbE^{+}_{\Omega_{\delta}} [ \sigma_{u_1}\ldots \sigma_{u_n} ]$.
\end{proof}

\section{Riemann-type boundary  value problems} \label{sec:bvp}

As already mentioned, Smirnov introduced in~\cite{Smirnov-conformal} (via their loop representation) complexified fermionic observables in the isoradial context, which happen to be a very powerfull tool to study the scaling limits of the Ising model defined grids with thinner and thinner mesh size. The most remarkable feature of those scaling limits is their conformal invariance/covariance i.e. sclaling limits compose nicely under conformal maps. In the next paragraph, we recall the discrete integration procedure for the imaginary part of the square a fermionic observable, which is the key tool to pass from discrete to continuum and prove conformal invariance.

\subsection{Discrete Integration procedure}

Let $\obsone, \obstwo : \midedgedomain  \to \bbC$ be corner $s$-holomorphic functions.
The next definition precises how to construct a bilinear form that can be interpreted as the imaginary part of the primitive of their product.
\begin{defn}\label{defn:primitivation_procedure}
  Let $\prim = H [\obsone, \obstwo ] := \braket{\obsone, \obstwo} :  \Omega ^{\Lambda}_{\delta} \to \bbR$ be defined  by the following rules.
  The integration with respect to $\dd z$ is the path integral.
  \begin{enumerate}
    \item If $v_1$ and $v_2$ are both in $\dualdomain$ and such that $\myr v_1 = \myr v_2$, define
    \begin{equation} \label{eq:bi_primal}
      \prim(v_2) - \prim(v_1) := \int_{v_1}^{v_2}
      \big[ \obsone(c_z^-) \overline{\obstwo} (c_z^+)
      -  \overline{\obstwo} (c_z^-) \obsone (c_z^+) \big] \dd z.
    \end{equation}
    \item If $u_1$ and $u_2$ are both in $\primaldomain$ and such that $\myr u_1 = \myr u_2$, define
    \begin{equation} \label{eq:bi_dual}
      \prim(u_2) - \prim(u_1) := - \int_{u_1}^{u_2}
      \big[ \obsone(c_z^-) \overline{\obstwo} (c_z^+)
      -  \overline{\obstwo} (c_z^-) \obsone (c_z^+) \big] \dd z.
      % \prim(u') - \prim(u) := - 2 \cdot \myi \int_u^{u'} \obsone(c_\nu^-) \overline{\obs (c_\nu^+)} \dd \nu.
    \end{equation}
    \item If $v$ and $u$ are horizontal neighbours in $\medialdomain$ with $v \in \dualdomain$ and $u \in \primaldomain$, define
    \begin{equation} \label{eq:bi_edge}
      \prim(v) - \prim(u) := - \delta \obsone(c_{vu}) \overline{ \obstwo(c_{vu}) }.
    \end{equation}
  \end{enumerate}
\end{defn}

Note that ~\eqref{eq:bi_primal} and~\eqref{eq:bi_dual} are \emph{path integrations}.
It can be easily checked that $\prim$ is locally (and hence globally) well defined up to \emph{one} additive constant on $\medialdomain $, by showing that the sum of differences~\eqref{eq:bi_primal}, \eqref{eq:bi_dual} and~\eqref{eq:bi_edge} along any closed contour is zero.
This check is easy on elementary rectangles and extends to any contour.
Below we show the computation on elementary rectangles.

% The above function $\prim$ is well defined up to an additive constant for the following reason.
Consider two corners $c_1$ and $c_2$ such that $\myr c_1 = \myr c_2$.
For $i=1, 2$, write $c_i = (v_i, u_i)$ and assume that $v_i$ is on the left and $u_i$ is on the right.
Then, 
\begin{align*}
  & [\prim(v_2) - \prim(v_1)] - [\prim(u_2) - \prim(u_1)] \\
  & \qquad = \int_{c_1}^{c_2}
    \big[ \obsone(c^-) \overline{\obstwo} (c)
    -  \overline{\obstwo} (c^-) \obsone (c) \big] \dd c \\
  & \qquad \qquad + \int_{c_1}^{c_2}
    \big[ \obsone(c) \overline{\obstwo} (c^+)
    -  \overline{\obstwo} (c) \obsone (c^+) \big] \dd c \\
  & \qquad = \int_{c_1}^{c_2} \obsone(c)  \big[ \overline{\obstwo}(c^+) - \overline{\obstwo}(c^-) \big]
    - \overline{\obstwo}(c) \big[ \obsone(c^+) - \obsone(c^-) \big] \dd c \\
  & \qquad = \icomp \delta \int_{c_1}^{c_2} \obsone(c) \cdot \partial_y \overline{\obstwo}(c)
    + \overline{\obstwo}(c)  \cdot \partial_y \obsone(c) \dd c \\
  & \qquad = \icomp \delta \int_{c_1}^{c_2} \partial_y \big[ \obsone(c) \overline{\obstwo}(c) \big] \dd c
    = -\icomp \int_{c_1}^{c_2} \partial_y \big[ \prim(v_c) - \prim(u_c) \big] \dd c.
\end{align*}

One can also easily check that $\prim$ is a bilinear form on the Hermitian space of corner $s$-holomorphic functions.
Given a complex-valued function $\obs : \midedgedomain \to \bbC$ defined at corners (requiring $\obs$ to be a spinor when defined on the double cover) one can define a \emph{primitive} of its square by simply posing $\prim  = \prim [\obs]:= \braket{ \obs, \obs}$.
This recovers the definition of the primitive $\prim$ in~\cite[Sec.~4.3]{Li-QI-SLE} where the integrals are one-dimensional integrals instead of path integrals. Recall that for $z \in \medialdomain$, we defined
\begin{equation} \label{eq:medial_observable}
  F_{\delta}^{\diamond}(z) := \obs(c_{z}^+) + \obs(c_{z}^-).
\end{equation}
Proposition~\ref{prop:s_hol_equiv} ensures that $F_{\delta}^{\diamond}$ is s-holomorphic on $\diamonddomain$.
The following proposition states that the primitive $\prim$ has a nice interpretation in terms of $\obsmedial$.
We omit the proof here since they can be found in~\cite[Section~4.3]{Li-QI-SLE}. 

\begin{prop}
  The bilinear form $\prim = \prim [\obs] $ has the following properties :
  \begin{enumerate}
    \item  $\prim $ is well defined simultaneously on $\primaldomain$ and $ \dualdomain$ up to \emph{one} additive constant.
    \item 
    One can interpret $\prim$ as the primitive of the imaginary part of $(\obsmedial)^2 $ i.e. 
    \begin{enumerate}
      \item Let $z, z' \in \diamonddomain$ such that $\myr z = \myr z'$ and $[zz'] \subset \diamonddomain$.
      Then we have
      \begin{equation} \label{eq:H_axis}
        \prim(z') - \prim(z) = \myi \int_z^{z'} \icomp \cdot (\obsmedial (\nu))^2 \dd \nu.
      \end{equation}
      \item Let $z \in \diamonddomain$ such that $z^-, z^+ \in \diamonddomain$.
      Then, 
      \begin{equation} \label{eq:H_neighbors}
        \prim(z^+) - \prim(z^-) = \myi [ \obsmedial(z)^2 (z^+ - z^-)].
      \end{equation}
    \end{enumerate}
    \item $H^{\delta} $ is subharmonic on primal lines and superharmonic on dual lines, i.e., 
    \[
      \laplacian \prim (u) \geq 0 \quad \mbox{ and } \quad \laplacian \prim (v) \leq 0
    \]
    for all $u \in \primaldomain$ and $v \in \dualdomain$, provided $\obsmedial $ doesen't branch over $u$ or $v$.
  \end{enumerate}
\end{prop}
One can see as in Remark 3.8 of~\cite{CHI-spin} that the above sub/super harmonicity property \emph{is preserved near one of the branchings} when the corner holomorphic observable \emph{vanished} at a corner nearby the branching (while it fails if the observable dosen't vanish near the branching.) One can also see (this is a consequence of proposition 3.15 in~\cite{Li-QI-SLE} and the sub/super harmonicity property) that the discrete primitive of the square $H$ satisfies the maximum and minimum principle i.e. in a connected set $S$ of $\Lambda(G)$ where the observable has no branchings (or vanishes near its branchings as described above), the maximum and the minimum of $H$ is attained at the boundary of $S$. 
When integrating observables definded on a double cover with a finite number of branchings, this integration procedure kills the branching structure and provides a well-defined function \emph{on the planar domain}.
We now present a trick already mentioned in~\cite{CIM-universality}, lemma 3.14 that takes advantages of the previous integration procedure to extact values of observables near their branchings.

\begin{prop} \label{prop:integration-trick}
  Let $F^{\diamond}_{v} $ and $ G_{u}^{\diamond} $ two diamond $s$-holomorphic spinors that branch respectively over $v \in \dualdomain$ and $u \in \primaldomain $, separated by a corner $c=c(v, u)$.
  Let $\calC $ be a \emph{planar} closed contour formed by lines of $\medialdomain $ and medial edges, surrounding $v$ and $u$ but no other branching of $F^{\diamond}_{v} $ and $ G_{u}^{\diamond} $.
  We have then
  \begin{equation}
    \myi \big[ \oint_{\calC} F^{\diamond}_{v}(z) G_{u}^{\diamond} (z) \dd z \big]
    = \pm 2 \delta F_{v}(c) \overline{G_{u}(c)}, 
  \end{equation}
  where $\pm$ only depends on the orientation of $\calC $ and the sheet of the lift of $c$.
\end{prop}

\begin{proof}
  The $s$-holomorphic functions $F^{\diamond}_{v}$ and $ G_{u}^{\diamond}$ are locally defined on two \emph{slightly different} double-covers, branching respectively around $v$ and $u$.
  One can naturally identify those two double-covers except at the two lifts of $c$. Consider the planar contour $\calC $ defined as $p_1 \sim p_2 \sim \ldots \sim p_r \sim p_{r+1}= p_1 $ defined as in Figure~\ref{fig:FG_contour_integral}.
  In our notations, either $p_i \sim p_{i+1} $ are two neighbours  $\medialdomain $ or $\Re[p_i] = \Re[p_{i+1}] $.

  The integration procedure introduced in Definition~\ref{defn:primitivation_procedure} for the product of $F^{\diamond}_{v}$ and  $G_{u}^{\diamond}$ is local and hence well defined  away from the branchings.
  In particular, local increments of $H[F^{\diamond}_{v}, G_{u}^{\diamond}] $ are well defined.
  Modifying step by step the contour (as done with pink arrows), it is possible to replace $\calC $ by any interior elementary rectangle surrounding $c$ and passing by $v$ and $u$.
  Splitting this last rectangle into two elementary rectangles containing the segment $[vu]$, the only remaing contribution comes from the two horizontal segments (drawn in red and blue), where the non-identification of double-covers at $c$ implies the increments \emph{do not} cancel each other. Fix a lift $c_0$ of $c$ to the double cover branching around $v$. Up to a global sign (depending on the orientation of the contour and the choice of the lift $c_0$), the total increment of the two horizontal rectangles is
  \begin{equation}
    \delta F_{v}(c_0)\overline{G_{u}(c_0)}-\delta F_{v}(c_0)\overline{G_{u}(c^\#_0)}.
  \end{equation}
  The mismatch obtained is $\pm2 \delta F_{v}(c_0)\overline{G_{u}(c_0)}$ using the spinor property.
\end{proof}

\begin{figure}[htb] \centering
  \includegraphics[scale=1.2, page=1]{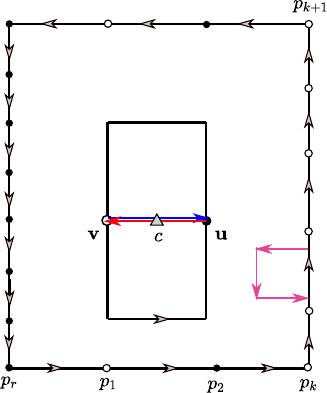}
  \caption{Illustration of the proof of Proposition~\ref{prop:integration-trick}: One first modifies a contour surrouding $c, \textbf{v}, \textbf{u}$ without changing the integral around $\calC$ (local moves represented by pink arrows) to get a standard rectangle passing by $\textbf{v}, \textbf{u}$ and surrounding $c$.
    We then split the rectangle in two different elementary rectangle, containing the segment $[uv]$. Due to the non-identification of double-covers at $c$, the increments in red and blue are \emph{the same} and sum up to $\pm 2\delta F_{v}(c)\overline{G_{u}(c)}$.}
  \label{fig:FG_contour_integral}
\end{figure}

\subsection{Caracterisation of observables via discrete boundary values problems}

For $\calR_{\delta}$ a discrete simply connected region and $z_{ext} \in \partial \calR_{\delta}$, one defines  $v_{out}(z_{ext})$ to be the tangent vector at $z$ oriented towards the exterior of $\calR_{\delta}$.
It is a classical fact~\cite{Li-QI-SLE, CS-universality} that the $s$-holomorphicity condition~\eqref{defn:s-hol} applied at a boundary vertex $z_{ext}$ implies that $\myi[F^{\diamond}(z_{ext})v_{out}^{1/2}(z_{ext})]=0$.
The latter fact reads as the fact that $H_\delta = H[F_{\delta}] $ is constant along the boundary of $ \calR_{\delta}$.
Since the primitivation procedure~\ref{defn:primitivation_procedure} is defined up to an additive constant, one can always assume that $H_\delta$ vanishes along the boundary. We write now the energy and spin observables~\eqref{defn:energy_correlator}~\eqref{defn:spin_correlator} as solutions to discrete Riemann-type boundary value problems.
We keep the notations introduced earlier, with $a \in \Omega_{\delta}^{\Upsilon} $, $ \varpi = \{ v,u_2,\ldots,u_n \} $, and $c_0 $ one of the lifts of $c(v u_1)$ in $\Omega_{\varpi_{ \delta}} $. We consider the following boundary value problems $\energyproblemdiscrete $ and $\spinproblemdiscrete $, whose respective solutions are denoted $F_{\calE} = \energysolutiondiscrete$ and $F_{\calS} = \spinsolutiondiscrete$
\begin{equation}
  \energyproblemdiscrete  : \left\{
    \begin{array}{lll}
      F_{\calE} \text{ is double-valued }  \pm 1 \text{ at } a^{\pm}  \text{ as in }~\eqref{defn:energy_correlator}. \\
      F_{\calE} \text{ is } s\text{-holomorphic everywhere in } \Omega_{a,\delta}.\\
      
      H_\delta[F_{\calE}] \text{ vanishes on } \partial \Omega_\delta.
    \end{array}
  \right.
\end{equation}

\begin{equation}
  \spinproblemdiscrete : \left\{
    \begin{array}{llll}
      F_{\calS}  \text{ is a spinor on } \Omega_{\varpi, \delta}  \text { branching around } v,u_2,\ldots,u_n. 	\\    	
      F_{\calS}(c_0)=1. \\
      F_{\calS} \text{ is } s \text{-holomorphic everwhere in } \Omega_{\varpi} \text{ (e.g. even near the branchings). } \\
      H_\delta[F_{\calS}] \text{ vanishes on } \partial \Omega_\delta.        
    \end{array}
  \right.
\end{equation}

\begin{lem}
  The observables $\energysolutiondiscrete$ and  $  {\small \bbE^{+}_{\Omega_{\delta}} [ \sigma_{u_1}\ldots \sigma_{u_n} ]^{-1} } \spinsolutiondiscrete$ are respectively the unique solutions to $\energyproblemdiscrete$ and $ \spinproblemdiscrete$.

\end{lem}

\begin{proof}
  The results of the previous sections imply that $\energysolutiondiscrete$ and $\frac{\spinsolutiondiscrete}{\bbE^{+}_{\Omega_{\delta}} [ \sigma_{u_1}\ldots \sigma_{u_n} ]}$ are indeed solutions the corresponding boundary value problems. To prove their uniqueness, consider $\hat{F}_{\calE , \calS} $ differences of two solutions of the corresponding boundary value problems. For the energy energy problem, the double-valuations at $a$ cancel each-other, hence the observable $\hat{F}_{\calE} $ is $s$-holomorphic everywhere thus $ H_{\delta}[\hat{F}_{\calE}] $ is well defined on $\Omega_{\delta} $ and satisfies the maximum and the minimum principle . Since $ H_{\delta}$ vanishes at the boundary, it implies that both $ H_{\delta}$ and $\hat{F}_{\calE} $ identically vanish. For the spin boundary value problem, on sees that $\hat{F}_{ \calS}$ vanishes near the branching $v$, preserving the subharmonicity of $ H_{\delta}$ at $v$. Thus $ H_{\delta}$ is subharmonic everywhere and vanishes on the boundary thus vanishes everywhere. Assume that $\hat{F}_{ \calS}$ is not trivial. In particular $H $ is not constant, and its maximum has to be attained at the boundary. This implies as in~\cite{CHI-mixed}, Proposition 2.16, equation (4.7) that the sign $\overline{\eta_c} \hat{F}_{ \calS} $ has to be preserved when traveling along the boundary. This is incompatible with the fact that $\overline{\eta_c }\hat{F}_{ \calS}$ branches around any of the other branchings. One can then conclude that $\hat{F}_{ \calS}$ is trivial.
\end{proof}

\subsection{Continuous boundary values problem}

The previous paragraph caracterizes discrete observables by discrete Riemann-type boundary value problems. The overall strategy to prove the convergence Ising related quantities start by proving the convergence of discrete observables to their natural continuous counterpart, which rise from naturally associated continuous boundary problems. For completeness, we recall simple facts on solutions to those continuous boundary value problems, which are summurized in a very complete way in~\cite[Sec.~3 and~5]{CHI-mixed}.

Given a simply connected domain $\Omega $, $a $ and  $ \varpi = \{ u_1,u_2,\ldots,u_n \} $, consider $\calP^{\calE}_{\Omega , a} $ and $\calP^{\calS}_{\Omega, \varpi} $ the following Riemann-Hilbert boundary value problems, whoses respective solutions are denoted $f_{\calE} = f^{\calE}_{\Omega, a}$ and $f_{\calS}= f^{\calS}_{\Omega, \varpi}$.

\begin{equation}
  \calP^{\calE}_{\Omega , a} : \left\{
    \begin{array}{lll}

      f_{\calE} \text{ is holomorphic everywhere in }  \Omega \backslash \{ a \}. \\
      f_{\calE}- \frac{2}{\pi(z-a)} \text{ is bounded near } a. \\

      H_{\calE}:= \myi \int^{z} f_{\calE}^{2}(\nu) \dd \nu  \text{ vanishes on } \partial\Omega.
    \end{array}
  \right.
\end{equation}

\begin{equation}\label{eq:BVP-spins}
  \calP^{\calS}_{\Omega, \varpi} : \left\{
    \begin{array}{llll}

      f_{\calS} \text{ is an holomorphic spinor in } \Omega_{\varpi}. \\

      H_{\calS}:=\myi \int^{z} f_{\calS}^{2}(\nu) \dd \nu  \text{ is bounded from above near } u_2,\ldots,u_n.\\
      H_{\calS}^{\dagger} := \myi \int^{z} [f_{\calS}(\nu)- \frac{e^{\icomp \frac{\pi}{4}}}{\sqrt{\nu-u_1}}]^2 \dd \nu \text{ is bounded near } u_1. \\
      H_{\calS} \text{ vanishes in } \partial \Omega.

    \end{array}
  \right.
\end{equation}

Each of the above boundary value problems has a unique solution see~\cite{CHI-mixed} (the uniqueness proofs go as in discrete, while their respective existence come from explicit constructions, first in concrete domains and then mapped to general domains). It is also to get the existence from construction in discrete. The solutions feature the following conformal rules. Given $ \phi : \Omega \to \phi(\Omega) $ a conformal map between two simply connected domains,  one has

\begin{equation}\label{eq:conformal_rule_energy}
  f^{\calE}_{\Omega,a}(z)=  |\phi'(a)| \cdot f^{\calE}_{\phi(\Omega),\phi(a)}[\phi(z)], 
\end{equation}
\begin{equation}\label{eq:conformal_rule_spin}
  f^{\calS}_{\Omega, \varpi}(z)=  \prod_{k=1}^{n}|\phi'(u_k)|^{\frac{1}{8}} \cdot f^{\calS}_{\Omega, \phi( \varpi)}[\phi(z)].   
\end{equation}

The link between Ising-related quantities and their scaling limits reads near the singularities of $ f^{\calE}_{\Omega, a}$ and $f^{\calS}_{\Omega, \varpi}$. For the solution of $\calP^{\calE}_{\Omega , a}$, one has the expansion
\begin{equation}\label{eqn:expansion_energy_observable}
  f^{\calE}_{\Omega,a}(z) - \frac{1}{\pi(z-a)} \overset{z \to a}{\longrightarrow} \frac{1}{\sqrt{2}\pi} \ell_{\Omega}(a),
\end{equation}
where $l_{\Omega}(a)$ is the conformal modulus of $\Omega$ seen from $a$~\cite[Thm.~1]{HS-energy-Ising}.

For the multiple energy correlation function, first recall that we denoted $\bbK(a_1, \ldots, a_r) $ the $r \times r$ matrix with coefficients $\bbK_{i,j}= \mathbf{1}_{i\neq j} (a_i-a_j)^{-1} $. Given $a_1,\ldots,a_r$ distinct points of $\bbH $, set

\begin{equation}\label{defn:multiple_energ_correlator}
  \langle \epsilon_{a_1} \epsilon_{a_2} \cdots  \epsilon_{a_r} \rangle_{\bbH}^{+} := \frac{1}{(\sqrt{2}i\pi)^{n} } \Pf(a_1,\ldots,a_r,\overline{a_1},\ldots, \overline{a_r}),
\end{equation}
where $\Pf $ denotes the usual Pfaffian operator.
Given $\varphi :\bbH \to \Omega $ a conformal map, one can define the continuous energy correlation in $\Omega $ via
\begin{equation}
  \langle \epsilon_{a_1} \epsilon_{a_2} \cdots  \epsilon_{a_r} \rangle_{\bbH}^{+}:= \prod^{r}_{k=1} |\varphi'(a_k)| \cdot \langle \epsilon_{\varphi (a_1)} \epsilon_{ \varphi (a_2)} \cdots  \epsilon_{ \varphi (a_r)} \rangle_{\Omega}^{+}.
\end{equation}

One can also see that the solution to $\calP^{\calS}_{\Omega_\varpi}$ admits an expansion near $u_1$ of the form
\begin{equation}\label{eq:def_coefficient_A}
  f^{\calS}_{\Omega, \varpi}(z)=  e^{\icomp \frac{\pi}{4}} \big[ \frac{1}{\sqrt{z-u_1}}+ 2  \calA_{[\Omega,u_1,\ldots,u_n]}  \sqrt{z-u_1} + O( (z-u_1)^{\frac{3}{2}}) \big].
\end{equation}
This expansion allows to define the coefficient $\calA_{(u_1,\ldots,u_n)}= \calA_{[\Omega,u_1,\ldots,u_n]}$.
Now, one can \emph{define} the \emph{continuous correlation fuction} 
\begin{equation}\label{eq:def_correlation_+}
  \langle \sigma_{u_1}\cdots \sigma_{u_n} \rangle_{\Omega}^{+}:= \exp \big( \myr \big[ \displaystyle \int^{(u_1,\ldots,u_n)}  \sum\limits_{k=1}^{n} \calA_{(x_1,\ldots,x_n)} \dd x_k \big] \big),
\end{equation}
with a normalization is chosen so that $\langle \sigma_{u_1}\cdots \sigma_{u_n} \rangle_{\Omega}^{+} \sim \frac{1}{|u_2-u_1|^{\frac14}} \langle \sigma_{u_3}\cdots  \sigma_{u_n} \rangle_{\Omega}^{+}$ as $u_1 \to u_2 $. The above integration should be understood as a integration on the space of $n$-tuples living in $\Omega^n $.

\medskip

In the special case $n=2$, one can also define the coefficient $\calB_{\Omega}(u_1,u_2)\in \bbR$ via the expansion of $f^{\calS}_{\Omega, \varpi} $ near $u_2$, such that
\begin{equation}\label{eq:def_coefficient_B}
  f^{\calS}_{\Omega, \varpi}= \frac{e^{-i\frac\pi4}}{(z-u_2)^\frac12}\calB_{\Omega}(u_1,u_2) + O(z-u_2)^{\frac12}.
\end{equation}
Due to the branching structure, the coefficient $\calB_{\Omega}(u_1,u_2)$ is defined up to the sign, thus one can simply chose it to be positive.
This allows to define the correlation function with free boundary conditions, which is at stake in the convergence of the model with $\free $ boundary conditions.

\begin{equation}\label{eq:def_correlation_free}
  \langle \sigma_{u_1}\sigma_{u_2} \rangle^{\free} := \calB_{\Omega}(u_1,u_2)\cdot \langle \sigma_{u_1}\sigma_{u_2} \rangle^{+}.
\end{equation}

For a conformal mapping $ \varphi $ and $\mathfrak{b} \in \{ +, \free \} $, one has the conformal rule

\begin{equation}
  \langle \sigma_{u_1}\sigma_{u_2} \rangle_{\Omega}^{\mathfrak{b}} = \langle \sigma_{\varphi(u_1)}\sigma_{\varphi(u_2)} \rangle_{\varphi(\Omega)}^{\mathfrak{b}}\cdot |\varphi'(u_1)|^{\frac18}|\varphi'(u_2)|^{\frac18}.
\end{equation}

\section{Derivation of main theorems}\label{sec:derivation_main_theorems}

The proofs of convergence of discrete Ising quantities start by proving that discrete fermionic observables constructed out of mixed correlators converge to their continuous counterparts. The general strategy is to prove the existence of subseqential limits (for the topology of the uniform convergence on compacts) and that any of those sub-sequential limits (now viewed in continuum) is solution to a continuous boundary value problem. As in the isoradial case, the precompactness is achieved by using the discrete primitive of the imaginary part of the square. Next lemma recalls a sufficient condition for the family $(F^{\diamond}_{\delta})_{\delta >0} $ to be precompact via some control on $(\prim)_{\delta >0} $.

\begin{lem}[Precompactness of s-holomorphic functions] \label{lem:precompactness}
  Let $Q \subset \Omega$ be a rectangular domain such that $9Q \subset \Omega$.
  Let $(\obsmedial)_{\delta > 0}$ be a family of diamond s-holomorphic functions on $\medialdomain$ and $\prim = \myi \int {\obsmedial}^2$.
  If $(\prim)_{\delta>0}$ is uniformly bounded on $9Q_\delta$, then $(\obsmedial)$ is precompact on $Q_\delta$.
\end{lem}
\begin{proof}
  This result is proven in~\cite[Thm.~5.3]{Li-QI-SLE}.
\end{proof}

\subsection{Convergence of discrete observable to continous one}
In this subsection we prove the convergence of properly normalized discrete fermionic observables in the semi-discrete lattice to their continuous counterpart. Those proofs are follow the path of the isoradial case.
Recall that the observables $\Psi^{\calE}_{\Omega_\delta, a} $ and $\Psi^{\calS}_{\Omega_\delta, \varpi}$ are respectively the unique solutions to $\calP^{\calE}_{\Omega_\delta, a}$ and $ \calP^{\calS}_{\Omega_\delta, \varpi}$.
As mentioned in Section~\ref{sec:results}, it is enough to prove the convergence results when $\Omega$ is smooth, since the analoguous theorems for general domains can be deduced using results in smooth domain, monotonicity with respect to boundary conditions and FKG inequality.

\begin{prop}\label{prop:convergence_energy_observable}
  Let the point $a $ be an interior to $ \Omega$. Then the sequence of observables $(\delta^{-1}\energysolutiondiscrete)_{\delta >0}$ converges to $\overline{\eta_a}f^{\calE}_{\Omega, a} $, uniformly on compact subsets of $\Omega \backslash \{ a \} $. 
\end{prop}

\begin{proof}

  Denote $F^{\dagger}_{\delta}:=  \delta^{-1} (\energysolutiondiscrete - G_{(a)}) $ where $ G_{(a)}$ is constructed in the appendix, $H_{\delta}:=  \myi [ \int  (\delta^{-1}\energysolutiondiscrete)^2 (\nu)  \dd \nu ] $ and $ H^{\dagger}_{\delta}:= \myi [ \int (F^{\dagger}_{\delta} )^2 (\nu)  \dd \nu ] $, both chosen with proper additive constants so that $H_{\delta}$ vanishes at $\partial \Omega_{\delta} $ and $ H^{\dagger}_{\delta}$ vanishes at some fixed interior point different from $a$ (this second choice is irrelevant). We also set $ (M_{r}^{\delta})^{2}:= \max\limits_{\Omega_\delta \backslash B(a,r)} |H^{\delta}| $. Since the singularities of $\energysolutiondiscrete$ and $G_{(a)}$ cancel each other, the observable $F^{\dagger}_{\delta}$ is $s$-holomorphic everywhere. There are now two potential scenarii (the second one is in fact impossible) :

  \paragraph{Scenario 1:}
  For any $r >0 $, the family $(M_{r}^{\delta})_{\delta >0}$ is uniformly bounded by a constant $C(r)$. Using the precompactness lemma~\ref{lem:precompactness}, one can extrat a subsequential limit from the family 
  $( \delta^{-1} \energysolutiondiscrete)_{\delta > 0}$ such that $\delta^{-1} \energysolutiondiscrete \to f $,  $H_{\delta} \to h= \myi [ \int^{z} f^{2}(\nu) \dd \nu ] $ and $H^{\dagger}_{\delta} \to h^{\dagger}= \myi [ \int^{z} (f - \frac{\overline{\eta_a}}{\pi(\nu -a)})^2 \dd \nu ] $  uniformly on compact subsets of 
  $\Omega \backslash \{ a \} $, with $f$ holomorphic away from $a$, $h$ harmonic except at $a$ and \emph{vanishing} along the boundary of $\Omega$. To obtain this last claim, one has to prove that the discrete boundary conditions of $H_{\delta}$ survive when passing to the limit in continuum (i.e. the function $h$ has Dirichlet boundary conditions). One can first use the boundary modification trick (\cite{Li-QI-SLE}, Lemma 5.1) and modify the weights of the Laplacian~\eqref{eq:mLaplacian} at the boundary to preserve the sub-harmonicity of $H_\delta $ on $\primaldomain $ even at the boundary. Using the subharmnocitity of $H_\delta$ on $\primaldomain $, the superharmonicity of $H_\delta$ on $\dualdomain $ ,the maximum principle for $H_\delta$ and the fact that $H_\delta$ vanishes along the boundary one gets

  \begin{equation*}
    -C(r_0) (1-\text{hm}^{\Omega_\delta \backslash B(a,r)}_{\partial\Omega_\delta}(z)) \leq H_\delta(z)\leq C(r_0) (1-\text{hm}^{\Omega_\delta \backslash B(a,r)}_{\partial\Omega_\delta}(z)),
  \end{equation*}
  where $\text{hm}^{\Omega_\delta \backslash B(a,r)}_{\partial\Omega_\delta}(z) $ stands for the harmonic measure of $\partial\Omega_\delta$ in $\Omega_\delta \backslash B(a,r) $, with respect to the random walk started at $z$.  Due to uniform crossing estimates for the random walk generated by the Laplacian~\eqref{eq:mLaplacian}, $\text{hm}^{\Omega_\delta \backslash B(a,r)}_{\partial\Omega_\delta}(z) $ goes to $1$ (uniformly in $\delta $) as $z$ approaches $\partial \Omega_\delta$. 

  \medskip

  On the other hand, the convergence of $ \delta^{-1} G_{(a)} $ to $\frac{\overline{\eta_a}}{\pi(z-a)} $ and the maximum principle applied to $ H^{\dagger}_{\delta}$  ensures that $ H^{\dagger}_{\delta}$ is uniformly bounded  near $a$ \emph{already in discrete}  (since at a fixed definite distance from $a$, $F^{\dagger}_{\delta}$ converges to $f-\frac{\overline{\eta_a}}{\pi(z-a)}$). In particular $F^{\dagger}_{\delta} $ is uniformly bounded near $a$.  Passing to the limit $ f-\frac{\overline{\eta_a}}{\pi(z-a)} $ is then \text{uniformly bounded} near $a$ (due to the maximum principle applied to the $s$-holomorphic functions $ F^{\dagger}_{\delta}$). We can thus conclude that $f$ is $\overline{\eta_a} f^{\calE}_{\Omega, a} $.

  \paragraph{Scenario 2:}
  Conversly, assume that $M_{r_0}^{\delta} \to \infty$ along some subsequence for some $r_0>0$ and consider $ \tilde{F}^{\calE}_{\Omega_\delta,a} := \frac{ \energysolutiondiscrete }{M_{r_0}^{\delta}} $, $\tilde{F}^{\dagger}_{\delta} := \frac{F^{\dagger}_{\delta}}{M_{r_0}^{\delta}}  $,  $\tilde{H}_{\delta} := \frac{H_{\delta}}{(M_{r_0}^{\delta})^{2}}  $ and $\tilde{H}^{\dagger}_{\delta} := \frac{H^{\dagger}_{\delta}}{(M_{r_0}^{\delta})^{2}}$. We first apply the Harnack inequality to the functions $\tilde{H}_{\delta}$~\cite[Prop. 3.25]{Li-QI-SLE} to deduce that $\tilde{H}_{\delta}$ is uniformly bounded on \emph{any compact subset} away from $a$.
  This way, one can repeat the previous reasoning  to extract a sub-sequential limit from the family $(\delta^{-1}\tilde{F}^{\calE}_{\Omega_\delta,a} )_{\delta >0}$ and $(\tilde{H}_{\delta})_{\delta >0}$, converging respectively (uniformly on compact subets of $\Omega \backslash \{a \} $) to  $\tilde{f} $ and $\tilde{h} = \myi [ \int \tilde{f}^{2}(\nu) \dd \nu]  $, with $\tilde{h}=0 $ on $\partial \Omega $. Up to an another extraction, we can also assume that the sequence of points $(z_{\delta})_{\delta >0} $ in $ \Omega\backslash B(a,\frac{r_0}{2}) $ such that $ \tilde{H}_{\delta}(z_{\delta})=1$ converges to $z_{\infty}$. Since $ (\delta M_{r_0}^{\delta})^{-1} G_{(a)}$ now converges to $0$, we have (still uniformly on compact subsets of $\Omega \backslash \{ a \} $) $\tilde{H}^{\dagger}_{\delta} \to \tilde{h}^{\dagger}=\tilde{h} + \text{cst} = \myi [ \int \tilde{f}^{2} \dd \nu] $. In particular, since $\tilde{h}^{\dagger} $ is bounded away from $a$ and satisfies the maximum principle near $a$ (since $\tilde{F}^{\dagger}_{\delta} $ has no singularity), hence so does $\tilde{h} $. In particular $f$ has no singularity at $a$, and thus is trivial. This is not compatible with the Dirichlet boundary conditions of $\tilde{h}$ and the fact that $\tilde{h}(z_{\infty})=1$.
\end{proof}

Let $\varpi^\delta = \{ v^{\delta},u^{\delta}_2,\ldots,u^{\delta}_n \}$ be points of $\Omega_{\delta} $, at a definite distance from each other, approximating respectively interior points $\varpi = \{u_1,\ldots,u_n \} $ of $\Omega $, with the usual convention that $ v^{\delta}\in \dualdomain$ and $ u^{\delta}_{2},\ldots,u^{\delta}_{n} \in \primaldomain$, and $v^{\delta} \sim u^{\delta}_{1} $.

\begin{prop}\label{prop:convergence_spin_observable}
  In the above mentioned context, the family of discrete observables
  $ \big( \bbE_{\Omega_\delta}^{+}[\sigma_{u_1}\ldots \sigma_{u_n}]^{-1} \delta^{-\frac12} \spinsolutiondiscrete\big)_{\delta >0}  $ converges to $(\tfrac1\pi)^{\frac12} f^{\calS}_{\Omega, \varpi}$ uniformly on compact subsets of  $ \Omega_{\varpi} \backslash \varpi   $.

\end{prop} 

\begin{proof}
  We keep the notations of the previous proof. Denote for the rest of the proof $F_\delta:= \delta^{-\frac12} \bbE_{\Omega_\delta}^{+}[\sigma_{u_1}\ldots \sigma_{u_n}]^{-1}F^{\calS}_{\Omega_\delta, \varpi} $, $ F^{\dagger}_{\delta}:= F_{\delta} - \delta^{-\frac12}G_{[v]}$, $H_{\delta}:=   \myi [ \int  (F_\delta)^2 (\nu) \dd \nu ]$ and $ H^{\dagger}_{\delta}:=  \myi [ \int (F^{\dagger}_{\delta})^2 (\nu)  \dd \nu ]$, with $H_{\delta}$ vanishing along $\partial \Omega_{\delta} $ and $ H^{\dagger}_{\delta}$ vanishing at a point different from the punctures. One can see that $ F^{\dagger}_{\delta}$ vanishes near the branching $v^{\delta}$. We also set $(M_{r}^{\delta})^{2}:= \max\limits_{\Omega \backslash \cup_{i=1\ldots n} B(u_i,r)} |H^{\delta}|$. There are two potential scenarii (the second one is again impossible).

  \paragraph{Scenario 1:}
  If for any $r>0$,  $(M_{r}^{\delta})_{\delta>0}$ is bounded by a constant $C(r)$.
  Using the precompactness lemma~\ref{lem:precompactness}, one can extract a subsequential limit from the family $(F_{\delta})_{\delta >0} $ as $\delta \to 0$. Any subsequential limit $f$ of $F_{\delta} $ is an holomorphic spinor on $\Omega_{\varpi} $ and $h$ is harmonic function away from the branchings.
  The sub/super harmonicity of $H_{\delta} $, the boundary modification trick and the control of discrete harmonic measures in smooth domains allow to preserve again boundary conditions from discrete to continuum, i.e. $H_{\delta} \to h= \myi[\int f^2 \dd \nu]$ uniformly on compact subsets of $\Omega \backslash \varpi $, with $h=0 $ on $\partial \Omega $. Since $F^{\dagger}_{\delta}$ vanishes near $v^\delta$, one can use the maximum and the minimum principle in discrete for $H^{\dagger}_\delta $ near $v^{\delta}$ and the convergence (away from $u_1$) of $F^{\dagger}_{\delta}$ to  $f- (\tfrac1\pi)^{\frac12} \frac{e^\frac{i\pi}{4}}{\sqrt{\nu-u_1}} $ to deduce that $H^{\dagger}_{\delta} \to h^{\dagger}= \myi[ \int (f- (\tfrac1\pi)^{\frac12} \frac{e^\frac{i\pi}{4}}{\sqrt{\nu-u_1}})^2 \dd \nu]$, the latter convergence being uniform near $u_1$. This ensures that $h^{\dagger} $ is bounded from above and from below near $u_1$. By subharmonicity property of $H_{\delta} $, we see that $H_{\delta} $ is bounded from above near $u^{\delta}_2,\ldots,u^{\delta}_n$ \emph{already in discrete}, which implies the same one sided bound for $h$ near $u_2,\ldots,u_n$.  This allows to identify $f$ as a multiple of  $f^{\calS}_{\Omega, \varpi}$ and the normalizing constant is fixed by the asymptotic near $u_1$.

  \paragraph{Scenario 2:}
  If $M_{r}^{\delta}\to \infty$ along some subsequence, one normalizes again by $M^{\delta}_{r} $ as in the previsous proposition and extracts a new subsequential limit $\tilde{f} $, which remains an holomorphic spinor on $\Omega_{\varpi} $, such that $\tilde{h}=  \myi [ \int \tilde{f}^2 \dd \nu] $ is harmonic in $\Omega \backslash \varpi  $, satisfies Dirichlet Boudary condition and is bounded from below near $u_2,\ldots,u_n$ and has positive outer normal derivative. We know that $\tilde{H}^{\dagger}_{\delta} $ satisfies the maximum principle in discrete at a small definite distance from $v^{\delta} $. Moreover $\tilde{h}=\tilde{h}^{\dagger} + \text{cst} $ thanks to the normalization by $M_{r}^{\delta}$. Thus $\tilde{h}$ is bounded from below near all points $u_1,\ldots,u_n $. This implies that $\tilde{h}$ vanishes identically. It remains to adapt the proof of~\cite[Lem.~3.10]{CHI-spin} which states that no subsequential limits of $ (M_{r}^{\delta})^{-2}H^{\delta}  $ is trivial. This (somewhat technical) proof only relies upon the same sub/super harmonicity arguments, and readingly addapts.
\end{proof}

Having explicit asymptotics of the kernels $G_{[u]}$, $G_{[v]} $ and $G_{(a)} $, one can prove the same statement only using one side of the sub/super harmonicity property of functions $H^{\delta}$, at least in smooth domains (e.g.~\cite{CIM-universality}). It could be even possible in principle to bypass the use of all sub/super harmonicity properties using ideas of~\cite{Che20}.

\subsection{Derivation of the main theorems}

We are now going to use the kernels constructed explicitely and the convergence theorems for $s$-holomorphic functions arising from boundary value problems to derive Theorems~\ref{thm:spins} and~\ref{thm:multiple_energy}.
It is important to notice that, comparing to the isoradial case, one has to be careful due to the denegeracy of the lattice in the vertical direction.
In particular, the use of the Dotsenko propagation equation becomes crucial to derive the results in the vertcical direction.
We start with the proof of~\ref{thm:ratios_spins} which is the convergence proof for correlation ratios of general spin-spin correlations.

\begin{proof}
  As before, we work with $\Omega_\delta $, an approximation of a bounded smooth simply connected domain $\Omega$.
  We assume that $ \varpi^{\delta} = \{ v^{\delta}, u^{\delta}_2,\ldots,u^{\delta}_n \}$ approximate respectively $u_1,\ldots,u_n $ interior points of $\Omega$, remaining at a definite distance from each other an from the boundary.
  For simplicity of proof reading, we denote $F^{\diamond}_{v^{\delta}}:= \delta^{-\frac{1}{2}}  F^{\calS}_{\Omega_\delta, \varpi^\delta} $.
  We start the proof of Theorem~\ref{thm:ratios_spins} with the derivation of correlation ratio in the horizontal direction.

  \begin{figure}[htb] \centering
    \includegraphics[scale=0.25, page=1]{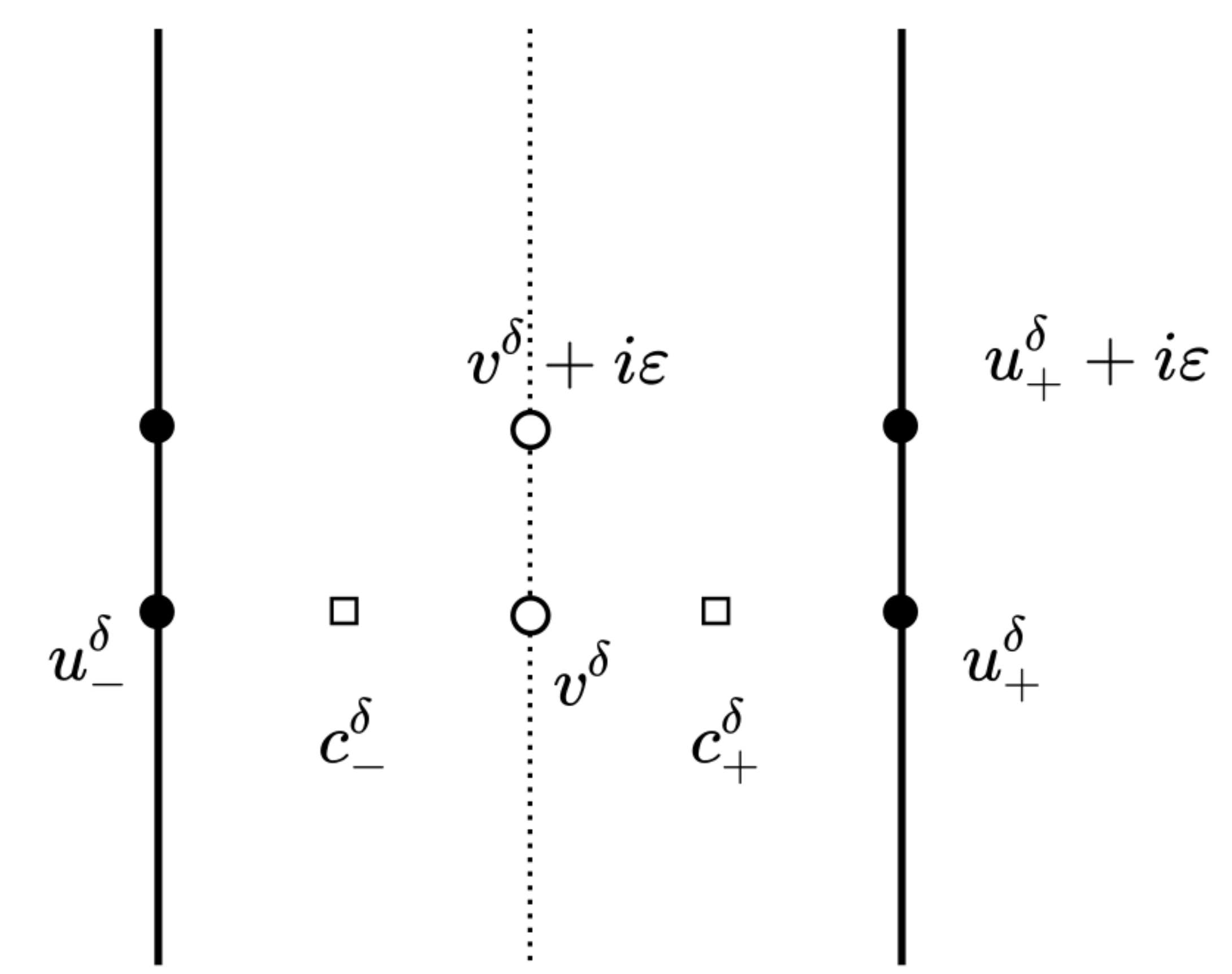} 
    \caption{Configuration of points used to derive convergence of correlation ratios, in the horizontal and the vertical direction. We chose $v^{\delta}$ and $v^{\delta}+i\varepsilon $ (respectively $u^{\delta}_{+} $ and $u^{\delta}_{+} + i\varepsilon$) to be on the same sheet of the double cover. For the proof in the vertical direction, deforming the contour leads to compute the two increments passing by $c^{\delta}_{+}$, which compensante each other due the opposite values of $\sgn(\mathcal{O}_{c^{\delta}_{+}},v^\delta) $ depending on the position of corners compared to $v^{\delta}$.}
    \label{fig:correlation_ratios}
  \end{figure}
  
  \paragraph{Correlation ratio in the horizontal direction}
  
  Consider the vertices $v^{\delta}, u^{\delta}_{-}, u^{\delta}_{+} $ and $c^{\delta}_{\pm} = (v^{\delta}u_{c^{\delta}_{\pm}})$ as in Figure~\ref{fig:correlation_ratios}.
  Fix $r > 0$ a (small) positive number and let $\calC^{\delta}_{r}$ be a discrete square of width $2r$ centered at $v^{\delta}$, oriented counter-clockwise.
  We apply the integration trick of Proposition~\ref{prop:integration-trick} twice with the contour $\calC^{\delta}_{r}$ to the functions $F^{\diamond}_{v^\delta}$ and $\delta^{-\frac12}G^{\diamond}_{[u^\delta_{-}]}$ on one hand, and to the functions $F^{\diamond}_{v^\delta}$ and $\delta^{-\frac12}G^{\diamond}_{[u^\delta_{+}]}$ on the other hand.
  One gets,
  \begin{eqnarray}\nonumber
    \frac{2 \delta \cdot \delta^{-\frac12}F_{v^\delta}(c^\delta_{+}) \cdot \delta^{-\frac12} \overline{G_{[u^{\delta}_{+}]} (c^\delta_{+})}  }{2 \delta \cdot \delta^{-\frac12}F_{v^\delta}(c^\delta_{-}) \cdot \delta^{-\frac12}\overline{G_{[u^{\delta}_{-}]} (c^\delta_{-})} } &=& \frac{\displaystyle\int_{ \calC^{\delta}_{r}}  \myi [ F^{\diamond}_{v^\delta}(z) \delta^{-\frac12}G^{\diamond}_{[u^{\delta}_{+}]}(z) \dd z  ]}{\int_{ \calC^{\delta}_{r}}  \myi [  F^{\diamond}_{v^\delta}(z)\delta^{-\frac12}G^{\diamond}_{[u^{\delta}_{-}]}(z) \dd z  ]} \\ \nonumber
                                                                                                                                                                                                                                                                                 &=& \frac{1 +    \frac{ \displaystyle\int_{ \calC^{\delta}_{r}}  \myi \big[  F^{\diamond}_{v^\delta}(z)\delta^{-\frac12}\big( G^{\diamond}_{[u^{\delta}_{+}]}(z) -G^{\diamond}_{[u^{\delta}_{-}]}(z) \big)  \dd z  \big] }{\displaystyle\int_{ \calC^{\delta}_{r}}  \myi \big[  F^{\diamond}_{v^\delta}(z)\delta^{-\frac12}\big( G^{\diamond}_{[u^{\delta}_{+}]}(z) +G^{\diamond}_{[u^{\delta}_{-}]}(z) \big)  \dd z  \big]}    }{1 -    \frac{ \displaystyle\int_{ \calC^{\delta}_{r}}  \myi \big[  F^{\diamond}_{v^\delta}(z)\delta^{-\frac12}\big( G^{\diamond}_{[u^{\delta}_{+}]}(z) -G^{\diamond}_{[u^{\delta}_{-}]}(z) \big)  \dd z  \big] }{\displaystyle\int_{ \calC^{\delta}_{r}}  \myi \big[  F^{\diamond}_{v^\delta}(z)\delta^{-\frac12} \big( G^{\diamond}_{[u^{\delta}_{+}]}(z) +G^{\diamond}_{[u^{\delta}_{-}]}(z) \big)  \dd z  \big]}  }.
  \end{eqnarray}

  The convergence of the spin fermionic observable given in Proposition~\ref{prop:convergence_spin_observable} ensures that,
  \begin{equation*}
    \frac{F_{v^\delta}^{\diamond}(z)}{\bbE_{\primaldomain}^{+}[  \sigma_{u_{+}^{\delta}} \ldots \sigma_{u^\delta_n} ]}    
    =  \frac{ e^{\icomp \frac{\pi}{4}} }{ \sqrt{\pi} }  \big( \frac{1}{\sqrt{z-u_1}}+ 2 \calA_{(u_1,\ldots,u_n)} \sqrt{z-v} + O( (z-v)^{\frac{3}{2}}) \big),
  \end{equation*}
  uniformly in $z\in \Omega^{\diamond}_{\varpi_\delta} $ remaining at a fixed distance from the punctures and the boundary, where the coefficient $\calA=\calA_{(\Omega,u_1,\ldots,u_n)}$ is defined in~\eqref{eq:def_coefficient_A}.
  Additionally, the asymptotic of the kernels $G^{\diamond}_{[u^{\delta}_{\pm}]}$ away from their branchings gives,
  \begin{equation*}
    \delta^{-\frac12}(G^{\diamond}_{[u^{\delta}_{+}]}(z)-G^{\diamond}_{[u^{\delta}_{-}]}(z))
    = \frac{ e^{-i\frac{\pi}{4}}}{ \sqrt{\pi} } \frac{(u^{\delta}_{+}-u^{\delta}_{-})}{2(z-v)^{\frac{3}{2}}} + O \Big( \frac{\delta^2}{(z-v)^{\frac{5}{2}}} \Big).
  \end{equation*}

  One first multiplies the two previous asymptotics, then uses classical bounds on interpolation of countour integrals via their discretization, together with the fact that  integrals of the continuous limits do \emph{not} depend on the contour on integration. Since $r$ can be chosen arbitrarily small, the residue theorem ensures that when $\delta \rightarrow 0$,
  \begin{align*}
    &  \myi \big[ \int_{ \calC^{\delta}_{r}}  \frac{F^{\diamond}_{v^\delta}(z)}{\bbE_{\primaldomain}^{+}[  \sigma_{u_{+}^{\delta}} \ldots \sigma_{u^\delta_n} ]}  \delta^{-\frac12}\big( G^{\diamond}_{[u^{\delta}_{+}]}(z) -G^{\diamond}_{[u^{\delta}_{-}]}(z) \big)  \dd z  \big] \\
    & = 2 \myr [\calA\cdot(u^{\delta}_{+}-u^{\delta}_{-})] + O(r\delta)+ o(\delta) \\
    & = 4\delta \textnormal{Re}[\calA_{(u_1,\ldots,u_n)}] + o(\delta).
  \end{align*}

  Similarly, for $\delta \rightarrow 0$, one has,
  \[
     \myi \big[ \int_{ \calC^{\delta}_{r}}  \frac{F^{\diamond}_{v^\delta}(z)}{\bbE_{\primaldomain}^{+}[  \sigma_{u_{+}^{\delta}} \ldots \sigma_{u^\delta_n} ]}\delta^{-\frac12}(G^{\diamond}_{[u^{\delta}_{+}]}(z) +G^{\diamond}_{[u^{\delta}_{-}]}(z) )  \dd z  \big] = 4+o(1),
  \]
  where $o$ is uniform provided $u_1, \ldots, u_n$ remain at a definite distance from each other and from the boundary.
  Finally, it is enough to notice that $|G_{[u^{\delta}_{+}]}(c^{\delta}_+)|= |G_{[u^{\delta}_{-}]}(c^{\delta}_-)|=1$ to conclude that,
  \begin{equation}
    \label{eq:correlation_ratio_horizontal}
    \frac{\bbE_{\primaldomain}^{+}[ \sigma_{u^{\delta}_{+}} \ldots \sigma_{u^\delta_n} ] }{\bbE_{\primaldomain}^{+}[ \sigma_{u^{\delta}_{-}} \ldots \sigma_{u^\delta_n} ]} = 1 +   2\delta \textnormal{Re}[\calA_{(u_1,\ldots,u_n)}] + o(\delta).
  \end{equation}

  \paragraph{Correlation ratio in the vertical direction}
  The proof in the vertical direction is more tedious and requires the use of the propagation equation~\eqref{eq:primal_two_term} and the asymptotics of correlation ratio in the horizontal direction derived in~\eqref{eq:correlation_ratio_horizontal}.
  We first fix $v^{\delta},v^{\delta}+i\epsilon ,u_{+}^{\delta}, u_{+}^{\delta}+i\epsilon, u_{-}^{\delta}$ as in Figure \ref{fig:correlation_ratios}, assuming $v^{\delta},v^{\delta}+i\epsilon$ and $u_{+}^{\delta}, u_{+}^{\delta}+i\epsilon, u_{-}^{\delta}$ are taken respectively on the same sheet. Applying the previous reasoning on a macroscopic contour $\calC^{\delta}_{r}$ of small radius $r$ one gets,
  \begin{equation}\label{eq:logarithmic_ratio_vertical_correlation}
    \frac{\bbE_{\primaldomain}^{+}[ \sigma_{u_{+}^{\delta}+i\epsilon} \ldots \sigma_{u^\delta_n} ] }{\bbE_{\primaldomain}^{+}[ \sigma_{u_{+}^{\delta}} \ldots \sigma_{u^\delta_n} ]}
    = \frac{1 + \frac{ \displaystyle \myi \big[ \int_{ \calC^{\delta}_{r}}   F^{\diamond}_{v^\delta+i\epsilon}(z)\delta^{-\frac12} G_{[u_{+}^{\delta}+i\epsilon]}(z) - F^{\diamond}_{v^\delta}(z)\delta^{-\frac12}G^{\diamond}_{[u_{+}^{\delta}]}(z)   \dd z  \big] }{\displaystyle   \myi \big[ \int_{ \calC^{\delta}_{r}}  F^{\diamond}_{v^\delta+i\epsilon}(z)\delta^{-\frac12}G^{\diamond}_{[u_{+}^{\delta}+i\epsilon]}(z) + F^{\diamond}_{v^\delta}(z)\delta^{-\frac12}G^{\diamond}_{[u_{+}^{\delta}]}(z)   \dd z  \big]}}{1 -     \frac{ \displaystyle  \myi \big[ \int_{ \calC^{\delta}_{r}}  F^{\diamond}_{v^\delta+i\epsilon}(z)\delta^{-\frac12} G_{[u_{+}^{\delta}+i\epsilon]}(z) - F^{\diamond}_{v^\delta}(z)\delta^{-\frac12}G^{\diamond}_{[u_{+}^{\delta}]}(z)   \dd z  \big] }{\displaystyle \myi \big[ \int_{ \calC^{\delta}_{r}}  F^{\diamond}_{v^\delta+i\epsilon}(z)\delta^{-\frac12}G^{\diamond}_{[u_{+}^{\delta}+i\epsilon]}(z) + F^{\diamond}_{v^\delta}(z)\delta^{-\frac12}G^{\diamond}_{[u_{+}^{\delta}]}(z)   \dd z  \big]} }. 
  \end{equation}

  We use the factorization
  \begin{align*}
    & F^{\diamond}_{v^\delta+i\epsilon}(z)G^{\diamond}_{[u_{+}^{\delta}+i\epsilon]}(z) - F^{\diamond}_{v^\delta}(z)G^{\diamond}_{[u_{+}^{\delta}]}(z) \\
    & = F^{\diamond}_{v^\delta}(z)\big( G^{\diamond}_{[u_{+}^{\delta}+i\epsilon]}(z)-G^{\diamond}_{[u_{+}^{\delta}]}(z) \big)
      + \big( F^{\diamond}_{v^\delta+i\epsilon}(z) - F^{\diamond}_{v^\delta}(z) \big) G^{\diamond}_{[u_{+}^{\delta}+i\epsilon]}(z),
  \end{align*}
  and treat separately the two terms on the RHS of the previous equality.

  \textbf{First term} Using the differentiability of the correlator with respect to its branching position (see~\eqref{eq:differential_correlator_position}), one can write the expansion,
  \begin{equation}
    \label{eq:diff_G}
    G^{\diamond}_{[u_{+}^{\delta}+i\epsilon]}(z)-G^{\diamond}_{[u_{+}^{\delta}]}(z) = i\epsilon \partial_{y} G^{\diamond}_{[u_{+}^{\delta}]}(z) + o(\epsilon),
  \end{equation}
  where $o$ is uniform provided $\epsilon$ is small enough.
  Thus, integrating on $\calC^{\delta}_{r} $, one has,
  \begin{align*} 
    &   \myi \big[\int_{ \calC^{\delta}_{r}}  F^{\diamond}_{v^\delta}(z) \delta^{-\frac12}\big( G^{\diamond}_{[u_{+}^{\delta}+i\epsilon]}(z)-G^{\diamond}_{[u_{+}^{\delta}]}(z) \big) \dd z \big] \\
    & = \epsilon   \myi \big[ \int_{ \calC^{\delta}_{r}}  \icomp \delta^{-\frac12}\partial_{y} G^{\diamond}_{[u_{+}^{\delta}]}(z)F^{\diamond}_{v^\delta}(z)  \dd z \big] + o(\epsilon).
  \end{align*}

  \textbf{Second term} We take advantage of the propagation equation stated previously.
  To simplify the proof reading of the following lines, we introduce the Kadanoff--Ceva correlator $Y_{[v^{\delta},u_{-}^{\delta},u_{+}^{\delta}]}(c):= \delta^{-\frac12}\eta_{c}\langle \mu_{v_c}\sigma_{u_c} \mu_{v^{\delta}} \sigma_{u^{\delta}_{+}}\sigma_{u^{\delta}_{-}} \sigma_{u_2} \ldots \sigma_{u_n}\rangle_{\Omega_\delta}  $, where $u^{\delta}_{\pm} $ are taken on the same sheet as in Proposition~\ref{prop:two_term}.
  We also denote $Y^{\diamond}_{[v^{\delta},u_{-}^{\delta},u_{+}^{\delta}]} $ its associated diamond $s$-holomorphic extension, see Proposition~\ref{prop:s_hol_equiv}.
  One sees that $Y^{\diamond}_{[v^{\delta},u_{-}^{\delta},u_{+}^{\delta}]} $ branches only around the three vertices $v^{\delta},u_{-}^{\delta},u_{+}^{\delta} $.
  In that case, Proposition~\ref{prop:two_term} reads,
  \begin{align*}
    & \delta^{-\frac12}\big(\langle \mu_{v_c} \sigma_{u_c} \mu_{v^{\delta}+i\epsilon} \sigma_{u^{\delta}_2} \ldots \sigma_{u^{\delta}_n} \rangle_{\Omega_\delta}
    - \langle \mu_{v_c} \sigma_{u_c} \mu_{v^{\delta}} \sigma_{u^{\delta}_2} \ldots \sigma_{u^{\delta}_n} \rangle_{\Omega_\delta} \big) \\
    & = (2\delta)^{-1} \sgn(\calO(c),v^{\delta}) Y_{[v^{\delta},u_{-}^{\delta},u_{+}^{\delta}]}(c) +O(\epsilon^2),
  \end{align*}
  where $O$ is uniform over corners $c$ remaining at a fixed distance away from $u_{+}^{\delta},u_{2}^{\delta}\ldots u_{n}^{\delta}$.
  Moreover, one notices that, provided $r$ is small enough, the function  $\sgn(\calO(c),v^{\delta}) Y_{[v^{\delta},u_{-}^{\delta},u_{+}^{\delta}]}$ has the branching structure of $ \langle \mu_{v_c}\sigma_{u_c} \mu_{v^{\delta}} \sigma_{u^\delta_2} \ldots \sigma_{u^\delta_n}\rangle_{\Omega_\delta} $ (this can be seen as a feature of Proposition \ref{prop:two_term}).
  Using again the expansion~\eqref{eq:diff_G} and integrating on $\calC^{\delta}_{r} $, one has,
  \begin{align*}
    & \myi \big[ \int_{ \calC^{\delta}_{r}} \big( F^{\diamond}_{v^\delta+i\epsilon}(z) - F^{\diamond}_{v^\delta}(z) \big) \delta^{-\frac12}G^{\diamond}_{[u_{+}^{\delta}+i\epsilon]}(z)  \dd z \big] \\ 
    & = \epsilon\cdot  (2\delta)^{-1}  \myi \big[ \int_{ \calC^{\delta}_{r}} \sgn(\calO(z),v^{\delta})Y^{\diamond}_{[v^{\delta},u_{-}^{\delta},u_{+}^{\delta}]}(z) \delta^{-\frac12}G^{\diamond}_{[u_{+}^{\delta}]}(z)  \dd z \big]+ o(\epsilon).
  \end{align*}
  A similar treatment shows that 
  \begin{align*}
    & \myi \big[  \int_{ \calC^{\delta}_{r}}  F^{\diamond}_{v^\delta+i\epsilon}(z)\delta^{-\frac12}G^{\diamond}_{[u_{+}^{\delta}+i\epsilon]}(z) + F^{\diamond}_{v^\delta}(z)\delta^{-\frac12}G^{\diamond}_{[u_{+}^{\delta}]}(z)  \dd z  \big] \\
    & = 2 \myi \big[  \int_{ \calC^{\delta}_{r}}   F^{\diamond}_{v^\delta}(z)\delta^{-\frac12}G^{\diamond}_{[u_{+}^{\delta}]}(z)   \dd z  \big] +o_{\epsilon \to 0}(1).
  \end{align*}
  Putting the previous asymptotics all together in~\eqref{eq:logarithmic_ratio_vertical_correlation} and taking the logarithm, one deduces that $\partial_y \log \bbE_{\primaldomain}^{+}[ \sigma_{u_1^{\delta}} \ldots \sigma_{u^{\delta}_n} ] $ exists and is given by,
  \begin{equation}
    \frac{  \displaystyle  \myi \big[ \int_{ \calC^{\delta}_{r}}  \icomp \delta^{-\frac12}\partial_{y} G^{\diamond}_{[u_{+}^{\delta}]}(z)F^{\diamond}_{v^\delta}(z)  \dd z \big]
      + \myi \big[ \int_{ \calC^{\delta}_{r}} \frac{\sgn(\calO(z),v^{\delta}) Y^{\diamond}_{[v^{\delta},u_{-}^{\delta},u_{+}^{\delta}]}(z)}{2\delta} \delta^{-\frac12}G^{\diamond}_{[u_{+}^{\delta}]}(z)  \dd z \big]}{
      \displaystyle \myi \big[ \int_{ \calC^{\delta}_{r}} F^{\diamond}_{v^\delta}(z)\delta^{-\frac12}G^{\diamond}_{[u_{+}^{\delta}]}(z)   \dd z  \big]}.
  \end{equation}
  We are now going to evaluate the contribution of the above terms as $\delta \rightarrow 0 $.
  Repeating computations similar to the proof in the horizontal direction, one sees that,
  \begin{align*}
    \displaystyle  \myi \big[ \int_{ \calC^{\delta}_{r}} F^{\diamond}_{v^\delta}(z)\delta^{-\frac12}G^{\diamond}_{[u_{+}^{\delta}]}(z)   \dd z  \big]
    & = 2\bbE_{\primaldomain}^{+}[ \sigma_{u^{\delta}_{+}} \ldots \sigma_{u^\delta_n} ]+o_{\delta \rightarrow 0}(\bbE_{\primaldomain}^{+}[ \sigma_{u^{\delta}_{+}} \ldots \sigma_{u^\delta_n} ]), \\
    \displaystyle  \myi \big[ \int_{ \calC^{\delta}_{r}} \icomp \delta^{-\frac12}\partial_{y} G^{\diamond}_{[u_{+}^{\delta}]}(z)F^{\diamond}_{v^\delta}(z)  \dd z \big]
    & =  2\bbE_{\primaldomain}^{+}[ \sigma_{u^{\delta}_{+}} \ldots \sigma_{u^\delta_n} ]\Re \big[i\mathcal{A}_{(u_1,\ldots, u_n)} \big] \\
    & \qquad \qquad + o_{\delta \rightarrow 0}(\bbE_{\primaldomain}^{+}[ \sigma_{u^{\delta}_{+}} \ldots \sigma_{u^\delta_n} ]),
  \end{align*}
  since the asymptotics of the derivative of the full-plane correlator away from its branching is given by (see \eqref{eq:differential_correlator_position})
  \[
    \delta^{-\frac12}\partial_{y} G_{[u^{\delta}_{+}]}(z) = - \frac{\icomp}{2 \sqrt{\pi} (z-u)^{\frac{3}{2}}} + O(\frac{\delta}{(z-u)^{\frac52}}).
  \]

  Now, we treat now the remaining term and prove it \emph{vanishes identically}.
  One first notes that the double covers $\Omega_{[v^{\delta}]} $ and $\Omega_{[u_{+}^{\delta},v^{\delta},u^{\delta}_{-}]} $ can be canonically identified. This only amounts to change the sheets of \emph{both} corners $c^{\delta}_{\pm} $ neighbouring $v^{\delta}$.
  Thus, one can use the integration procedure of Proposition~\ref{prop:integration-trick} applied to the correlators $\sgn(O(z),v^{\delta})Y^{\diamond}_{[v^{\delta},u_{-}^{\delta},u_{+}^{\delta}]} $ and $\delta^{-\frac12}G^{\diamond}_{[u_{+}^{\delta}]} $ to modify the contour and make it an elementary one passing by $c^{\delta}_{+}$.
  This is indeed a fair operation since $\sgn(O(z),v^{\delta})$ is locally constant away from the branchings.
  Thus, by contour deformation, the only terms that do not \emph{trivially} cancel out are the increments of $H[Y^{\diamond}_{[v^{\delta},u_{-}^{\delta},u_{+}^{\delta}]}, \delta^{-\frac12}G^{\diamond}_{[u_{+}^{\delta}]}(z)] $ along the segment $[v^{\delta}u^{\delta}_{+}]$ (since the two double covers $\Omega_{[v^{\delta}]}$ and $ \Omega_{[u^{\delta}_{+}]}$ are \emph{not} identified at lifts of $c^{\delta}_{+}$.
  We claim that those two increments in fact \emph{do} compensate each other.
  Indeed, fix a lift of $c_{+}^{\delta} $ in $\Omega_{[v^{\delta}]} $.
  Since the contour is oriented counter-clockwise, the ``upper'' increment passing by $c_{+}^{\delta} $ is oriented from left to right.
  Up to a global sign choice, the sum of those two increments (as in the proof of Proposition~\ref{prop:integration-trick}) equals,
  \begin{equation}\label{eq:vanishing-increment}
    \delta \cdot (+1) F_{v^{\delta}}(c^{\delta}_{+}) \delta^{-\frac12}G_{[u_{+}^{\delta}]}(c^{\delta}_{+}) - \delta \cdot (-1)F_{v^{\delta}}(c^{\delta}_{+}) \delta^{-\frac12}G_{[u_{+}^{\delta}]}((c^{\delta}_{+})^\#),
  \end{equation}
  where in the above equation $(+1)$ in the first term comes from the sign of the 'upper' increment $\sgn(\mathcal{O}_{c^{\delta}_{+}},v^{\delta})=+1 $ (i.e. coming from a deformation of the contour in the part of the graph \emph{above} $v^{\delta}$) while $(-1)$ in the second term comes from the sign of the 'lower' increment $\sgn(\mathcal{O}_{c^{\delta}_{+}},v^{\delta})=+1 $ (coming from a deformation of the contour in the part of the graph \emph{below} $v^{\delta}$). Thus \eqref{eq:vanishing-increment} vanishes.

  In conclusion, one can conclude that

  \begin{equation}\label{eq:correlation_ratio_vertical}
    \partial_y \log \bbE_{\primaldomain}^{+}[ \sigma_{u_1^{\delta}} \ldots \sigma_{u^{\delta}_n} ] = \myr[ \icomp \calA_{(u_1,\ldots,u_n)}] + o_{\delta \to 0} (1),
  \end{equation}
  where $o_{\delta \to 0} (1)$ is uniform in bounded compact subsets of $\Omega $.  \medskip
  Whenever one has the two asymptotics~\eqref{eq:correlation_ratio_horizontal} and~\eqref{eq:correlation_ratio_vertical}, it is enough to integrate them along a sequence of paths linking $b^{\delta}_1 $ to $a^{\delta}_1$, $b^{\delta}_2 $ to $a^{\delta}_2$ ... $b^{\delta}_n $ to $a^{\delta}_n$ to recover that 
  \begin{equation}
    \frac{\bbE_{\primaldomain}^{+}[\sigma_{a^{\delta}_1}\cdots\sigma_{a^{\delta}_n}]}{\bbE_{\primaldomain}^{+}[\sigma_{b^{\delta}_1}\cdots\sigma_{b^{\delta}_n}]} \underset{\delta \to 0}{\longrightarrow} \exp \big( \displaystyle \myr \big[ \int^{(a_1,\ldots,a_n)}_{(b_1,\ldots,b_n)} \sum\limits_{k=1}^{n} \calA_{(x_1,\ldots,x_n)} \dd x_k \big] \big).
  \end{equation}
  This concludes the proof.

\end{proof}

One can now use a similar scheme to prove Theorem~\ref{thm:disorder-spins}.

\begin{proof}
  Let $v^{\delta}_{2} \in \dualdomain$ be a neighbour of $u^{\delta}_2$. Repeating the previous integration formula~\ref{prop:integration-trick} to extract values near the branchings with the observables $F^{\diamond}_\delta:= \delta^{-\frac12} \bbE_{\Omega_\delta}^{+}[\sigma_{u^{\delta}_1}\sigma_{u^{\delta}_2}]^{-1}F^{\calS}_{\Omega_\delta, \{v_{1}^{\delta},u^{\delta}_2 \}} $ and $\delta^{-\frac12}G_{[v^{\delta}_{2}]} $ via a macroscopic countour $\calC^{\delta}_{r} $ at a small distance $r>0$ from $u^{\delta}_2$, one gets 
  \begin{equation}
    2\; \frac{\bbE_{\primaldomain}^{+}[\mu_{v^{\delta}_1}\mu_{v^{\delta}_2}]}{\bbE_{\primaldomain}^{+}[\sigma_{u^{\delta}_1}\sigma_{u^{\delta}_2}]}
    = \displaystyle \myi \big[ \int_{\calC^{\delta}_{r}} \delta^{-\frac12}G^{\diamond}_{[v^{\delta}_{2}]}(z)F^{\diamond}_\delta(z)\dd z \big].
  \end{equation}
  Using the asymptotics
  \begin{align*}
    & f^{\calS}_{\Omega, \varpi}= \frac{1}{\sqrt{\pi}}  \frac{e^{-\icomp\frac\pi4}}{(z-u_2)^\frac12}\calB_{\Omega}(u_1,u_2) + O \big( (z-u_2)^{\frac12} \big), \\
    & \delta^{-\frac{1}{2}}G^{\diamond}_{[v^{\delta}_{2}]}(z)
      = \frac{1}{\sqrt{\pi}} \frac{e^{\icomp \frac\pi4}}{(z-u_2)^{\frac12}} + O\big( \frac{\delta^2}{(z-v)^{\frac{5}{2}}} \big),
  \end{align*}
  one gets the result exactly as in the previous proof (taking again $r$ arbitrarily small).
  To conclude, one uses Krammers--Wannier duality to note that $\bbE_{\primaldomain}^{+}[\mu_{v^{\delta}_1}\mu_{v^{\delta}_2}] =  \bbE^{\free}_{\Omega^{\circ}_\delta}[\sigma_{v^{\delta}_1} \sigma_{v^{\delta}_2}] $.

\end{proof}

Now that once Theorem~\ref{thm:ratios_spins} is proven, one has to find the proper normalizing factors in front of the correlation ratio to prove completely Theorem~\ref{thm:spins}.
This goes by a comparaison with the full-plane lattice.
This proof is done here by induction over $n$ and recalls elements introduced in~\cite{CHI-spin}. 

\begin{proof}
  We start with the case $n = 2$.
  Define $\langle \sigma_a \sigma_{b} \rangle_{\bbC}^{+}:= |b-a|^{\frac14}$.
  In this case, \ref{thm:ratios_spins} implies that,
  \begin{align*}
    & \frac{\delta^{-\frac14} \calC^2 \bbE_{\Omega_\delta}^{+} [\sigma_{a}\sigma_{b}] }{\langle \sigma_a \sigma_{b} \rangle_{\Omega}^{+}}  = \big( \frac{\langle \sigma_a \sigma_{b'} \rangle_{\Omega}^{+}}{\langle \sigma_a \sigma_{b} \rangle_{\Omega}^{+}} \cdot \frac{\bbE_{\Omega_\delta}^{+} [\sigma_{a}\sigma_{b}]}{\bbE_{\Omega_\delta}^{+} [\sigma_{a}\sigma_{b'}]}  \big) \cdot \frac{\bbE_{\Omega_\delta}^{+} [\sigma_{a}\sigma_{b'}]}{\bbE_{\bbC_\delta}^{+} [\sigma_{a}\sigma_{b'}]} \cdot
      \frac{\bbE_{\bbC_\delta}^{+} [\sigma_{a}\sigma_{b'}]}{\delta^{-\frac14}\calC^2 \langle \sigma_a \sigma_{b'} \rangle_{\bbC}}
                                                                                                                                               \cdot \frac{ \langle\sigma_a \sigma_{b'} \rangle_{\bbC}^{+}}{\langle \sigma_a \sigma_{b} \rangle_{\Omega}^{+}}.
  \end{align*}

  Sending $\delta $ to $0$ first makes the first of the RHS on both lines converge to $1$. Sending then $b'$ to $a$ and using the fact that $ \lim_{b' \to a} \lim_{\delta \to 0} \tfrac{\bbE_{\Omega_\delta}^{+} [\sigma_{a}\sigma_{b}]}{\bbE_{\bbC_\delta}^{+} [\sigma_{a}\sigma_{b'}]} =1  $ (which is true by FKG inequality and the fact that $\calB(a,b')$ goes to $1$ as $b'$ approches $a$), it simply remains to notice that $\langle \sigma_a \sigma_{b'} \rangle_{\bbC}^{+} (\langle \sigma_a \sigma_{b'} \rangle^{+}_{\Omega})^{-1} $ goes to $1$ as $b'$ goes to $a$. The latter is true when  $\Omega = \bbH $ and holds by conformal invariance  for general domains. \newline
  For $n=1$, one notices that,
  \begin{equation}
    \frac{\bbE_{\Omega_{\delta}}^{+}[\sigma_a]^{2}}{\delta^{\frac{1}{4}}\calC^2 (\langle \sigma_a  \rangle_{\Omega}^{+})^2 } = \big( \frac{1}{(\langle \sigma_a  \rangle_{\Omega}^{+})^2} \frac{\bbE_{\Omega_{\delta}}^{+}[\sigma_a]}{\bbE_{\Omega_{\delta}}^{+}[\sigma_b]} \big) \cdot \frac{\bbE_{\Omega_{\delta}}^{+}[\sigma_a]\bbE_{\Omega_{\delta}}^{+}[\sigma_b]}{\bbE_{\Omega_{\delta}}^{+}[\sigma_a \sigma_b]  } \cdot \frac{\bbE_{\Omega_{\delta}}^{+}[\sigma_a \sigma_b]}{\delta^\frac14 \calC^2}.
  \end{equation}
  Sending $\delta $ to $0$ first the first term converges to $(\langle \sigma_a  \rangle_{\Omega}^{+} \langle \sigma_b  \rangle^{+}_{\Omega})^{-1} $ and the last term converges to $ \langle \sigma_a \sigma_b  \rangle_{\Omega}^{+} $. As $b$ approaches the boundary, the second term can be made arbitrarily close to $1$ (uniformly in $\delta $, this is a consequence of GHS inequality and the fact that $\calB_{\Omega}(a,b)$ goes to $0$ in that regime). Still when $b$ approaches the boundary, the product  $ \langle \sigma_a \sigma_b  \rangle_{\Omega}^{+}  (\langle \sigma_a  \rangle_{\Omega}^{+} \langle \sigma_b  \rangle^{+}_{\Omega})^{-1} $ is arbitrarily close to $1$.
  For $n\geq 3$, assume that the result is already proven for all $n'<n$. One decomposes
  \begin{eqnarray}
    \frac{\bbE_{\Omega_{\delta}}^{+}[\sigma_{a_1}\cdots\sigma_{a_n}]}{\calC^{n}\delta^{\frac{n}{8}} \langle \sigma_{a_1}\ldots \sigma_{a_n}  \rangle_{\Omega}^{+}}  & = &  \frac{\bbE_{\Omega_{\delta}}^{+}[\sigma_{a_1}\cdots\sigma_{a_n}]}{\bbE_{\Omega_{\delta}}^{+}[\sigma_{b}\cdots\sigma_{a_n}]} \cdot \frac{\langle \sigma_{b}\ldots \sigma_{a_n}  \rangle_{\Omega}^{+}}{\langle \sigma_{a_1}\ldots \sigma_{a_n}  \rangle_{\Omega}^{+}} \times \\
                                                                                                                                                                    &  & \frac{\bbE_{\Omega_{\delta}}^{+}[\sigma_{b}\cdots\sigma_{a_n}]}{\bbE^{+}_{\Omega_\delta}[\sigma_b] \bbE_{\Omega_{\delta}}^{+}[\sigma_{a_2}\cdots\sigma_{a_n}]} \times  \\
                                                                                                                                                                    &  & \frac{\bbE^{+}_{\Omega_\delta}[\sigma_b] }{\calC\delta^\frac18} \cdot  \frac{\bbE_{\Omega_{\delta}}^{+}[\sigma_{a_2}\cdots\sigma_{a_n}]}{\calC^{n-1}\delta^\frac{n-1}{8}} \cdot \frac{1}{\langle \sigma_{b}\ldots \sigma_{a_n}  \rangle_{\Omega}^{+}}.
  \end{eqnarray}

  Sending first $\delta $ to $0$, the RHS of the first line converges to $1$ (due to Theorem~\ref{thm:ratios_spins}).
  As $b$ approaches the boundary, the second line can be made arbitrarily close to $1$ (one side is FKG inequality while the other one is Russo--Seymour--Welsh type estimates), uniformly in $\delta $ small enough.
  Finally, the last line converges to $ \langle \sigma_{b} \rangle_{\Omega}^{+}\langle \sigma_{a_2}\ldots \sigma_{a_n}  \rangle_{\Omega}^{+} (\langle \sigma_{b}\ldots \sigma_{a_n}  \rangle_{\Omega}^{+})^{-1} $ due to the induction hypothesis.
  It is then enough to note that the last mentionned quantity can be be made arbitrarily close to $1$ provided $b$ is close enough to the boundary (this is due to the structure of continuous correlation function, see~\cite[Sec.~5]{CHI-mixed}).
\end{proof}

We now pass to the proof of convergence of Theorems~\ref{thm:single_energy} and~\ref{thm:multiple_energy}. For simplicity, we only sketch the proof of multiple-energy correlations since it can be derived clasically by induction using the Pfaffian structure of the model, up to the amount of introducing the formalism of multi-point spin-dirsorder correlator (i.e. introducing a correlator as a function depending on several corners).

\begin{proof}
  We start with the case $n=1$. One first notices within the derivation of proposition~\ref{prop:convergence_energy_observable}  that the function $F^{\dagger}_{\delta}:=  \delta^{-1} (\energysolutiondiscrete - G_{(a)} ) $ converges to $\overline{\eta_a} [ f^{\calE}_{\Omega, a} - \frac{1}{\pi} \frac{1}{z-a} ] $  \emph{uniformly} on compact subsets of $\Omega $ (since $F_{\delta}^{\dagger} $ has no singularity). In particular the convergence holds at $v_a$ and $u_a$. On the other hand, one has the expansion $f^{\calE}_{\Omega,a}(z) = \overline{\eta_a} [\frac{1}{\pi(z-a)} + \frac{1}{\sqrt{2}\pi} l_{\Omega}(a) + O(z-a)]$ near $a$, which allows to identify the limit of $F_{\delta}^{\dagger}[v_a] = \overline{\eta_a} \delta^{-1}\bbE^{+}_{\primaldomain}[\epsilon_{a^{\delta}}]$, $F_{\delta}^{\dagger}[u_a] = -\overline{\eta_a} \delta^{-1} (\bbE^{+}_{\primaldomain}[\mu_a\mu_{a-2\delta}]-\bbE^{+}_{\bbC^{\bullet}_\delta}[\mu_a\mu_{a-2\delta}])$.

  For the $n$-point energy case, one can reccursively construct as in the discrete case a multi-point fermionic correlator~\cite[Sec.~2.4]{CHI-mixed} as a function of $2r$ corners
  \begin{equation}
    G(c_1,c_2,\ldots,c_{2r}) \mapsto \eta_{c_1} \eta_{c_2} \ldots \eta_{c_{2r}} \langle \chi_{c_1} \chi_{c_2} \ldots \chi_{c_{2r}} \rangle_{\Omega_\delta},
  \end{equation}
  with proper double valuations at the corners $c_i$ and consistent choices of $\eta_c $.
  One gets that away from its diagonal (when the $c_i$ are distinct), one has
  \begin{equation}\label{eq:pfaffian_energy}
    G(c_1,c_2,\ldots,c_{2r}) = \Pf \big[ \eta_{c_k}\eta_{c_l} \langle \chi_{c_k} \chi_{c_l} \rangle_{\Omega_\delta} \big]_{1\leq k,l\leq 2r}.
  \end{equation}
  Multiple energy correlation can be recovered as values of $G$ near their diagonal (i.e. taking the corners $c_{k}\sim c_{k+r}$ to be adjacent) and the convergence of the single energy density proved in the previous lines together with the Pfaffian formula~\eqref{eq:pfaffian_energy} directly yields the result.
\end{proof}

The proof of Theorems~\ref{thm:spins_formulas} and~\ref{thm:energy_formulas} is a simple application of Theorems~\ref{thm:spins} and~\ref{thm:multiple_energy} with explicit formulas for solution to boundary value problmes, which can be found in~\cite[Sec.~7]{CHI-mixed} for the upper-half plane, and in the appendix for rectangles.

We are now able to prove Theorem~\ref{thm:rotational_invariance} that states rotational invariance of the model at criticality.

\begin{proof}
  Fix $\epsilon >0 $ and $\vec{s}$ a unitary vector. Let $\bbD^{\delta}_{R} $ be the approximation of $\bbD(0,R) $ by $\bbC_\delta $. There exist $R(\epsilon)>0$ large enough such that, uniformy in $\delta $ and $s$, 
  \begin{equation}
    1-\epsilon \leq \frac{\bbE_{\bbD^{\delta}_{R}}^{+}[\sigma_{-s}\sigma_{s}]}{\bbE_{\bbC_\delta}^{+}[\sigma_{-s}\sigma_{s}]} \leq 1+\epsilon \nonumber
  \end{equation}
  Moreover, now one can use the convergence theorem in $\bbD^{\delta}_{R} $, which states that $\bbE_{\bbD^{\delta}_{R}}^{+}[\sigma_{-s}\sigma_{s}] \bbE_{\bbD^{\delta}_{R}}^{+}[\sigma_{-1}\sigma_{1}]^{-1} $ converges~\footnote{This comes from rotational invariance of $\langle \sigma_{-s}\sigma_{s} \rangle_{\bbD_{R}}^{+} $, which comes itself from rotational invariance of the associated boundary value problem.} to $\langle \sigma_{-s}\sigma_{s} \rangle_{\bbD_{R}}^{+} (\langle \sigma_{-s}\sigma_{s} \rangle_{\bbD_{R}}^{+} )^{-1} =1 $.
  Hence, the quantity $\langle \sigma_{-s}\sigma_{s} \rangle_{\bbD^{\delta}_{R}}^{+} (\langle \sigma_{-s}\sigma_{s} \rangle_{\bbD^{\delta}_{R}}^{+} )^{-1} $ goes to $1$ as $\delta \to 0 $. Sending $\delta $  to $0$ first and then $\epsilon $ to $0$, one can deduce that $\bbE_{\bbC_\delta}^{+}[\sigma_{-s}\sigma_{s}] \bbE_{\bbC_\delta}^{+}[\sigma_{-1}\sigma_{1}]^{-1} $ is arbitrarily close to $1$ as $\delta $ goes to $0$. We conclude using the asymptotic in the vertical direction derived in Theorem~\ref{thm:crit-homogen}.

\end{proof}

\section{Orthogonal polynomials and full-plane expectation}

\label{sec:polynomials}

In this section, we discuss full-plane spin-spin correlations in the horizontal direction (taken in the infinite volume limit) for the quantum Ising model, at and below the criticality.
In the homogeneous square-grid case, those results are known since the 70's with the work of McCoy and Wu~\cite{McWu-Ising}, and were originally derived relying upon the formalism of Toeplitz determinants.
Here, we follow the strategy of~\cite{CHM-zig-zag-Ising}, where a simplification using modern techniques only using the theory of {\emph{real-valued}} orthogonal polynomials was developed.
This last mentioned method admits a generalization to the case of the quantum Ising model.
Instead of using known results for rectangular grids and then smashing the lattice (which can raise problems when exchanging limits), we prefer remaining with the formalism of disorder insertion in the semi-discrete lattice to derive results for the quantum model.

\begin{figure}[htb] \centering
  \includegraphics[scale=1.2, page=1]{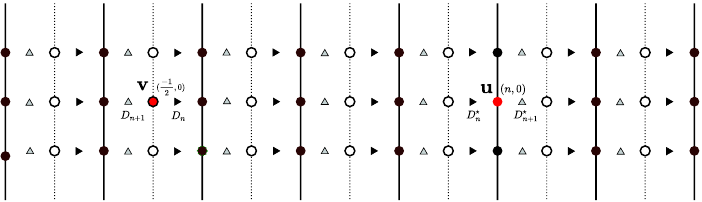}
  \caption{Local picture for the full-plane observable with two branchings (denoted in red).}
  \label{fig:Global_figure}
\end{figure}

\begin{figure}[htb] \centering
  \includegraphics[scale=1, page=1]{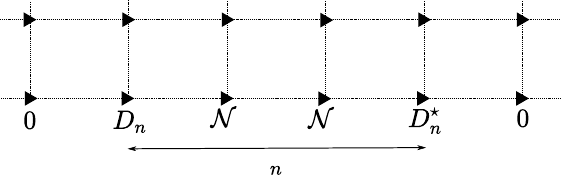} \hspace{10mm}
  \includegraphics[scale=1, page=1]{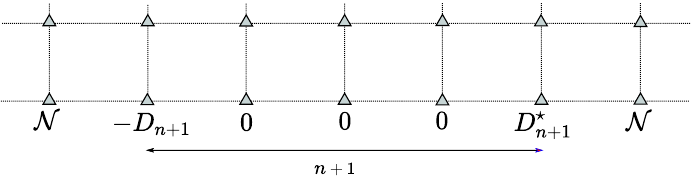}
  \caption{Symmetrized (left) and antisymmetrized (right) semi-discrete boundary value problems in the upper-half plane.}
  \label{fig:sym-anti-bvp}
\end{figure}

\subsection{Full-plane observable with two branchings}

Assume that the semi-discrete lattice on which the Ising model is so that its primal vertices coincide with~\mbox{$\bbZ \times \bbR$} and its dual vertices are~\mbox{$(\bbZ +\frac{1}{2})\times \bbR$}, see Fig.~\ref{fig:Global_figure}.
It is knwon~\cite{DCLM-universality} that at and below criticality, the quantum Ising model only admits two extremal Gibbs measures (coming from~`$+$' and `$-$' boundary conditions at infinity).
Those extremal measure are in particular translationally invariant in the infinite volume limit. Given~$n\ge 0$, we define the horizontal and next-to-horizontal correlations
\begin{align*}
  D_n := \bbE_{\bbZ \times \bbR} \big[ \sigmap_{(0, 0)} \sigmap_{(n, 0)}]\,, \qquad
  & D_n^{\star}:= \bbE^\star_{(\bbZ  \times \bbR)^\star}[\sigmad_{(-\frac{1}{2}, 0)} \sigmad_{(n-\frac{1}{2}, 0)}]\,, \\
  D_{n+1}:=\bbE_{\bbZ  \times \bbR}\big[\sigmap_{(-1, 0)} \sigmap_{(n, 0)}]\,, \qquad
  & D_{n+1}^{\star}:= \bbE^\star_{(\bbZ  \times \bbR)^\star}[\sigmad_{(-\frac{1}{2}, 0)} \sigmad_{(n+\frac{1}{2}, 0)}]\,, \\
  \widetilde{D}_{n+1} := \partial_{y} \big[ \bbE_{\bbZ  \times \bbR} \big[\sigmap_{(-1, y)} \sigmap_{(n, 0)}] \big]_{y=0}\,, \qquad
  & \widetilde{D}^{\star}_{n+1} := \partial_{y} \big[ \bbE^\star_{(\bbZ  \times \bbR)^\star}[\sigmad_{(-\frac{1}{2}, y)}\sigmad_{(n+\frac{1}{2}, y)}] \big]_{y=0}\,, 
\end{align*}
where the expectations in the second column are taken for the \emph{dual} quantum Ising model with dual parameters (see Section~\ref{sec:KW-duality}).
In particular, one can view those dual expectactions as disorder-disorder correlations for the quantum Ising model defined on $\Omega^{\bullet}$.

Let~$\mathbf{v}=(-\frac{1}{2}, 0)$ and $\mathbf{u}=(n, 0)$. Below we rely upon the \emph{full-plane} observable~$F_{[\mathbf{v}, \mathbf{u}]}$ which can be thought of as a subsequential limit of observables defined on an increasing sequence $(G_n)_{n\geq 0} $ of finite graphs exhausting the semi-discrete lattice. The existence of this pointwise (subsequential) limit is justified by the bound
\begin{equation}
  \label{eq:X-bound-onG}
  \big|\Corr{G_n}{\mu_{v(c)}\mu_{\mathbf{v}}\sigma_{u(c)}\sigma_{\mathbf{u}}}\big|\
  \le\ \Corr{G_n}{\mu_{v(c)}\mu_{\mathbf{v}}}\
  =\ \bbE_{G_n^\bullet}^{\star}[\sigmad_{v(c)}\sigmad_{\mathbf v}]\ \le\ 1, 
\end{equation}
together by the equicontinuity of correlators (so it is enough to apply the diagonal process on $(\bbZ \pm \frac{1}{4}) \times \bbQ$). The uniqueness of the limit (and hence its non dependance on the exhaustion used above) follows from Lemma~\ref{lem:uniqueness}.
Let~$[\bbZ\times\bbR;\mathbf v, \mathbf u]$ denote the double cover of the lattice $ \bbZ\times\bbR$ branching over~$\mathbf v$ and~$\mathbf u$.
We now introduce the following \emph{symmetrized} and \emph{anti-symmetrized} versions of the observable~$F_{[\mathbf{v}, \mathbf{u}]}(\cdot)$ on eastern and western corners, respectively (see Fig.~\ref{fig:Global_figure}):
\begin{align}
  \label{eq:Xsym-def}%X^{sym}(\bullet)=
  \Xsym(c) := \tfrac{e^{\icomp\frac{\pi}{4}}}{2}[F_{[\mathbf{v}, \mathbf{u}]}(c)+F_{[\mathbf{v}, \mathbf{u}]}(\bar{c})], \quad &c\in [(\bbZ -\tfrac{1}{4})\times \bbR;\mathbf v, \mathbf u], \\
  \label{eq:Xanti-def}%X^{ant}(\bullet)=
  \Xanti(c) := \tfrac{e^{-i\frac{\pi}{4}}}{2}[F_{[\mathbf{v}, \mathbf{u}]}(c)-F_{[\mathbf{v}, \mathbf{u}]}(\bar{c})], \quad &c\in [(\bbZ +\tfrac{1}{4})\times\bbR;\mathbf v, \mathbf u], 
\end{align}
where the continuous conjugation $z\mapsto \bar{z}$ on~$[(\bbZ \pm \tfrac{1}{4}) \times \bbR; \mathbf v, \mathbf u]$ is defined so it maps the segment $[\pm \tfrac{1}{4}, n \pm\tfrac{1}{4}] \times \{ 0 \} $ between $\mathbf v$ and $\mathbf u$ to itself (i.e., the conjugate of each point located over this segment is on the \emph{same} sheet of the double cover).
Once~$z \mapsto \bar{z}$ is specified in between of the branching points, it {can be `continuously' extended} to the entire {double cover} $[(\bbZ \pm \tfrac{1}{4}) \times \bbR; \mathbf v, \mathbf u]$. In particular, the points~$c$ located over the real line but outside of the segment~$[\pm \frac{1}{4}, n\pm\frac{1}{4}]$ are mapped by~$z \mapsto \overline{z}$ to their counterparts~$c^\#$ on the \emph{other} sheet of the double cover.
% (Note that there is an alternative choice of the mapping $z\mapsto\bar{z}$ for which the points~$c\in [\tfrac{1}{4};n-\tfrac{1}{4}]\times \{ 0 \}$ are mapped to their counterparts~$c^\star$ located on the other sheet of the double cover.)

We now list some basic properties of the observables~$\Xsym$ and~$\Xanti$ and show they are sufficient to characterize them uniquely.
Due to~\eqref{eq:X-bound-onG} we have
\[%\begin{equation}\label{eq:X-bound}
  \phantom{x}\;\qquad \big|\Xsym(k\!-\!\tfrac{1}{4}, s)\big|\le 1\quad \text{and}\quad \big|\Xanti(k\!+\!\tfrac{1}{4}, s)\big|\le 1\quad \text{for all}\ \ (k, s)\in\bbZ  \times \bbR.
\]%\end{equation}
Equation~\eqref{eq:mLaplacian} ensures that the observables~$\Xsym$ and~$\Xanti$ are massive harmonic away from the branching points~$\mathbf v, \mathbf u$.
In particular, one has
\begin{align}
  \label{eq:X-laplacian}
  [\Delta^{(m)}\!\Xsym]((k\!-\!\tfrac{1}{4}, s))=0\quad \text{and}\quad [\Delta^{(m)}\!\Xanti]((k\!+\!\tfrac{1}{4}, s))=0\quad \text{if}\ \ s\ne 0.
\end{align}
Moreover, the spinor property of the observable~$F_{[\mathbf v, \mathbf u]}$ coupled  with the above described choice of the conjugation ensures
\begin{align}
  \label{eq:Xsym-vanishes}
  \Xsym((k\!-\!\tfrac{1}{4}, 0))\!=\!0, \ k\not\in[0, n], \quad & [\Delta^{(m)}\!\Xsym]((k\!-\!\tfrac{1}{4}, 0))\!=\!0, \ k\in [1, n\!-\!1];\\
  \label{eq:Xanti-vanishes}
  \Xanti((k\!+\!\tfrac{1}{4}, 0))\!=\!0, \ k\in [1, n], \quad & [\Delta^{(m)}\!\Xanti]((k\!+\!\tfrac{1}{4}, 0))\!=\!0, \ k\not\in [0, n\!+\!1].
\end{align}
Finally, the definition of~$F_{[\mathbf v, \mathbf u]}$ as a correlator imply (recall  $\sigma_{u}\sigma_{u} = 1 $, $ \mu_{v}\mu_{v}=1 $)
\begin{align}
  \label{eq:Xsym-values}
  \Xsym((-\tfrac{1}{4}, 0))=D_n, \qquad & \Xsym((n\!-\!\tfrac{1}{4}, 0))=D_n^\star;\\
  \label{eq:Xanti-values}
  \Xanti((-\tfrac{3}{4}, 0))=-D_{n+1}, \qquad & \Xanti((n\!+\!\tfrac{1}{4}, 0))=D_{n+1}^\star, 
\end{align}
where we assume that these pairs of corners are located on the \emph{same sheet} of the double cover~$[(\bbZ \pm \tfrac{1}{4})\times\bbR;\mathbf v, \mathbf u]$ \emph{as viewed from the upper half-plane}; this explains why the value~$D_{n+1}$ at~$(-\tfrac{3}{4}, 0)$ comes with an opposite sign.

\begin{lem} \label{lem:uniqueness}
  (i) The uniformly bounded observable~$\Xsym$ given by~\eqref{eq:Xsym-def} is uniquely characterized by the properties~\eqref{eq:X-laplacian}, \eqref{eq:Xsym-vanishes} and its values~\eqref{eq:Xsym-values} near~$\mathbf v$ and~$\mathbf u$.

  \smallskip
  \noindent (ii) Similarly, the uniformly bounded observable~$\Xanti$ given by~\eqref{eq:Xanti-def} is uniquely characterized by the properties~\eqref{eq:X-laplacian}, \eqref{eq:Xanti-vanishes} and its values~\eqref{eq:Xanti-values} near~$\mathbf v$ and~$\mathbf u$.
\end{lem}

\begin{proof} (i) Let $X_1$ and $X_2$ be two bounded spinors satisfying~\eqref{eq:X-laplacian}, \eqref{eq:Xsym-vanishes} and~\eqref{eq:Xsym-values}.
  Let $(Z_k)_{k\geq 0}$ be the semi-discrete Brownian Motion with killing started at~$c\in[(\bbZ -\tfrac{1}{4})\times\bbR;\mathbf u, \mathbf v]$ that corresponds to the massive Laplacian~$\Delta^{(m)}$.
  This semi-discrete Brownian Motion almost surely hits the points located over the set~$\{(k\!-\!\tfrac{1}{4}, 0), k\not\in [1, n\!-\!1]\}\}$ or dies. Since the process $(X_1-X_2)(Z_k) $ is a bounded martingale with respect to the canonical filtration, the optional stopping theorem yields~$X_1(c)\!-\!X_2(c)=0$. The proof of (ii) is similar.
\end{proof}

% In the next section we construct bounded spinor satisfying~\eqref{eq:X-laplacian}--\eqref{eq:Xanti-values} by means of the Fourier transform and orthogonal polynomials techniques.
The next lemma allows one to {construct explicitely $X_{[\mathbf{v}, \mathbf{u}]}$ in Section~\ref{subsect:construction-homogen}} to get a recurrence relation between consecutive spin-spin full-plane correlations. For~$n\ge 1$, denote
\begin{equation}
  \label{eq:L-def}
  L_n := - \frac{\theta}{\theta^2 + {\theta^\star}^2} \cdot \widetilde{D}_{n}, \qquad
  L_n^\star := - \frac{\theta^{\star}}{\theta^2 + {\theta^\star}^2} \cdot \widetilde{D}^{\star}_{n}.
\end{equation}
% \begin{equation}
%   \label{eq:L-def}
%   L_n := \frac{\theta\theta^{\star}}{\theta^2+(\theta^\star)^2} \cdot[D_{n}+\theta\cdot \widetilde{D}_{n}], \qquad
%   L_n^\star:=\frac{\theta\theta^{\star}}{\theta^2+(\theta^\star)^2} \cdot[D^{\star}_{n}+\theta^{\star}\cdot \widetilde{D}^{\star}_{n}].
% \end{equation}

\begin{lem} \label{lem:laplacians}
  For each~$n\ge 1$, the following identities are fulfilled:
  \begin{align}
    \label{eq:Xsym-laplacians}
    -[\Delta^{(m)}\Xsym]((-\tfrac{1}{4}, 0))=L_{n+1}, \qquad & -[\Delta^{(m)}\Xsym]((n\!-\!\tfrac{1}{4}, 0))=L_{n+1}^\star;\\
    \label{eq:Xanti-laplacians}
    -[\Delta^{(m)}\Xanti]((-\tfrac{3}{4}, 0))=-L_n, \qquad & -[\Delta^{(m)}\Xanti]((n\!+\!\tfrac{1}{4}, 0))=L_n^\star, 
  \end{align}
  with the same choice of points on the double covers~$[(\bbZ \pm \tfrac{1}{4})\times\bbZ ;\mathbf v, \mathbf u]$ as above. %in~\eqref{eq:Xsym-values},~\eqref{eq:Xanti-values}.
  If~$n=0$, the identities~\eqref{eq:Xsym-laplacians} should be replaced by~$-[\Delta^{(m)}\Xsym]((-\tfrac{1}{4}, 0))=L_1\!+\!L_1^\star$ while~\eqref{eq:Xanti-laplacians} hold with $L_0:=\frac{\theta^{\star}}{\theta^2 + {\theta^{\star}}^{2}}$ and $L_0^\star:=\frac{\theta}{\theta^2 + {\theta^{\star}}^{2}}$.
\end{lem}

\begin{proof}
  We proof the first item. Due to the symetrization procedure (which amounts to take a cut of the double cover in the horizontal axis, outside of the branchings) the value of $\Xsym $ at $(-\frac{5}{4}, 0)$ vanishes. Thus one can reapply the strategy of the proof of~\eqref{eq:Psi_mLaplacian} but this time, and $[\Delta^{(m)}\Xsym]$ equals $e^{\icomp \frac{\pi}{4}} \frac{\theta\theta^{\star}}{\theta^2 + {\theta^{\star}}^{2}} F_{[\textbf{v}, \textbf{u}]}(-\frac{5}{4}, 0)$. Applying the equation~\eqref{eq:dual_two_term} at $c=(-\frac{5}{4}, 0) $ gives

  Applying the equation~\eqref{eq:Chi_dy2} at $c=(-\frac{3}{4}, 0) $    
  \begin{equation}
    \tilde{D}_{n+1} =\sgn(\calO, u) \cdot \theta^\star \braket{ \chi_{(-\frac{5}{4}, 0)} \calO }. \nonumber
  \end{equation}
  The check of the sign $\sgn(\calO, u) = -1 $ is left to the reader. The other items are proved in a similar maner.
\end{proof}

\subsection{Construction via the Fourier transform and orthogonal polynomials} \label{subsect:construction-homogen}
In this subsection, we construct explicitely in Lemma~\ref{lem:solution-sym} and Lemma~\ref{lem:solution-anti} two bounded functions satisfying the properties~\eqref{eq:X-laplacian}--\eqref{eq:Xanti-values} using Fourier transform and orthogonal polynomials techniques.
Due to the uniqueness statement proved in Lemma~\ref{lem:uniqueness}, these solutions must coincide with~$\Xsym$ and~$\Xanti$.
Instead of the double covers~$[(\bbZ \pm\tfrac{1}{4})\times\bbR;\mathbf u, \mathbf v]$, we work only in the upper half-plane~$\bbZ \times \bbR_{+}$ only (the link between the two setups is fully explained in Lemma~\ref{lem:H-to-dbl-cover}).

Given a function~$V:\bbZ \times \bbR \to \bbR$, we keep the definition~\eqref{eq:mLaplacian} for the massive Laplacian~$[\Delta^{(m)}V](k, s)$ when~$s>0 $ and introduce the values
\begin{equation} \label{eq:N-def}
  [\calN V](k, 0):= -\frac12 \left[ \partial_y^+ V(k, 0)  + \partial_y^- V(k, 0) \right], 
\end{equation}
% \begin{equation} \label{eq:N-def}
%   [\calNV](k, 0):=  V(k, 0)- \frac{\partial^{+}_{yy} V(k, 0)}{\theta^{2}+{\theta^{\star}}^{2}}-\frac{\theta\theta^{\star}}{\theta^{2}+{\theta^{\star}}^{2}}[V(k\!-\!1, 0)+V(k\!+\!1, 0)], 
% \end{equation}
where $\partial_y^+$ and $\partial_y^-$ stand respectively for the vertical partial derivative in the upper and the lower half planes.
The operator $\calN$ should be understood as a version of the normal second derivative of~$V$ at $(k, 0)$, thus leaving the functions at stake in their respective half-planes. We now define two intermediate explicit problems~$\mathbf{[P^{sym}_n]}$ and~$\mathbf{[P^{anti}_n]}$ whose solutions are useful to derive a reccurence relation on spin-spin expectation in the full-plane.
As a consequence of the uniqueness Lemma~\ref{lem:uniqueness}, solving those problems is equivalent to construct the functions~$X^\mathrm{sym}_{[\mathbf{v}, \mathbf{u}]}$ and $X^\mathrm{anti}_{[\mathbf{v}, \mathbf{u}]}$, respectively; see also Fig.~\ref{fig:sym-anti-bvp}.

\smallskip

\begin{itemize}
  \item $\mathbf{[P^{sym}_n]}:$ given~$n\ge 1$, construct a bounded function~$V:\bbZ \times\bbR\to\bbR$ such~that the following conditions are fulfilled:
  \begin{align*}
    [\Delta^{(m)}V](k, s)=0~\text{if}~s>0;\qquad & [\calN V](k, 0)=0~\text{for}~k\in \llbracket1, n\!-\!1\rrbracket;\\
    V(k, 0)=0~\text{for}~k\not\in  \llbracket 0, n \rrbracket;\qquad & V(0, 0)=D_n\ \ \text{and}\ \ V(n, 0)=D_n^\star\, .
  \end{align*}
  \item $\mathbf{[P^{anti}_{n+1}]}:$ given~$n\ge 0$, construct a bounded function~$V:\bbZ \times  \bbR \to\bbR$ such~that the following conditions are fulfilled:
  \begin{align*}
    \quad\qquad [\Delta^{(m)}V](k, s)=0~\text{if}~s>0;\qquad & [\calN V](k, 0)=0~\text{for}~k\not\in \llbracket0, n\!+\!1\rrbracket;\\
    V(k, 0)=0~\text{for}~k\in \llbracket 1, n\rrbracket;\qquad & V(0, 0)=-D_{n+1};\quad V(n\!+\!1, 0)=D_{n+1}^\star\, .
  \end{align*}
\end{itemize}

\begin{lem}\label{lem:H-to-dbl-cover}
  Assume that ~$V^\mathrm{sym}_n$  (resp., $V^\mathrm{anti}_{n+1}$) is a solution to $\mathrm{[P^{sym}_n]}$ (resp., $\mathrm{[P^{anti}_{n+1}]}$). Then, the following holds:
  \begin{align}
    \label{eq:Vsym-laplacians}
    [\calN V^\mathrm{sym}_n](0, 0)=L_{n+1}, \qquad & [\calN V^\mathrm{sym}_n](n, 0)=L_{n+1}^\star;\\
    \label{eq:Vanti-laplacians}
    [\calN V^\mathrm{anti}_{n+1}](0, 0)=-L_n, \, \qquad & [\calN V^\mathrm{anti}_{n+1}](n\!+\!1, 0)=L_n^\star.
  \end{align}
\end{lem}

\begin{proof}
  Take a section of the double cover~$[(\bbZ \pm \frac{1}{4})\times\bbR;\mathbf v, \mathbf u]$ with an horizontal cut along the segment $[-\frac{1}{2}, n] \times \{ 0 \}$ for the problem~$\mathrm{[P^{sym}_n]}$ and outside the segment $[-\frac{1}{2}, n] \times \{ 0 \}$ (in the horizontal axis) for the problem~$\mathrm{[P^{anti}_{n+1}]}$.
  Define two functions respectively on the right and left corners of the lattice by
  \[
    V^\mathrm{sym}_{[\mathbf v, \mathbf u]}((\pm k\!-\!\tfrac{1}{4}, s))
    := V^\mathrm{sym}_n(k, s)
    \qquad V^\mathrm{anti}_{[\mathbf v, \mathbf u]}((\pm k\!+\!\tfrac{1}{4}, s))
    :=\pm V^\mathrm{anti}_{n+1}(k, s).
  \]
  Since both functions vanish on their respective cuts, they can be viewed as bounded spinors on the double covers~$[(\bbZ \pm \frac{1}{4})\times\bbR;\mathbf v, \mathbf u]$, which satisfy the entire set of conditions described in~\eqref{eq:X-laplacian}--\eqref{eq:Xanti-values}.
  Using the uniqueness Lemma~\ref{lem:uniqueness}, this allows to identify~$\Xsym=V^\mathrm{sym}_{[\mathbf v, \mathbf u]}$ and~$V^\mathrm{anti}_{[\mathbf v, \mathbf u]}=V^\mathrm{anti}_{[\mathbf v, \mathbf u]}$.
  The identities~\eqref{eq:Vsym-laplacians}, \eqref{eq:Vanti-laplacians} are now easily deduced from~\eqref{eq:Xsym-laplacians}, \eqref{eq:Xanti-laplacians} and the definition~\eqref{eq:N-def}.
\end{proof}

Let~$V$ be a solution to the problem~$\mathrm{[P^{sym}_n]}$.
We are going to construct $V$ explicitly, starting first with a heuristic and more intuitive argument. Assume for a moment that all the Fourier series
\begin{equation}
  \textstyle \widehat{V}_s(e^{\icomp t}):=\sum_{k \in \bbZ} V(k, s)e^{\icomp kt}, \quad s\ge 0, \quad t\in [0, 2\pi], \nonumber
\end{equation}
are well-defined. Using the massive harmonicity condition ~$[\Delta^{(m)}V](k, s)=0$ for~$s \neq 0$, this implies (provided the serie/derivation exchange is proper) that
\begin{equation}\label{eq:Fourier-series-reccurence}
  \big[ 1 - \frac{2\theta\theta^{\star}}{\theta^2+{\theta^{\star}}^{2}}\cos t \big] \cdot \widehat{V}_s(e^{\icomp t})
  \ =\ \frac{1}{\theta^2+{\theta^{\star}}^{2}} \widehat{V}^{''}_s(e^{\icomp t}), 
\end{equation}
where the derivative in the above equation is taken with respect to the variable~$s$.
A general solution to {the differential equation}~\eqref{eq:Fourier-series-reccurence} is a linear combination of the functions $e^{\pm s w(t)}$ where $w(t) := \OPweight$ solves the quadratic equation, 
\[
  1 - \frac{2\theta\theta^{\star}}{\theta^2 + {\theta^{\star}}^{2}} \cos t  = \frac{w(t)^2}{\theta^2+{\theta^{\star}}^{2}}.
\]
Let $\OPweight$ to be the non-negative root of the above equation, i.e., 
\[
  \OPweight = \sqrt{ (\theta^2 + {\theta^{\star}}^{2}) - 2 \theta \theta^\star \cos t } = |\theta - \theta^\star e^{\icomp t}|.
\]
At level~$s=0$, the function $\widehat{V}_0(e^{\icomp t})=Q_n(e^{\icomp t})$ is an unknown trigonometric polynomial of degree $n$.
Since we want Fourier series $\widehat{V}_s$ with \emph{bounded} coefficients, we are tempted to state that $\widehat{V}_s(e^{\icomp t}) = Q_n(e^{\icomp t}) \cdot e^{- s\OPweight}$ for~$s \geq 0$.
A direct computation at level $s=0$ proves that
\begin{align}
  \label{eq:NV-fourier}
  \textstyle \sum_{k\in\bbZ }[\calN V](k, 0)e^{\icomp kt}\ & =\ \OPweight Q_n(e^{\icomp t}).
  % \label{eq:w-def}
  % w(t;\theta\!, \theta^{\star})\
  %   % &=\ 1-\sin\theta^{\mathrm{h}}\cos\theta^{\mathrm{v}}\cos t- \cos\theta^{\mathrm{h}}\sin\theta^{\mathrm{v}}\cdot y_-(t)\\
  % & :=\ 1-\frac{2\theta\theta^{\star}}{\theta^2+{\theta^{\star}}^{2}}\cos t.
\end{align}
The crucial point here is that the LHS of the previous equation \emph{should not containany monomials of the family}~$ \{e^{\icomp t}, \dots, e^{\icomp (n-1)t} \}$, which simply reads as an orthogonality condition for~$Q_n(e^{\icomp t})$.
Let us make now rigourous the previous analysis to construct and identify the unique solution to the problem $\mathrm{[P^{sym}_n]}$.

\begin{lem}%[Solution to the problem~$\mathrm{[P^{sym}_n]}$]
  \label{lem:solution-sym}
  Let~$n\ge 1$. If a trigonometric polynomial \mbox{$Q_n(e^{\icomp t})=D_n\!+\ldots+\!D_n^{\star}e^{\icomp nt}$} of degree~$n$ has prescribed free ($D_n)$ and leading ($D^{\star}_n$) coefficients and is orthogonal to the family $\{e^{\icomp t}, \ldots, e^{\icomp (n-1)t}\}$ with respect to the measure $\OPweight \frac{dt}{2\pi}$ on the unit circle, then the discrete function defined by
  \[
    \textstyle V(k, s):=\frac{1}{2\pi}\int_{-\pi}^{\pi} e^{-ikt} Q_n(e^{\icomp t}) e^{-s \OPweight} \dd t
  \]
  is uniformly bounded and solves $\mathrm{[P^{sym}_n]}$ in the upper half-plane. Moreover, 
  \begin{equation}
    \label{eq:<Qn>=L-sym}
    \langle Q_n, 1\rangle_{\frac{w}{2\pi}dt} =L_{n+1} \quad\text{and}\quad \langle Q_n, e^{\icomp nt}\rangle_{\frac{w}{2\pi}dt}=L_{n+1}^\star, 
  \end{equation}
  where the scalar product is taken with respect to the same measure on the unit circle.
\end{lem}
\begin{proof} The values of~$V(k, s)$ defined above are uniformly bounded since~$| e^{-s \OPweight}| \le 1$. The massive harmonicity property~$[\Delta^{(m)}V](k, s)=0$ for~$s>0$ is straightforward and the properties required for~$V(k, 0)$ and~$[\calN F](k, 0)$ follow directly from the assumptions made on~$Q_n$. The identities in~\eqref{eq:Vsym-laplacians} give~\eqref{eq:<Qn>=L-sym}.
\end{proof}

On can perform a similar construction to solve~$\mathrm{[P^{anti}_{n+1}]}$, {see Fig.~\ref{fig:sym-anti-bvp}.} The only difference is that at level~$s=0$, it is now required that
$\widehat V_0(e^{\icomp t})$ does not contain monomials~$e^{\icomp t}, \dots, e^{\icomp (n+1)t}$ while
\begin{equation}
  \label{eq:wV=poly-anti}
  \textstyle \sum_{k\in\bbZ}[\calN V](k, 0)e^{\icomp kt}\ =\ w(t;\theta\!, \theta^{\star}) \widehat V_0(e^{\icomp t}) \ =\ -L_n+\ldots+L_n^\star e^{\icomp (n+1)t}
\end{equation}
is a trigonometric polynomial of degree~$n\!+\!1$. In other words, the polynomial in the RHS has to be ortgonal to~$\{e^{\icomp t}, \dots, e^{\icomp nt}\}$ with respect to the weight
\begin{equation}
  \label{eq:whash-def}
  w^\# (t;\theta\!, \theta^{\star})\; :=\; (w(t;\theta\!, \theta^{\star}))^{-1}, \qquad t\in [0, 2\pi].
\end{equation}
provided that~$w^\#$ is integrable on the unit circle. One can easily see from its definition that this integrability condition holds if and only if~$\theta\neq \theta^{\star}$. The special case $\theta=\theta^{\star} $ corresponds the model at criticality and will require modification as explained in Section~\ref{subsect:crit-homogen}.
\begin{lem}%[Solution to the problem~$\mathrm{[P^{anti}_{n+1}]}$]
  \label{lem:solution-anti}
  Let~$n\ge 0$ and assume that~$\theta\neq \theta^{\star}$. If a trigonometric polynomial $Q^\#_{n+1}(e^{\icomp t})=-L_n+\ldots+L_n^{\star}e^{\icomp (n+1)t}$ of degree~$n+1$ has prescribed free ($-L_n)$ and leading ($L^{\star}_n$) coefficients is orthogonal to the family $\{e^{\icomp t}, \ldots, e^{\icomp nt}\}$ with respect to the measure~$w^\#\!(t;\theta\!, \theta^{\star})\frac{dt}{2\pi}$, then the function defined by
  \begin{equation}
    \label{eq:solution-anti}
    \textstyle V(k, s):=\frac{1}{2\pi}\int_{-\pi}^{\pi} e^{-ikt} Q^\#_{n+1}(e^{\icomp t}) e^{-s \OPweight} w^\# (t;\theta\!, \theta^{\star}) dt
  \end{equation}
  is uniformly bounded and solves $\mathrm{[P^{anti}_{n+1}]}$. Moreover, 
  \begin{equation}
    \label{eq:<Qn>=L-anti}
    \langle Q^\#_{n+1}, 1\rangle_{\frac{w^\#}{2\pi}dt} = -D_{n+1} \quad\text{and}\quad \langle Q^\#_{n+1}, e^{\icomp (n+1)t}\rangle_{\frac{w^\#}{2\pi}dt}= D_{n+1}^\star, 
  \end{equation}
  where the scalar product is taken with respect to the same measure on the unit circle.
\end{lem}
\begin{proof}
  The proof is identical to the one of Lemma~\ref{lem:solution-sym}.
\end{proof}

\subsection{Horizontal spin-spin correlations below criticality} \label{subsect:below-crit-homogen}
In this section, we derive the asymtotics of the horizontal spin-spin correlations~$D_n$ (as~$n\to \infty$) by combining Lemmas~\ref{lem:solution-sym} and~\ref{lem:solution-anti}. We assume that~$\theta<\theta^{\star}$ and recall that~$D_n^{\star}\to 0$ as~$n\to \infty$. The latter claim is classical, and can be derived from the monotonicity of~$D_n$ with respect to the temperature together with the fact that~$D_n=D_n^\star\to 0$ as~$n\to\infty$ at cricicality, result which is discussed in the very next section. Even without the next section, one can simply mention RSW estimates at criticality to conclude.

\begin{theo}\label{thm:magnetization_below_criticality}
  Let~$\theta < \theta^{\star} $. The spontaneous magnetization~$\calM(\theta, \theta^{\star})$ of the quantum Ising model defined below satisfies
  \begin{equation}
    \label{eq:M=magnetization_below_criticality}
    \calM(\theta, \theta^{\star})\ := \lim\limits_{n\to \infty} D_n^{1/2}
    \ =\ (\theta^2 + {\theta^\star}^2)^{-\frac12} (1 - (\tfrac{\theta}{\theta^\star})^2)^{\frac18}.
  \end{equation}
\end{theo}
Since we developped a formalism fitted to~\cite{CHM-zig-zag-Ising}, this proofs goes as in Theorem 3.6. The next lemma has exactly the same proof as in~\cite{CHM-zig-zag-Ising}.

\begin{lem}\label{lem:reccurence_correlation_Verbrusky}
  Set $\Phi_n(z):=z^n+\ldots -\alpha_{n-1}$ be the $n$-th unitary orthogonal polynomial with respect to the measure~$w(t;\theta, \theta^\star)\frac{dt}{2\pi}$ on the unit circle.
  We also set~$\Phi_n^\star(z) :=z^{n}\Phi_n(z^{-1})=-\alpha_{n-1}z^n+\ldots+1$ the polynomial reciprocal to $\Phi_n$.
  Moreover, write $\beta_n :=  \| \Phi_n \|^2 = \braket{\Phi_n, e^{\icomp nt}} = \braket{\Phi^\star_n, 1}$ and $\beta_n^\# := \| \Phi_n^\# \|^2$, where $\Phi_n^\#$ is the $n$-th unitary orthogonal polynomial with respect to the measure defined in~\eqref{eq:whash-def}.
  Then, 
  \begin{align*}
    % \label{eq:Dn-to-Ln+1-squares}
    L_{n+1}^2-(L_{n+1}^\star)^2 & = \beta_n\beta_{n-1}\cdot (D_n^2-(D_n^\star)^2)\quad \text{for}\ \ n\ge 1, \\
    % \label{eq:Ln-to-Dn+1-squares}
    D_{n+1}^2-(D_{n+1}^\star)^2& = \beta^\#_{n+1}\beta^\#_n\cdot (L_n^2-(L_n^\star)^2)\quad \text{for}\ \ n\ge 0.
  \end{align*}
  As a consequence, by applying the above two formulas recursively, we also have, 
  \begin{equation}
    \label{eq:Dn-square-diff}
    D_{2m+1}^2-(D_{2m+1}^\star)^2 = \prod_{k=0}^{2m+1}\beta^\#_k\cdot\prod_{k=0}^{2m-1}\beta_k\cdot (L_0^2-(L_0^\star)^2), 
  \end{equation}
  where~$L_0^2 - (L_0^\star)^2 = \frac{ \theta^2 - {\theta^\star}^2 }{ ( \theta^2 + {\theta^\star}^2 )^2 }$.
\end{lem}

Now we can complete the proof of Theorem~\ref{thm:magnetization_below_criticality}.
\begin{proof}
  Recall that~$D^\star_{2m+1}\to 0$ as~$m\to \infty$. It remains to apply the Szegő theory (e.g., see~\cite[Section~5.5]{Szego-Grenander-book} or~\cite[Theorems~8.1~and~8.5]{Simon-OPUC-one-foot}) to the weights $w(t; \theta\!, \theta^{\star})$ and $ w^\#(t;\theta\!, \theta^{\star})$.
  Let $q = \theta / \theta^\star$.
  We can write
  \begin{align*}
    w(t; \theta\!, \theta^{\star}) \ = \theta^* w_{q}(t), \qquad & \text{where}\ 
                                                                   w_q(t) = |1 - q e^{\icomp t}|, 
  \end{align*}
  Since~$w^\#(t;\theta\!, \theta^{\star}) = (w(t;\theta\!, \theta^{\star}))^{-1}$, we have
  \begin{equation}
    \label{eq:product_beta_beta_sharp}
    \textstyle \lim_{m\to\infty} \prod_{k=0}^{2m+1}\beta^\#_k\cdot\prod_{k=0}^{2m-1}\beta_k\ =\ {\theta^*}^{-2}\cdot G^2, 
  \end{equation}
  where
  \begin{align*}
    \log G & = \frac{1}{4\pi} \iint_{\bbD} \frac{\dd}{\dd z} \Big| \log(1 - q z) \Big|^2 \dd A(z)
             = \sum_{k\geq 1} \frac{q^{2k}}{4k} = - \frac14 \log (1-q^2).
  \end{align*}
  Moreover, 
  \[
    L_0^2 - (L_0^\star)^2
    = \frac{\theta^2 - {\theta^\star}^2}{ (\theta^2 + {\theta^\star}^2)^2 }
    = \frac{1}{{\theta^\star}^{2}} \frac{1-q^2}{(1+q^2)^2} .
  \]
  Putting all the factors together, one gets~\eqref{eq:M=magnetization_below_criticality}.

\end{proof}

\begin{rem}\label{rem:energy_polynomials}
  In the proof of lemma~\ref{lem:reccurence_correlation_Verbrusky}, one uses the identities (3.24) and (3.26) from~\cite{CHM-zig-zag-Ising}, which are
  \begin{align}
    \left[
    \begin{array}{c}
      L_{n+1}^\star \\ L_{n+1}
    \end{array}
    \right]
    & =\ \beta_{n-1} \left[
      \begin{array}{cc}
        1 & \alpha_{n-1} \\
        \alpha_{n-1} & 1
      \end{array}
                       \right]
                       \;
                       \left[
                       \begin{array}{c}
                         D_n^\star \\
                         D_n
                       \end{array}
    \right],    \label{eq:Dn-to-Ln+1} \\
    \left[
    \begin{array}{c}
      D_{n+1}^\star \\ -D_{n+1}
    \end{array}
    \right]
    & =\ \beta^\#_n \left[
      \begin{array}{cc}
        1 & \alpha^\#_n \\
        \alpha^\#_n & 1
      \end{array}
                      \right]
                      \;
                      \left[
                      \begin{array}{c}
                        L_n^\star \\
                        -L_n
                      \end{array}
    \right].    \label{eq:Ln-to-Dn+1}
  \end{align}
  We notice that, \eqref{eq:Ln-to-Dn+1} applied to $n=0$ provides the formula,
  \[
    D_1= \frac{\beta_0^\#}{\theta^2 + (\theta^\star)^2} \cdot [\theta^\star-\alpha_0^\#\theta],
  \]
  which is the \emph{energy density} of the quantum Ising model.
\end{rem}

Note that Equation~\eqref{eq:Dn-square-diff} becomes useless at criticality since $\theta = \theta^\star $ and thus $L_0 = L_0^\star$.
The asymptotic analysis needs to be treated in differently as explained in the following section.

\subsection{Asymptotics of horizontal correlations~$D_n$ as~$n\to\infty$ at criticality} \label{subsect:crit-homogen}

We now work at the critical and isotropical point of the model, i.e., $\theta = \theta^{\star}$. We generalize to the quantum Ising model the classical result by Mc-Coy and Wu that states that spin-spin correlations~$D_m$ decay like~$m^{-1/4}$ when $m\to \infty $. The power-law type decay can be deduced by usual RSW arguments, but finding the correct asymptotic requires some additional work, which we perform now. We keep following the formalism of~\cite{CHM-zig-zag-Ising}. The proof of the next theorem goes as the one of Theorem 3.9 and we only give its outlines.
\begin{theo} \label{thm:crit-homogen}
  Let~$\calC_\sigma:=2^{\frac{1}{6}}e^{\frac{3}{2}\zeta'(-1)}$. For the critical quantum Ising model, one has the asymptotic as $ m\to\infty$
  \begin{equation}\label{eq:crit-homogen}
    D_m\ \sim\ \calC_\sigma^2\cdot (2m)^{-1/4}.
  \end{equation}
\end{theo}

\begin{proof}
  A straightforward computation shows that
  \begin{equation*}
    w_{c}(t)=2(\sin\tfrac{1}{2}t)^2\, .
  \end{equation*}  
  thus the weight $w_{c}^\#:=w_{c}^{-1}$ is not integrable anymore and the arguments of the proof of Theorem~\ref{thm:magnetization_below_criticality} have to be modified. We still can remark that $D_n=D_n^\star$, $L_n=L_n^\star$ by Kramers--Wannier duality. Still one can keep using~\eqref{eq:Dn-to-Ln+1}. This allows to switch to the framework of orthogonal polynomials on the \emph{the segment} $[-1, 1]$ instead of the working with weights defined on the unit circle. Let
  \begin{equation}
    \label{eq:w-real}
    \overline{w}_{c}(x)\ :=\ [\, 1-x^2\, ]^{1/2}\,, \quad x\in [-1;1], 
  \end{equation}
  and ~$P_n(x)=x^n+\dots$ be the unitary orthogonal polynomial of degree~$n$ on the segment~$[-1, 1]$ for the weight $\overline{w}_{c}(x)$. The trigonometric polynomial
  \[
    Q_n(e^{\icomp t})\ :=\ D_n \cdot e^{\frac{1}{2}\icomp nt}\cdot 2^nP_n(\cos\tfrac{1}{2}t)
  \]
  fits the construction given in Lemma~\ref{lem:solution-sym} to solve~$[\mathrm{P^{sym}_n}]$ (it has the proper orthogonality relations with prescribed proper free and leading coefficients $D_n$ (recall that  $P_n$ is unitary). The formula~\eqref{eq:<Qn>=L-sym} gives for $n \geq 1$
  \begin{equation}\label{eq:Dn-to-Ln+1-real}
    L_{n+1}=\frac{1}{2\pi}\int_{-\pi}^{\pi} Q_{n}(e^{\icomp t})w_c(t)dt= \pi^{-1}2^{2n}\cdot\|P_n\|^2_{\overline{w_c}dx}\cdot D_n.
  \end{equation}  
  An analoguous computation for $n=0$ gives (using~$D_0=1$ and due to the modification required in Lemma~\ref{lem:laplacians} for the case $n=0$),
  \begin{equation}
    \label{eq:D0-to-L1-real}
    \textstyle 2L_1\ =\ 2\pi^{-1}\int_{-1}^1P_0(x)\overline{w_c}(x)dx = 2\pi^{-1}\cdot\|P_0\|^2_{\overline{w_c}dx}
  \end{equation}
  We construct similarly a solution to~$\mathrm{[P^{anti}_{n+1}]}$ treated in Lemma~\ref{lem:solution-anti} in the supercritical regime. Denote~$P_n^\#(x)$ be the unitary orthogonal polynomial of degree~$n$ on the segment~$[-1, 1]$ for the weight,
  \begin{equation}
    \label{eq:whash-real}
    \overline{w}_{c}^\#(x)\ :=\ [\, 1-x^2\, ]^{-1/2}\,, \quad x\in [-1, 1], 
  \end{equation}
  and,
  \[
    Q^\#_{n+1}(e^{\icomp t})\ :=\ L_n\cdot (e^{\icomp t}\!-\!1)e^{\frac{1}{2}\icomp nt}\cdot 2^nP^\#_n(\cos\tfrac{1}{2}t).
  \]
  The formula~\eqref{eq:solution-anti} indeed provides a solution to the problem~$\mathrm{[P^{anti}_{n+1}]}$, and now the product~$(e^{\icomp t}-1)w_{c}^\#(t)$ \emph{becomes integrable} in unit circle as the additional factor $(e^{\icomp t}\!-\!1)$ removes the singularity of~$w_c^\#$ at~$t=0$. Finally, the computation~\eqref{eq:<Qn>=L-anti} is still valid and implies for $n\geq 0$,
  \begin{equation}
    \label{eq:Ln-to-Dn+1-real}
    D_{n+1}=-\frac{1}{2\pi}\int_{-\pi}^{\pi} Q^\#_{n+1}(e^{\icomp t})w_{c}(t)^{-1}dt= \pi^{-1}2^{2n}\cdot \|P^\#_n\|^2_{\overline{w}_{c}^\#dx}\cdot L_n.
  \end{equation}

  Since~$L_0=1$ (see Lemma~\ref{lem:laplacians}). Using the reccurence relations~\eqref{eq:D0-to-L1-real}, \eqref{eq:Dn-to-Ln+1-real} for~$n=1, \dots, m\!-\!1$, and~\eqref{eq:Ln-to-Dn+1-real} for~$n=0, \dots, m$, one gets
  \begin{align}
    \label{eq:DmDm+1-crit}
    D_{m+1}D_{m}\ &=\ \pi^{-2m-1} 2^{2m^2}\prod\nolimits_{k=0}^{m-1}\|P_k\|^2_{\overline{w_c}dx} \cdot \prod\nolimits_{k=0}^{m}\|P^\#_k\|^2_{\overline{w_{c}}^\#dx}\,, 
  \end{align}
  where the weights~$w_c(x)$ and~$w^{\#}_{c}(x)$ on~$[-1, 1]$ are given by~\eqref{eq:w-real} and~\eqref{eq:whash-real}.
  This is a classical problem of orthogonal theory, and going back to the unit circle with a  $|t|$-type singularity of the weights appears when~$e^{\icomp t}=1$. General results (accounted, e.g., in~\cite{Deift-Its-Krasovsky-Annals}) ensure that,
  \[
    D_{m+1}D_{m}\ \sim\ 2^{2/3}e^{6\zeta'(-1)}(2m)^{-1/2}, \quad m\to\infty, 
  \]
  To conclude, it is enough to prove that~$D_{m+1}\sim D_m$ as~$m\to\infty$. This can be seen in the convergence of horizontal correlation ratios in a finite box~\ref{thm:ratios_spins} and Russo-Seymour-Welsh estimates to pass to the full-plane. 
\end{proof}

\subsection{Asymptotics of correlations above criticality}

\begin{prop}
  The horizontal correlation lenght equals $\xi = \frac{1}{2}\log(\frac{\theta}{\theta^{\star}}).$
\end{prop}

\begin{proof}
  One first notices that using~\eqref{eq:Dn-to-Ln+1} and~\eqref{eq:Ln-to-Dn+1}, one has, for all $n$,
  \begin{equation}\label{eq:reccursion_rule_supercritical}
    D_{n+2} = q_n D_n + r_n,
  \end{equation}
  with $q_n = \beta^\#_{n+1}\beta_{n-1}(1-\alpha^\#_{n+1}\alpha_{n-1})$ and $r_n = \beta^\#_{n+1}\beta_{n-1}(\alpha_{n-1}-\alpha^\#_{n+1})D_n^{\star}  $.
  The above formula gives, by induction, the following formula,
  \begin{equation}
    D_{2n} = \bigg( D_0 + \sum\limits_{j=0}^{n-1} \frac{r_j}{\prod\nolimits_{k=0}^{j}q_k} \bigg) \prod \limits_{m=0}^{n-1} q_m.
  \end{equation}

  Note that $\alpha^\#_{n+1}$ and $\alpha_{n-1}$ converge to $0$ exponentially fast~\cite[Theorem~9.1]{Simon-OPUC-one-foot} with the same exponential rate $\lim_{n \to \infty} |\alpha_n|^{1/n} = \lim_{n \to \infty} |\alpha_n^\#|^{1/n} = R^{-1}$, with $R = \theta^\star / \theta > 1$ given by the radius of the disk on which $w$ has an analytic continuation.
  We also know that $D_n$ converges to 0 as $n$ tends to $\infty$.
  Moreover, using~\eqref{eq:product_beta_beta_sharp}, when $n$ tends to infinity, we have that the product $\prod_{m=0}^{n-1} q_m$ converges to a positive constant and that $\beta_{n+1}^\# \beta_{n-1}$ tends to 1.
  The aforementioned convergences all together imply that the factor
  \[
    D_0 + \sum\limits_{j=0}^{n-1} \frac{r_j}{\prod\nolimits_{k=0}^{j}q_k}
  \]
  goes to 0 as $n \to \infty$, and can be consequently rewritten as the remainder of a converging infinite series,
  \[
    - \sum\limits_{j=n}^{\infty} \frac{r_j}{\prod\nolimits_{k=0}^{j}q_k}.
  \]
  The ratio of the two consecutive terms in the above remainder is given by,
  \begin{equation}
    \label{eq:D0-remainder}
    \frac{r_{j+1} / q_{j+1}}{r_j}
    = \frac{ \alpha_j - \alpha_{j+2}^\# }{ \alpha_{j-1} - \alpha_{j+1}^\# }
    \frac{1}{ \beta_{j+1}^\# \beta_{j-1} }
    \frac{1}{1 - \alpha_{j+2}^\# \alpha_j},
  \end{equation}
  which tends to $R^{-1}$ as explained below. Thus, the whold remainder~\eqref{eq:D0-remainder} converges exponentially fast at rate $R^{-1}$.

  We prove now the annouced fact that the given ratio converges to $R^{-1}$. The last two terms in~\eqref{eq:D0-remainder} converge to $1$ exponentially fast. Let $D(z)=\theta^{\frac{1}{2}}(1-\frac{1}{R}z)^{\frac{1}{2}} $ and $D^\#(z)=\theta^{-\frac{1}{2}}(1-\frac{1}{R}z)^{-\frac{1}{2}} $ the Szego functions naturally associated to the weights $w $ and $w^\# $.  Let $D^{-1}(z)=\sum\limits_{j=0}^{\infty} d_{j,-1}z^{j} $ and $(D^\#)^{-1}(z)=\sum\limits_{j=0}^{\infty} d^\#_{j,-1}z^{j} $ their decomposition as power series. 

  Asymptotically, we have $d_{j,-1} =  \theta^{-\frac{1}{2}}{\frac{-1}{2} \choose j  }R^{-j} \sim \theta^{-\frac{1}{2}} (-1)^{j}\Gamma(\frac{1}{2})^{-1} j^{-\frac{1}{2}} $ and  $d^\#_{j,-1} = \theta^{\frac{1}{2}} {\frac{1}{2} \choose j  }R^{-j} \sim \theta^{\frac{1}{2}}(-1)^{j}\Gamma(-\frac{1}{2})^{-1} j^{-\frac{3}{2}} $.
  We now use the corollary (2) of Theorem 7.2.1 \cite[Theorem 7.2.1]{Simon-Book} that states that the asymptotics on  $d_{j,-1} d^\#_{j,-1}$ imply directly that $\alpha_j \sim (C_1)R^{-j}j^{-\frac{1}{2}} $ and $\alpha_j \sim C_{1}^\#R^{-j}j^{\frac{1}{2}} $.
  This concludes the proof.

\end{proof}

\appendix

\section{Construction of infinite volume correlators}\label{sec:infinite_volume_corrlators}

In this section, we construct infinite-volume correlators in the semi-discrete lattice, which we rely on to identify the scaling limit fermionic observables in general simply connected domains.
Those constructions are similar to the isoradial case, and are made out of a proper integration of the so-called discrete exponentials, introduced in the semi-discrete lattice in~\cite{Li-QI-SLE}, following the path of~\cite{Kenyon-dimer-conformal, Baxter-book, CS-universality}.
One of the key features here is to extend discrete exponentials to the entire semi-dirscrete lattice, including the corner graph.

\begin{defn}
  Let $\delta > 0$.
  Take $p\in \medialdomain \cup \midedgedomain $ be a vertex on either the medial graph or the corner graph.
  Fix $\lambda \in \bbC$.
  We define the \emph{semi-discrete exponential}
  $\exp_\delta (\lambda,\cdot,p)$ \emph{normalized at} $p$ using the following reccursion rules.
  \begin{itemize}
    \item $\exp_\delta (\lambda, p, p) := 1$.
    \item $\exp_\delta (\lambda, q, p) := \exp_\delta \big[ \frac{ \lambda(q-p) }{ (1+\frac{\lambda \delta}{2})(1-\frac{\lambda \delta}{2}) } \big] $ for any $q$ such that $\myr (q) = \myr (p) $.
    \item $\exp_\delta (\lambda, c, p) := \exp_\delta (\lambda, q, p) (1+\lambda(c-q))^{-1}$ for each pair of adjacent vertices $q \in \medialdomain$ and $c \in \midedgedomain$.
  \end{itemize}
\end{defn}
Moreover, the reccursion rules ensure that if $u \sim v$ are neighbours in $\medialdomain$, we have,
\begin{equation*}
  \exp_\delta(\lambda, v, p)
  = \exp_\delta(\lambda, u, p)
  \frac{ 1 + \tfrac{\lambda \delta}{2} (v-u) }
  { 1 - \tfrac{\lambda \delta}{2} (v-u) }.
\end{equation*}
One can easily check that the semi-discrete exponential is well-defined and also this definition coincides with the one from~\cite{Li-QI-SLE} if we restrict it to the semi-discrete medial lattice.

\begin{lem}
  Given $p_0\in \medialdomain$ and $ \lambda \in \bbC$, the function $ c\mapsto \exp(\lambda,c,p_0)$ is discrete holomorphic everywhere.
\end{lem}

\begin{proof}
  This is a straightforward computation that we precise here. Assume that $\delta=1$, fix  $\lambda \in \bbC $ and set $F=\exp_{1}(\lambda,\cdot,p_0) $.
  Assuming that $c^{-}\sim c \sim c^{+}$ are naturally ordered neighbors, one has,
  \begin{equation*}
    F(c^{+}) = \frac{ 1 + \tfrac{\lambda }{2} }{ 1 - \tfrac{\lambda }{2} } \,F(c),
    \quad F(c^{-}) = \frac{ 1 - \frac{\lambda }{2} }{ 1 + \frac{\lambda }{2} } \,F(c),
    \quad \partial_y F(c) = \frac{ \icomp \lambda}{ (1+\tfrac{\lambda }{2} )(1 - \tfrac{\lambda }{2})} \,F(c).
  \end{equation*}
  Thus checking discrete holomorphicity~\eqref{defn:s-hol2} amounts to verify that 
  \begin{equation*}
    F(c)\big[\frac{1}{2} (\frac{1+\frac{\lambda}{2}}{1-\frac{\lambda}{2}} - \frac{1-\frac{\lambda}{2}}{1+\frac{\lambda}{2}} ) - \frac{1}{i} \frac{i\lambda}{(1+\frac{\lambda}{2})(1-\frac{\lambda}{2})}   \big]=0
  \end{equation*}
  which holds.

\end{proof}

We are now in position to construct the infinite volume-correlators out of semi-discrete exponentials. 
\begin{defn}[Infinite-volume energy correlator]
  Given $a\in \midedgedomain$, we define the full-plane energy correlator normalized at $a$ by setting for $c\neq a \in \midedgedomain$
  \begin{equation}
    \eta_c\langle \chi_{c}\chi_{a} \rangle_{\bbC_{\delta}} := \frac{1}{2\pi} \displaystyle \int_{-\bbR_{+}\overline{c-a}} \frac{ \overline{\eta_a}\exp_{\delta}(\lambda,c,a)}{1-i\overline{\eta_a}^2\frac{\lambda^2}{4}} \dd \lambda.
  \end{equation}
  The previous definition should be understood at $a^{\pm}$ by
  \begin{equation}\label{eq:double_valuation_energy_correlator}
    \eta_a\langle \chi_{a^{\pm}}\chi_{a} \rangle_{\bbC_{\delta}} := \pm \frac{1}{2\pi} \displaystyle \int_{\bbR_{+}} \frac{\overline{\eta_a}}{1-i\overline{\eta_a}^2\frac{\lambda^2}{4}} \dd \lambda.
  \end{equation}
  We also denote $G_{(a)}:= \eta_a\langle \chi_{c}\chi_{a} \rangle^{\diamond}_{\bbC_{\delta}}  $ its diamond $s$-holomorphic counterpart.

\end{defn}

\begin{defn}[Infinite-volume spin correlators]
  Given $u\in \primaldomain$ and $v\in \dualdomain $, we also define the infinite volume correlators $\eta_c\langle \chi_{c} \sigma_u \rangle_{[\bbC_{\delta},u]}  $ and $\eta_c \langle \chi_{c} \mu_v \rangle_{[\bbC_{\delta},v]}$ by setting
  \begin{equation}
    \eta_c \langle \chi_{c} \sigma_u \rangle_{[\bbC_{\delta},u]} :=  \tfrac{e^{-\frac{i\pi}{4}}}{2\pi} \displaystyle \int_{-\bbR_{+}\overline{c-u}} \frac{1}{\lambda^{1/2}}  \exp_{\delta}(\lambda,c,u) \dd \lambda,
  \end{equation}
  \begin{equation}
    \eta_c \langle \chi_{c} \mu_v \rangle_{[\bbC_{\delta},v]} :=  \tfrac{e^{\frac{i\pi}{4}}}{2\pi} \displaystyle \int_{-\bbR_{+}\overline{c-v}} \frac{1}{\lambda^{1/2}}  \exp_{\delta}(\lambda,c,v) \dd \lambda.
  \end{equation}
  The correlators $\langle \chi_{c} \sigma_u \rangle_{[\bbC_{\delta},u]}  $ and  $\langle \chi_{c} \mu_v \rangle_{[\bbC_{\delta},v]}$ are respectively defined in $[\bbC_{\delta},u] $ and $[\bbC_{\delta},v]$, the double covers of the semi-discrete lattice $\bbC_{\delta} $ respectively branching around $u$ and $v$.
  We also denote $G^{\diamond}_{[u]} := \langle \chi_{c} \sigma_u \rangle^{\diamond}_{[\bbC_{\delta},u]} $ and $G^{\diamond}_{[v]}:=  \langle \chi_{c} \mu_v \rangle^{\diamond}_{[\bbC_{\delta},v]} $ the naturally associated diamond s-holomorphic extensions.
\end{defn}

\begin{prop}
  The energy correlator $\langle \chi_{c}\chi_{a} \rangle^{\diamond}_{\bbC_{\delta}}$ , double valued at ${a}^{\pm}$ as in~\eqref{defn:energy_correlator}, is corner $s$-holomorphic \emph{everywhere} (even at $a^{\pm}$ with the proper graph modifications). Moreover, uniformly on compact subsets of $\bbC\backslash \{ a \} $, one has the asymptotic $\delta^{-1}G_{(a)}(z) = \frac{\overline{\eta_a}}{\pi(z-a)}(1+O(\delta))$ as $\delta \to 0$.
\end{prop}

\begin{proof}
  The corner discrete holomorphicity comes from the fact that one can (using simple deformation of contours) replace the line $-\bbR_{+}\overline{c-a}$ by a contour surrounding all the poles of the integrand at nearby corners.
  One can then use the same contour for nearby corners and the corner discrete holomorphicity comes directly from the one of discrete exponentials.
  In order to compute asymptotics, a simple rescaling of the grid allows to work with $\delta = 1 $ and $|c-a|\to \infty $.
  One can cut the integral in $3$ parts:  $|\lambda|\leq |c-a|^{-\frac{1}{3}} $ , $\lambda \in [|c-a|^{-\frac{1}{3}};|c-a|^{+\frac{2}{3}} ] $ and $|\lambda|\geq |c-a|^{\frac{2}{3}} $.
  Using the fact that $\eta_c^2 = \frac{i}{v_c - u_c} $,  the integrand $\exp(\lambda,c,a)(1-i\overline{\eta_a}^2\frac{\lambda^2}{4})^{-1}$ has the following asymptotics 
  \begin{itemize}
    \item $\overline{\eta_a} \exp(\lambda(c-a)+O(\lambda^2)+ O(|c-a|\lambda^3))$ as $\lambda \to 0$.
    \item $4 \eta_a  \eta_c^2  \lambda^{-2} \exp(4\lambda^{-1}\overline{(c-a)}+O(\lambda^{-2})+ O(|c-a|\lambda^{-3})) $ as $\lambda \to  \infty$.
    \item $\exp(\lambda,c,a)(1-i\overline{\eta_a}^2\frac{\lambda^2}{4})^{-1}=O[\exp(-|c-a|^{\frac{1}{2}})]$ for $\lambda \in [|c-a|^{-\frac{1}{3}};|c-a|^{+\frac{2}{3}} ] $ as in the proof of~\cite[Prop.~3.20]{Li-QI-SLE}.
  \end{itemize}
  Applying the Laplace method one directly gets the expansion 
  $\eta_c \langle \chi_{c}\chi_{a} \rangle_{\bbC_{1}}= \frac{1 }{2\pi} [\frac{\overline{\eta}_a }{c-a} + \frac{ \eta_a \eta_{c}^{2} }{\overline{c-a}}] + O(|c-a|^{-2}) $ as $|c-a|\to \infty $.
  The values at $a^{\pm} $ are obtained by a straightfoward computation.
  The check of complex sign of the projection $\eta_c\langle \chi_{c}\chi_{a} \rangle_{\bbC_{\delta}}$ is left as an exercice to the reader (e.g. see Remark in~\cite[Sec.~5.2]{CIM-universality}).

\end{proof}

\begin{prop}
  The correlator $\eta_c\langle \chi_{c} \sigma_u \rangle_{[\bbC_{\delta},u]}$ is a well-defined corner $s$-holormophic function everywhere on $[\bbC_{\delta},u]$.
  Moreover one has the convergence (uniformly on compacts away from $u$) that $ \delta^{-\frac12}G^{\diamond}_{[u]}(z) = (\tfrac1\pi)^{\frac12} \frac{e^{-\frac{i\pi}{4}}}{\sqrt{z-u}} + O(\frac{\delta^2}{(z-u)^{\frac{5}{2}}}$).
  Similarly the correlators $\eta_c\langle \chi_{c} \mu_v \rangle_{[\bbC_{\delta},v]}$ is a well-defined corner $s$-holormophic function everywhere on the double cover of $\bbC_{\delta} $ branching around $v$.
  One also has the convergence uniformly on compacts away from $v$) to $ \delta^{-\frac12}G^{\diamond}_{[v]}(z) = (\tfrac1\pi)^{\frac12}\frac{e^{\frac{i\pi}{4}}}{\sqrt{z-v}} + O(\frac{\delta^2}{(z-u)^{\frac{5}{2}}}$).
  At corner neighbouring their branchings, both corner s-holomorphic correlators have an absolute value $1$, with appropriate complex sign.
\end{prop}

\begin{proof}
  The spirit of the proof is very similar to the energy correlator.
  Deforming again the half-line into contour (this time avoiding the poles of discrete exponentials and $u$) and using the a asymptotics of discrete exponentials allows to transfer the discrete holomorphicity of discrete exponentials to the discrete holomorphicity of the correlators.
  Rescaling again the lattice, we work with $\delta=1 $ when $|c-u|\to \infty $. This time the integrand has the asymptotics,
  \begin{itemize}
    \item $e^{-\frac{i\pi}{4}} \lambda^{-\frac{1}{2}}\exp(\lambda(c-u)+ O(|c-u|^2\lambda^2))$ as $\lambda \to 0$.
    \item $ 2 e^{\frac{i\pi}{4}}\eta_c^2\lambda^{-\frac{3}{2}} \exp(4\lambda^{-1}\overline{(c-u)}+\calO(|c-u|^{-2}\lambda^{-2}) $ as $\lambda \to  \infty$.
    \item $\lambda^{-\frac{1}{2}}\exp(\lambda,c,u)=O[\exp(-|c-u|^{\frac{1}{2}})]$ for $\lambda \in [|c-u|^{-\frac{1}{3}};|c-u|^{+\frac{2}{3}} ] $ as in the proof of~\cite[Prop.~3.20]{Li-QI-SLE}.
  \end{itemize}
  This proves again that the integral is well defined and the Laplace method provides the announced asymptotics. The spinor property comes from the spinor property of the function $\lambda^{-\frac12}$ and the values near the branchings are computed using the residue theorem. The check of complex sign of the projection $\eta_c \langle \chi_{c} \sigma_u \rangle_{[\bbC_{\delta},u]} $ (i.e. not only the complex sign of its asymptotic which is given by the above computation) is left as an exercice to the reader (e.g. see Remark in~\cite[Sec.~5.2]{CIM-universality}).

\end{proof}

\begin{prop}\label{eq:differential_correlator_position}
  The correlator $G_{[u^\delta]} $ is differentiable with respect to the position of its branching $u$, i.e. uniformly over $z$'s at a fixed distance from $u^\delta$, one has,
  \begin{equation}
    \delta^{-\frac12}\partial_{y}G_{[u^{\delta}]}(z):=  \lim\limits_{\varepsilon \rightarrow 0}\frac{G_{[u^\delta+i\epsilon]}(z)-G_{[u^\delta]}(z)}{\epsilon} =  -(\tfrac1\pi)^{\frac12}  \frac{i}{2(z-u)^{\frac{3}{2}}} + O(\frac{\delta}{(z-u)^{\frac52}}),
  \end{equation}
  where the asymptotic is given when $\delta \rightarrow 0 $.
\end{prop}

\begin{proof}
  Let $c^\delta$ be a fixed corner, away from the $u^\delta$ and $u^\delta+i\epsilon $.
  By translation invariance of the lattice, one clearly has $G_{[u^\delta+i\epsilon]}(c^\delta + \icomp \epsilon) = G_{[u^\delta]}(c^\delta)  $, thus we immediately deduce that $G_{[u^\delta+i\epsilon]}(c^\delta) = G_{[u^\delta]}(c^\delta - i\epsilon)$.
  This implies that,
  \begin{eqnarray}  \nonumber
    \frac{G_{[u^\delta+i\epsilon]}(c^\delta)-G_{[u^\delta]}(c^\delta)}{\epsilon}    & = & \frac{G_{[u^\delta]}(c^\delta-i\epsilon)-G_{[u^\delta]}(c^\delta)}{\epsilon} \\  \nonumber
                                                                                    & \underset{\epsilon \to 0}{\longrightarrow} & \partial_{y}G_{[u^\delta]}(c^\delta) = -i \frac{G_{[u^\delta]}(c^\delta+\delta) - G_{[u^\delta]}(c^\delta-\delta) }{2\delta},
  \end{eqnarray} 
  where the last equality is justified by the fact that $G_{[u^\delta]}(\cdot) $ is discrete holomorphic away from $u$.
  The announced asympotic is now easily recovered by the two terms asymptotics of $G_{[u^\delta]}$ away from its branching.
\end{proof}

\section{Hyperbolic metric on a rectangular domain}
\label{sec:hyperbolic_rectangle}

Let $0 < k < 1$ and define $\calR(k)$ to be the rectangle given by $(-K(k), K(k)) \times (0, K'(k))$, where $K$ denotes the complete elliptic integral of the first kind,
\begin{align*}
  K(k) = \int_0^1 \frac{ \dd t }{ \sqrt{ (1-t^2) (1 - k^2t^2) } },
\end{align*}
and $K'$ denotes its complementary, i.e. $K'(k) = K( \sqrt{1 - k^2} )$.
Note that $k \mapsto K(k) / K'(k)$ is strictly increasing with limits $0$ and $+\infty$ when $k$ tends to $0$ and $1$.
Hence, the whole family $\{ \calR(k), 0 < k < 1 \}$ describes all the rectangles of all the possible aspect ratios.

\begin{prop}
  Let $0 < k < 1$.
  The following Schwarz--Christoffel mapping $f_k$ transforms the upper half-plane $\bbH$ into the rectangle $\calR(k)$,
  \[
    f_k(z) = \int_0^z \frac{ \dd t }{ \sqrt{ (1-t^2) (1 - k^2t^2) } }.
  \]
\end{prop}

\begin{proof}
  We can easily see that $f_k$ is indeed a Schwarz--Christoffel mapping where $-k^{-1}$, $-1$, 1 and $k^{-1}$ are mapped to vertices of the resulting polygon with angles all equal to $\tfrac\pi2$, thus a rectangle.
  Moreover, the images of $-1$ and $1$ via $f_k$ can be computed as follow, $f_k(1) = K(k)$, $f_k(-1) = -K(k)$.
  Finally, we compute $f_k(k^{-1}) - f_k(1) = \icomp K'(k)$ by a change of variables, similarly for $f_k(-k^{-1}) - f_k(-1)$.
\end{proof}

The hyperbolic metric of a simply connected domain $\calD$ is defined by
\begin{equation}
  \label{eq:hyperbolic_metric}
  \ell_\calD(a) = \frac{2 |h'(a)|}{ 1 - |h(a)|^2}, \qquad a \in \calD,
\end{equation}
where $h$ is \emph{any} conformal map sending $\calD$ to the unit disk $\bbD$.
It is important to note that such a conformal map is well-defined only up to a Möbius transformation, which does not change the metric defined via~\eqref{eq:hyperbolic_metric}.

\begin{prop}
  Given $k>1$.
  The hyperbolic metric on $\calR(k)$ is given by
  \[
    \ell_{\calR(k)}(a) = \frac{ \cn(a, k) \dn(a, k) }{ \myi \sn(a, k) }, \qquad a \in \calR(k).
  \]
\end{prop}

\begin{proof}
  Let $g : \bbH \to \bbD$ be a conformal map between the upper half-plane and the unit disk given by $g(z) = \frac{z - \icomp}{z + \icomp}$ for $z \in \bbH$.
  Then the map $h := g \circ f_k^{-1}$ is a conformal map between $\calR(k)$ and $\bbD$.
  Moreover, $f_k^{-1}$ is the inverse of an elliptic integral and is given by the Jacobi elliptic function $a \mapsto sn(a, k)$.
  Then the result follows by applying~\eqref{eq:hyperbolic_metric} and the properties of the Jacobi elliptic functions.
\end{proof}

\bibliographystyle{alpha}
\bibliography{qi-conformal}

\end{document}